\documentclass[11pt,a4paper]{article}
\usepackage[hmargin=2cm,vmargin=1.85cm,paper=a4paper]{geometry}
\usepackage{bm}
\usepackage{amsmath, amssymb, amsthm}
\usepackage{bbm}
\usepackage{amsmath}
\usepackage[english]{babel}
\usepackage{microtype}
\usepackage[export]{adjustbox}
\usepackage{hyperref}
\usepackage{multirow}
\usepackage{enumerate}
\usepackage{graphicx}
\usepackage{float}
\usepackage{color}
\usepackage{placeins}
\usepackage{accents}
\usepackage{amsfonts}
\usepackage{caption}
\usepackage{mathtools}
\mathtoolsset{showonlyrefs}
\usepackage{breqn}
\usepackage{indentfirst}
\usepackage{calrsfs}
\usepackage{dutchcal}

\newtheorem{theorem}{Theorem}[section]
\newtheorem{proposition}[theorem]{Proposition}
\newtheorem{lemma}[theorem]{Lemma}
\newtheorem{corollary}[theorem]{Corollary}
\theoremstyle{definition}
\newtheorem{remark}[theorem]{Remark}
\newtheorem{assumption}[theorem]{Assumption}

\newtheorem{definition}[theorem]{Definition}

\newcommand\restr[2]{{\left.\kern-\nulldelimiterspace #1 \vphantom{\big|} \right|_{#2}}}
\newcommand{\vertiii}[1]{{\left\vert\kern-0.25ex\left\vert\kern-0.25ex\left\vert #1\right\vert\kern-0.25ex\right\vert\kern-0.25ex\right\vert}}
\newcommand\abs[1]{\left|#1\right|}
\newcommand\absb[1]{\Big|#1\Big|}
\let\div\undefined
\DeclareMathOperator{\div}{div}
\DeclareMathOperator{\dom}{dom}
\newcommand{\mres}{\mathbin{\vrule height 1.6ex depth 0pt width
0.13ex\vrule height 0.13ex depth 0pt width 1.3ex}}
\DeclareMathAlphabet{\pazocal}{OMS}{zplm}{m}{n}
\newcommand{\R}{\mathbb{R}}

\newcommand{\Cb}{\pazocal{C}}

\newcommand{\Lb}{\pazocal{L}}

\newcommand{\Fb}{\pazocal{F}}

\newcommand{\Hb}{\pazocal{H}}

\newcommand{\Gb}{\pazocal{G}}

\newcommand{\Vb}{\pazocal{V}}

\newcommand{\Wb}{\pazocal{W}}

\newcommand{\Mp}{\pazocal{M}}
\newcommand{\Mb}{\mathfrak{M}}
\newcommand{\Nb}{\mathfrak{N}}

\DeclareMathOperator{\dd}{d\!}
\newcommand{\me}{\mathrm{e}}

\DeclareMathOperator{\dist}{dist}

\numberwithin{equation}{section}

\definecolor{Mygray}{rgb}{0.85,0.85,0.85}
\definecolor{green2}{rgb}{0,0.7,0.2}	
\definecolor{purple2}{rgb}{0.53,0.2,1}	
\definecolor{brown}{rgb}{0.65,0.25,0.12}

\title{Weak formulations of the nonlinear Poisson-Boltzmann equation in biomolecular electrostatics}
\author{
  Jos\'{e} A. Iglesias\thanks{Johann Radon Institute for Computational and Applied Mathematics (RICAM), Austrian Academy of Sciences, Altenbergerstra{\ss}e 69, 4040 Linz, Austria (\texttt{jose.iglesias@ricam.oeaw.ac.at}).}
  \and
Svetoslav Nakov\thanks{Institute for Theoretical Physics, Johannes Kepler University, Altenbergerstra{\ss}e 69, 4040 Linz, Austria (\texttt{svetoslav.nakov@jku.at}).}}
\date{}                     

\begin{document}
\maketitle
%
%

%

\begin{abstract}
We consider the nonlinear Poisson-Boltzmann equation in the context of electrostatic models for a biological macromolecule, embedded in a bounded domain containing a solution of an arbitrary number of ionic species which is not necessarily charge neutral. The resulting semilinear elliptic equation combines several difficulties: exponential growth and lack of sign preservation in the nonlinearity accounting for ion mobility, measure data arising from point charges inside the molecule, and discontinuous permittivities across the molecule boundary. Exploiting the modelling assumption that the point sources and the nonlinearity are active on disjoint parts of the domain, one can use a linear decomposition of the potential into regular and singular components. A variational argument can be used for the regular part, but the unbounded nonlinearity makes the corresponding functional not differentiable in Sobolev spaces. By proving boundedness of minimizers, these are related to standard $H^1$ weak formulations for the regular component and in the framework of Boccardo and Gallou\"et for the full potential. Finally, a result of uniqueness of this type of weak solutions for more general semilinear problems with measure data validates the strategy, since the different decompositions and test spaces considered must then lead to the same solution.
\end{abstract}



\section{Introduction}
Detailed studies of biomolecular electrostatics provide a tool for the rational design and optimization in diverse fields such as biocatalysis, antibody and nanobody engineering, drug composition and delivery, molecular virology, and nanotechnology \cite{Protein_electrostatics_review_2020}. A commonly accepted and widely used approach is based on solving the nonlinear Poisson-Boltzmann equation (PBE) which provides a mean field description of the electrostatic potential in a system of biological macromolecules immersed in aqueous solution, such as water. In this model, the solvated molecule is represented at an atomic level of detail in a molecule-shaped cavity with a low dielectric permittivity and point partial charges at atomic positions, whereas the water molecules and ions in the solvent are implicitly treated and accounted for by an isotropic dielectric continuum with high dielectric permittivity \cite{Fogolari_Brigo_Molinari_2002}. A further simplification would be to linearize the ionic contributions. However, in that case the resulting description is not accurate enough when the biomolecules are highly charged, as is the case for DNA, RNA and phospholipid membranes such as polylysine \cite{Honig_Nicholls_95}.

For these reasons, there is considerable interest in the nonlinear PBE in the scientific computing and biophysics community. Since it is an equation combining a strong nonlinearity and measure data, the existence and uniqueness of solutions for it, while mostly within the scope of known PDE techniques, is not quite trivial. However in the applications literature these are often assumed as standard folklore or insufficiently justified while possibly using weak formulations that are not rigorously appropriate for equations with measure data. A common approach is to consider the natural energy associated to the problem, which is convex, and applying variational arguments. However, this should also be treated with care, since the nonlinearity of the PBE is strong enough to prevent the energy functional from being differentiable in Sobolev spaces, so the Euler-Lagrange equation does not necessarily hold. 

\subsection*{Contributions}
In this work, we provide a complete treatment for the existence and uniqueness of weak solutions of the nonlinear Poisson-Boltzmann equation when applied to electrostatic models for biological macromolecules, and without assuming that the solvent is necessarily charge neutral. Our methods are centered on the spaces where elliptic PDE with measure right hand side can be rigorously formulated, which are not directly compatible with a variational approach, and on which we prove uniqueness of solutions. For existence, out of the possible approaches to show particular solutions, we focus on the additive solution-splitting techniques which are most explicit and most commonly used in the applications literature. These still involve an energy functional which is not differentiable and discontinuous at every point of its domain (see Section \ref{Sections_remarks_on_J}). However, by proving an a priori estimate for boundedness of minimizers we are able to return to standard $H^1$ weak formulations which can be directly approximated numerically.

The insistence on the molecular setting is not gratuitous: our analysis specifically uses the particularities of this setting, in which the nonlinearity and right hand side are active on disjoint parts of the domain. This separation plays a role in the main existence result in Theorem \ref{Theorem_Existence_for_full_potential_GPBE} by allowing us to decompose the full potential into a regular component that satisfies an elliptic equation with a more regular right hand side in $H^{-1}$, and another term representing the contribution of the point charges through the Newtonian potential. The main assumption to be able to perform this decomposition is that the dielectric permittivity is constant in a neighborhood of the point charges. 

A notable feature of our approach is that we treat weak formulations for the complete nonlinear PBE in the framework of weak solutions for elliptic equations with measure right hand side as defined by Boccardo and Gallou\"et in \cite{Boccardo_Gallouet_1989, Boccardo_Gallouet_Nonlinear_equations_with_RHS_measures_1992} and maintain this unified framework throughout. The other rigorous works that we are aware of treating the biomolecular situation for the PBE use the mentioned type of decompositions as a fixed ansatz, whereas we use it as a way to obtain particular solutions of the general formulation, for which we prove a uniqueness result that will cover any such approach. A prominent such work is \cite{BoLi_paper_2009} using a variational perspective, and where charge neutrality is required (see in particular \cite{BoLiErratum2011}). Since the decomposition is fixed a priori, it uses ad-hoc spaces which can be roughly described as ``$W^{1,1}$ around the point charges, but $H^1$ elsewhere'', whereas we work in the sharp spaces for problems with measure data, or ``in dimension $d$, $W^{1,\frac{d}{d-1}-\varepsilon}$ for all $\varepsilon >0$'', and we assume the interface to be just $C^1$ instead of $C^2$.

The assumption that the interface between the molecular and solvent regions is $C^1$ plays no role in the existence, but is required for our uniqueness result in Theorem \ref{Theorem_uniqueness_for_GPBE}. We will justify below that this assumption is often satisfied for a very common interpretation of the molecular geometries, which cannot be expected to be $C^2$.

Further, for the regular component we treat weak formulations with Sobolev test spaces instead of just minimizers or distributional solutions, and adopt decomposition schemes commonly used in the physical and numerical literature, proving their equivalence through our uniqueness result. The fact that the regular component of a solution obtained by such a decomposition satisfies a weak formulation involving $H^1$ spaces means that this component can be numerically approximated by means of well studied methods, such as standard conforming finite elements. Besides, it also means that the duality approach for error estimation is applicable to obtain both a priori near-best approximation results and to compute guaranteed a posteriori error bounds, as done in \cite{Kraus_Nakov_Repin_PBE1_2018, Kraus_Nakov_Repin_PBE2_2020}. Having a $C^1$ interface also has practical implications, since in this case it is easier to represent exactly with curved elements or isogeometric analysis.

Moreover, a boundedness estimate for the regular component of the potential, as proved in Theorem \ref{Thm_Boundedness_of_solution_to_general_semilinear_elliptic_problem}, has physical implications in its own right. A growing body of literature, starting with \cite{Borukhov_Andelman_Orland_1997}, treats modifications of the Poisson-Boltzmann model with more tame nonlinearities (reflecting finite size ions) on the grounds that the original PBE model may produce unphysically high ion concentrations and potentials. A boundedness estimate for the original implicit-solvent PBE puts some theoretical limits to these concerns, on the level of potentials outside the molecule.

\subsection*{Organization of this paper}
After introducing the general PBE and its linearized version in their physical context, in Section \ref{Section_Setting} we introduce the notion of weak solutions we will work with and provide an existence and uniqueness result for them in the linearized setting. In Section \ref{Section_Splitting} we review two natural linear splittings of solutions, either of which can be used to decouple the contributions of the nonlinearity and of the measure data. Section \ref{Section_Existence_and_uniqueness_of_solution} contains the main results: in Section \ref{Section_Existence_GPBE} we prove existence of weak solutions through a variational argument and boundedness estimate, and Section \ref{Section_uniqueness_for_GPBE} treats uniqueness.

\subsection{Physical formulation}\label{Section_Physics}
We study an interface problem modelling a biological system consisting of a (macro) molecule embedded in an aqueous solution, e.g., saline water. These are embedded in a bounded computational domain $\Omega\subset \mathbb R^d$ with $d\in \{2, 3\}$. The part of it  containing the molecule is denoted by $\Omega_m\Subset \Omega\subset\mathbb R^d$ (see Figure \ref{All_regions_2D_and_3D}) and the one containing the solution with the moving ions is denoted by $\Omega_s$ and defined by $\Omega_s=\Omega\setminus\overline{\Omega_m}$. The interface of $\Omega_m$ and $\Omega_s$ is denoted by $\Gamma=\overline {\Omega_m}\cap\overline{\Omega_s}=\partial\Omega_m$, and the outward (with respect to $\Omega_m$) unit normal vector on $\partial \Omega_m$ by $\bm n_\Gamma$. Usually, the molecular region $\Omega_m$ is prescribed a low dielectric coefficient $\epsilon_m\approx 2$, whereas the solvent region $\Omega_s$ is prescribed high dielectric coefficient $\epsilon_s\approx 80$. We will assume that the function $\epsilon$,  describing the dielectric coefficient in $\Omega$, is constant in the molecule region $\Omega_m$ and Lipschitz continuous in the solvent region $\overline{\Omega_s}$ with a possible jump discontinuity across the interface $\Gamma$, i.e.,
\begin{equation} \label{definition_of_piecewise_constant_epsilon_LPBE}
\epsilon(x)=\left\{
\begin{aligned}
&\epsilon_m, &x&\in \Omega_m,\\
&\epsilon_s(x), &x&\in\Omega_s.
\end{aligned}
\right.
\end{equation} 
We note that in presence of moving ions, more refined models include the so-called ion exclusion layer (IEL). This is a region in which no ions can penetrate and which surrounds the bio-molecules. It is denoted by $\Omega_{IEL}$ and the part of $\Omega_s$ accessible for ions is denoted by $\Omega_{ions}=\Omega_s\setminus\overline{\Omega_{IEL}}$. With this notation, we have $\Omega_s=\big(\overline{\Omega_{IEL}}\setminus\Gamma\big)\cup\Omega_{ions}$ (see Figure \ref{All_regions_2D_and_3D}). We remark that considering this region is optional (indeed many works  do not use it), in which case one may think of only two regions $\Omega_m$ and $\Omega_{ions}=\Omega_s$ and the notation in some of our results below would be simpler. We give precise mathematical definitions for these sets in Section \ref{Section_Molecular_Surface}.
\begin{figure}[!ht]
    \centering
      \includegraphics[width=0.47\linewidth, valign=m]{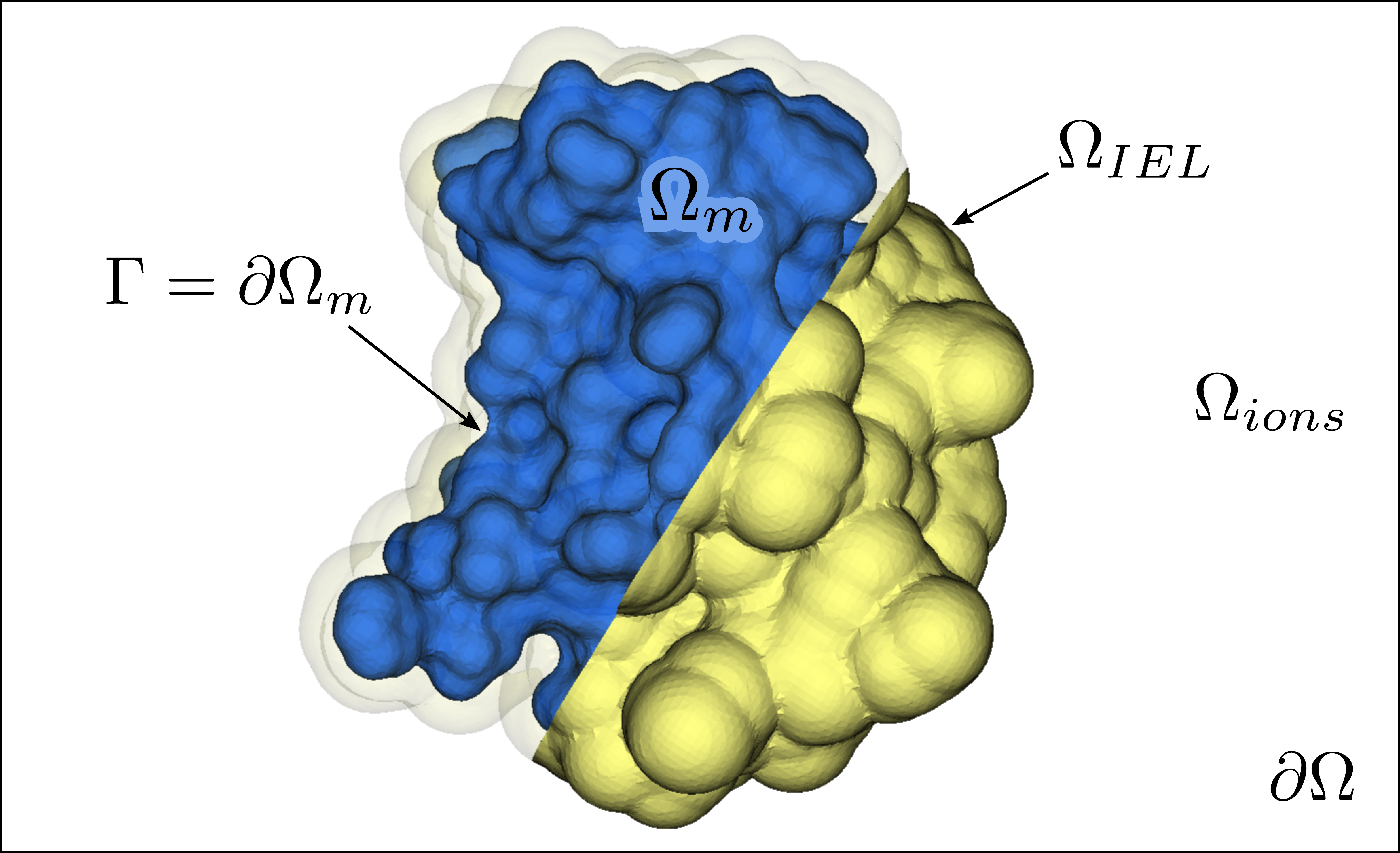}
\caption{Computational domain $\Omega$ with molecular domain $\Omega_m$ in blue, ion exclusion layer $\Omega_{IEL}$ in yellow and see-through, and ionic domain $\Omega_{ions}$. These domains were constructed from an insulin protein with the procedures described in Section \ref{Section_Molecular_Surface}.
  \label{All_regions_2D_and_3D}}
\end{figure}

The electrostatic potential $\hat{\phi}$ is governed by the Poisson equation  which is derived from Gauss's law of electrostatics. In CGS (centimeter-gram-second) units, the Poisson equation reads
\begin{subequations}\label{general_form_PBE}
\begin{equation} \label{general_PBE_1_classical_form}
-\nabla\cdot\big(\epsilon\nabla \hat{\phi}\big)=4\pi \rho  \quad\text{ in } \Omega_m \cup \Omega_s.
\end{equation} 

Here, $\rho:=\chi_{\Omega_m}\rho_m+\left(\chi_{\Omega_{IEL}}\rho_{IEL}+\chi_{\Omega_{ions}}\rho_{ions}\right)$ denotes the charge density in $\Omega$, where $\rho_m$, $\rho_{IEL}$, and $\rho_{ions}$ are the charge densities\footnote{Since $\Omega_s=\left(\overline{\Omega_{IEL}}\setminus\Gamma\right)\cup \Omega_{ions}$, then $\chi_{\Omega_{IEL}}\rho_{IEL}+\chi_{\Omega_{ions}}\rho_{ions}$ gives the charge density in $\Omega_s$.} in $\Omega_m$, $\Omega_{IEL}$, and $\Omega_{ions}$, respectively, and $\chi_U$ denotes the characteristic function of the set $U$, defined by $\chi_U(x)=1$ if $x\in U$ and $\chi_U(x)=0$ elsewhere. In the molecular region $\Omega_m$, there are only fixed partial charges so the charge density is 
\[
\rho_m=\sum_{i=1}^{N_m}{z_ie_0\delta_{x_i}},
\]
where $N_m$ is the number of fixed partial charges, $z_i$ is the valency of the $i$-th partial charge, $x_i \in \Omega_m$ its position, and $e_0$ is the elementary charge, and $\delta_{x_i}$ denotes the delta function centered at $x_i$.  In the region $\Omega_{IEL}$ there are no fixed partial charges, nor moving ions and therefore the charge density there is $\rho_{IEL}=0$. 
In the region $\Omega_{ions}$, there are moving ions whose charge density is assumed to follow a Boltzmann distribution and is given by
\[
\rho_{ions}=\sum_{j=1}^{N_{ions}}{M_j\xi_j e_0\me^{-\frac{\xi_je_0\widehat{\phi}}{k_BT}}},
\]
where $N_{ions}$ is the number of different ion species in the solvent, $\xi_j$ is the valency of the $j$-th ion species, $M_j$ is its average concentration in $\Omega_{ions}$ measured in ${\# \textrm{ions}}/{\textrm{cm}^3}$,  $k_B$ is the Boltzmann constant, and $T$ is the absolute temperature. For more information on the physical constants used in the text see Table~\ref{Table_physical_constants_in_CGS}.

\begin{table}[ht]
\centering
\begin{tabular}{ |c|c|c| }
\hline
abbreviation &  name & value in CGS derived units\\
\hline \hline
$N_A$		& Avogadro's number		&$6.022140857 \times 10^{23}$ \\ 
$e_0$		& elementary charge			&$4.8032424\times 10^{-10}$ esu	\\
$k_B$		& Boltzmann's constant		&$1.38064852\times 10^{-16}$erg $\text{K}^{-1}$\\
\hline			
\end{tabular}
\captionof{table}{Physical constants used in this section. Here K denotes Kelvin as a unit of temperature, esu is the statcoulomb unit of electric charge, and erg the unit of energy which equals $10^{-7}$ joules.}\label{Table_physical_constants_in_CGS} 
\end{table}

The physical problem requires that the potential $\hat{\phi}$ and the normal component of the displacement field $\epsilon\nabla \hat{\phi}$ are continuous across the interface $\Gamma$. Thus, the equation \eqref{general_PBE_1_classical_form} is supplemented with the following continuity conditions
\begin{align}
\big[\hat{\phi}\big]_\Gamma&=0,\label{general_PBE_2_classical_form}\\
\big[\epsilon\nabla \hat{\phi}\cdot \bm n_\Gamma\big]_\Gamma&=0,\label{general_PBE_3_classical_form}
\end{align}
where $\left[\cdot\right]_\Gamma$ denotes the jump across the interface $\Gamma$ of the enclosed quantity.  Finally, the system \eqref{general_PBE_1_classical_form}, \eqref{general_PBE_2_classical_form}, \eqref{general_PBE_3_classical_form} is complemented with the boundary condition
\begin{equation} \label{general_PBE_4_classical_form}
	\hat{\phi}=\hat{g}_\Omega \quad\text{ on } \partial\Omega.
\end{equation} 
We notice that in fact the physical problem prescribes a vanishing potential at infinity, that is $\hat\phi(x)\to 0$ as $\abs{x}\to\infty$. However in most practical situations one uses a bounded computational domain and imposes the Dirichlet boundary condition \eqref{general_PBE_4_classical_form} instead. In this case, the function $\hat{g}_\Omega$ can be prescribed using the exact solution of a simpler problem in the full $\R^d$ for the linearized equation with constant solvent permittivity, which can be expressed explicitly through Green functions (see Eq. (5) of \cite{Bond_2009} or \cite{Zhou_Payne_Vasquez_Kuhn_Levitt_1996}, for example).
\end{subequations}

By introducing the new functions $\phi=(e_0\hat{\phi})/(k_B T)$ and $g_\Omega=(e_0 \hat{g}_\Omega)/(k_B T)$ equations \eqref{general_PBE_1_classical_form}--\eqref{general_PBE_4_classical_form} can be written in distributional sense in terms of the dimensionless potential $\phi$:
\begin{equation} \label{GPBE_dimensionless}
\begin{aligned}
-\nabla\cdot\left(\epsilon\nabla \phi\right)+b(x,\phi)&=\Fb \quad \text{ in } \Omega,\\
\phi&=g_\Omega\quad\text{ on } \partial\Omega,
\end{aligned}
\end{equation} 
where $b(x,t):\Omega\times \mathbb R\to\mathbb R$ is defined by
\begin{equation} \label{definition_b_GPBE}
b(x,t):=-\frac{4\pi e_0^2}{k_BT} \sum_{j=1}^{N_{ions}}{\overline M_j(x)\xi_j \me^{-\xi_j t}} \text{ for all } x\in\Omega,\,t\in\mathbb R
\end{equation} 
with $\overline M_j(x):=\chi_{\Omega_{ions}}M_j$ and 
\begin{equation}\label{definition_F} 
\Fb:=\frac{4\pi e_0^2}{k_BT}\sum_{i=1}^{N_m}{z_i\delta_{x_i}}.
\end{equation} 
Observe that if the condition
\begin{equation} \label{charge_neutrality_condition}
\sum_{j=1}^{N_{ions}}{M_j\xi_j}=0
\end{equation} 
holds, then the solvent is electroneutral and we refer to this as the {\it charge neutrality condition}. Obviously, in this case we have that $b(x,0)=0$ for all $x\in\Omega$. This is a quite standard assumption, which we do \emph{not} enforce for our analytical results. In nearly all biophysics models involving the PBE the solvent is charge neutral, but there are also some exceptions. One such is the so-called cell model (see Eq. (17) in \cite{Sharp_Honig_1990}) in which the macromolecule possesses a net charge, and exactly enough counterions of only one species are present to keep the volume to which all the charges are confined globally electrically neutral.
 
We notice as well that $b(x_i,t)=0$ at the positions $x_i$ of the point charges in the molecular region. More precisely, $b(x,t)=0$ for a.e. $x\in \Omega_m\cup\Omega_{IEL}$. This observation will be crucial to our analysis below.

Under the assumption that there are only two ion species in the solution with the same concentration $M_1=M_2=M$, which are univalent but with opposite charge, i.e $\xi_j=(-1)^j,\,j=1,2$, we obtain the equation
\begin{equation} \label{PBE_dimensionless}
\begin{aligned}
-\nabla\cdot\left(\epsilon \nabla \phi\right)+\overline{k}^2 \sinh\left(\phi \right)&=\Fb \quad\text{ in } \Omega,\\
	\phi&=g_\Omega \quad\text{ on } \partial\Omega .
	\end{aligned}
\end{equation} 
The coefficient $\overline{k}$ is defined by
\begin{equation} \label{definition_of_piecewise_constant_k}
\overline{k}^2(x)=\left\{
\begin{aligned}
&0, &x&\in \Omega_m\cup\Omega_{IEL},\\
&\overline{k}_{ions}^2=\frac{8\pi N_A e_0^2I_s}{1000k_BT}, &x&\in \Omega_{ions},
\end{aligned}
\right.
\end{equation} 
where $N_A$ is Avogadro's number and the ionic strength $I_s$, measured in moles per liter (molar), is given by
\[
I_s=\frac{1}{2}\sum_{j=1}^{2}{c_i \xi_j^2}=\frac{1000M}{N_A}
\]
with $c_1=c_2=\frac{1000M}{N_A}$, the average molar concentration of each ion (see \cite{Niedermeier_Schulten_1992, Bashford_2004}). 

Equation \eqref{PBE_dimensionless} is often referred to as the Poisson-Boltzmann equation \cite{Sharp_Honig_1990, Oberoi_Allewell_1993,Holst2012}. On the other hand, we will refer to \eqref{GPBE_dimensionless} as the General Poisson-Boltzmann equation (GPBE).
The GPBE \eqref{GPBE_dimensionless} can be linearized by expanding $b(x,\cdot)$ in Maclaurin series. We obtain the linearized GPBE (LGPBE) equation for the dimensionless potential $\phi$:
\begin{equation} \label{LGPBE_dimensionless}
\begin{aligned}
-\nabla\cdot\left(\epsilon\nabla \phi\right)+\overline{m}^2\phi &=\Fb + \ell \quad\text{ in } \Omega,\\
	\phi&=g_\Omega \quad\text{ on } \partial\Omega ,
	\end{aligned}
\end{equation} 
	where 
	\begin{equation} \label{definition_overline_m_and_l}
	\overline m^2(x):=\sum_{j=1}^{N_{ions}}{\overline M_j(x)\xi_j^2} \quad\text{ and }\quad \ell(x):=\sum_{j=1}^{N_{ions}}{\overline M_j(x) \xi_j}.
	\end{equation} 
	
\subsection{The molecular surface}\label{Section_Molecular_Surface}
As we are working with a mesoscopic model the subdomains appearing have to be precisely defined, and there is not necessarily only one way to do so; in fact, even the representation with sharp cutoffs is a modelling assumption. A common starting point is to consider the molecule as occupying the van der Waals set $\Vb$ (with $\partial \Vb$ the corresponding surface) which is given as a union of spheres centered at the positions where the atoms can be located and with radius the van der Waals radius of each element.

Since we are interested in the interaction with an ionic solution, our molecule boundary should be even larger than the van der Waals set, and account for the regions that cannot be accessed by the solvent. This is known as the solvent excluded surface (SES) which is formed by rolling a solvent probe modelled as a sphere on the van der Waals surface (see \cite{Lee_Richards_1971, Richards_SES_1977, Greer_Bush_1978}). A popular precise definition  is that of Connolly \cite{Connolly_Analytic_Description_1983} in which the SES is taken to be the surface traced by the boundary of the solvent sphere, which we now describe mathematically.

\begin{figure}[!ht]
    \centering    
      \includegraphics[width=0.38\linewidth, valign=m]{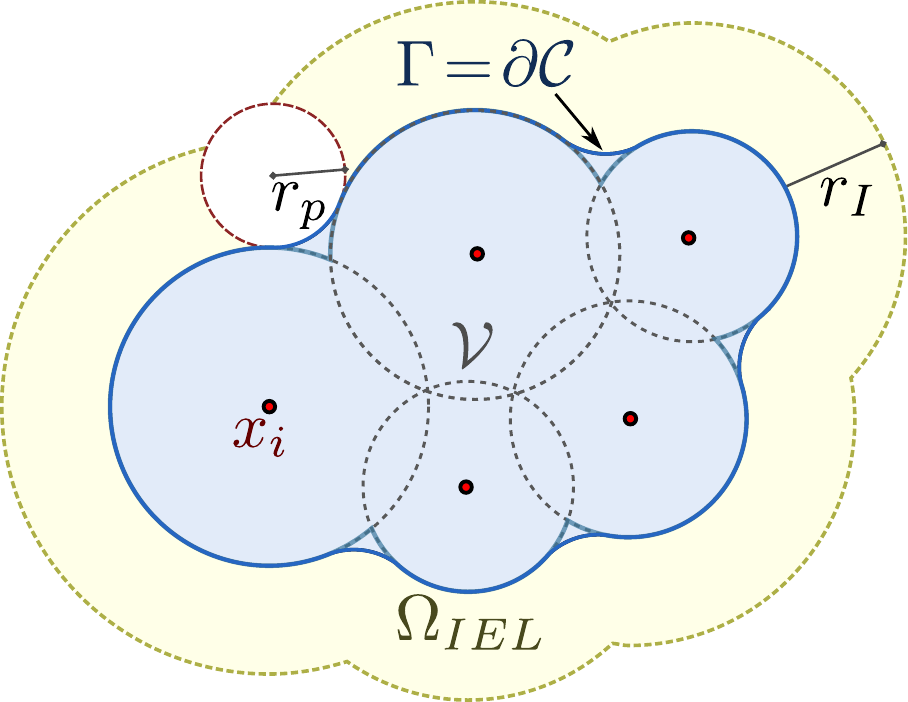} 
\caption{Construction of the Connolly surface $\partial \Cb$ from the van der Waals set $\Vb$ (a union of balls centered at the positions of the charges $x_i$) by rolling a spherical probe of radius $r_p$, and construction of $\Omega_{IEL}$ enlarging $\Vb$ by $r_I$.
  \label{Molecular_Surfaces}}
\end{figure}

Given an open set $\Sigma \subset \R^d$ and radius $r>0$ we can define the set generated by ``rolling a ball'' of radius $r$ inside it or outside it as
\[[\Sigma]_r := \bigcup_{B(x,r) \subset \Sigma} B(x,r)\ \text{ and }\ [\Sigma]^r := \mathrm{int}\left( \R^d \setminus \big[\R^d \setminus \Sigma\big]_r\right),\]
where $\mathrm{int}$ denotes the interior so that both remain open. Moreover, whenever $[\Sigma]_r = \Sigma$ or $[\Sigma]^r = \Sigma$ we say that a ball of radius $r$ can roll freely inside it or outside it, respectively.

With this, given a van der Waals set $\Vb$ and the van der Waals radius $r_p$ of a probe solvent molecule we can define the Connolly set $\Cb := [\Vb]^{r_p}$ with $\partial \Cb$ the Connolly surface (see Figure \ref{Molecular_Surfaces}). Now even if $\Vb$ is a union of balls, so clearly we can roll a ball inside it with the smallest radius and $[\Vb]_{r_{\Vb}}=\Vb$ for some $r_{\Vb}$, this is not necessarily the case for $\partial \Cb$ (see Section 3 in \cite{Walther_1999} for a counterexample). However if additionally $[\Cb]_{r_0} = [\Cb]$ for some $r_0$, then we have that \[[\Cb]_{\min(r_0, r_p)} = [\Cb]^{\min(r_0, r_p)}=\Cb\]
and in this situation we can apply Theorem 1 of \cite{Walther_1999} to conclude that $\partial \Cb \in C^{1,1}$. Intuitively, the condition $[\Cb]_{r_0} = [\Cb]$ will be satisfied when neither $\Vb$ nor $\R^d \setminus \Vb$ contain passages which are thinner than $2r_0$.

Under these assumptions, if one chooses $\Omega_m = \Cb$ then $\partial \Omega_m$ is in particular $C^1$, a condition which will be needed for all our uniqueness results below (cf. the proofs of Theorem~\ref{Theorem_uniqueness_of_linear_elliptic_problems_with_diffusion_and_measure_rhs} and Theorem~\ref{Theorem_uniqueness_for_GPBE}), but for none of the existence ones. We also remark that the Connolly surface we have just described is never $C^2$: since it is a union of pieces of spheres, the curvature of $\partial C$ jumps along their intersections.

Above we have introduced the ion exclusion layer around the molecular region which no ions can penetrate. A commonly used definition is to enlarge every ball of the van der Waals set $\Vb$ by the van der Waals radius $r_I$ of a probe ion molecule. If $r_I>r_p$, we can write (see Figure~\ref{All_regions_2D_and_3D})
\begin{equation}
\Omega_{IEL} = \left\{ x \in \Omega \, :\, \dist(x\, , \, \Vb) < r_I\right\}\setminus \overline{\Omega_m}.\end{equation}
The regularity of the outer boundary $\partial\Omega_{IEL}\setminus\Gamma$ of $\Omega_{IEL}$ plays no role in our analysis below.

\section{Functional analytic setting}\label{Section_Setting}
Our first goal is to give a meaningful notion of a weak solution to the problems
GPBE \eqref{GPBE_dimensionless} and the LGPBE \eqref{LGPBE_dimensionless}, which ultimately will ensure uniqueness. The semilinear elliptic equation \eqref{GPBE_dimensionless} combines several features that significantly complicate its treatment: a discontinuous dielectric coefficient $\epsilon$, a measure right hand side $\Fb$ defined in \eqref{definition_F}, and an unbounded nonlinearity $b(x,\cdot)$ defined in \eqref{definition_b_GPBE}. Before we get into the solution theory of \eqref{LGPBE_dimensionless} and \eqref{GPBE_dimensionless}, we introduce some notation concerning the function spaces that will be used.

\begin{assumption}[Domain and permittivity]\label{Assumption_Domain_permittivity}
The domain $\Omega$ is assumed to be open bounded and with Lipschitz boundary $\partial \Omega$, whose outward unit normal vector exists almost everywhere (with respect to the area measure) and is denoted by $\bm n_{\partial\Omega}$. The molecule subdomain $\Omega_m$ is strictly inside $\Omega$, i.e., $\overline \Omega_m \subset \Omega$, and the interface $\Gamma=\partial \Omega_m$ is assumed to be $C^1$. We assume that the boundary data $g_\Omega$ is globally Lipschitz on the boundary, that is $g_\Omega\in C^{0,1}(\partial\Omega)$. Moreover, the dielectric permittivity $\epsilon$ is assumed\footnote{Treating anisotropic permittivities (when $\epsilon$ is a symmetric matrix) is straightforward with the same methods used, as long as all these assumptions are satisfied.} constant in $\Omega_m$, equal to $\epsilon_m$, variable in $\Omega_s$ such that $\epsilon_s\in C^{0,1}(\overline{\Omega_s})$, and is allowed to have a jump discontinuity across the interface $\Gamma$. 
\end{assumption}

For $1\leq p < \infty$, the standard Sobolev space $W^{1,p}(\Omega)$ consists of functions which together with their first order weak partial derivatives lie in the Lebesgue space $L^p(\Omega)$. The subspace $W_0^{1,p}(\Omega)$ for $1\leq p<\infty$ denotes the closure of all smooth functions with compact support, $C_c^\infty(\Omega)$, in $\Omega$ with respect to the strong topology of $W^{1,p}(\Omega)$. Given $g \in W^{1-1/p, p}(\partial \Omega)$ and recalling (see Theorem 18.34 in \cite{Leoni_2017}) that the trace operator of $W^{1,p}(\Omega)$ denoted by $\gamma_p$ is surjective onto this space we define the set
\[W_g^{1,p}(\Omega):=\{v\in W^{1,p}(\Omega) \,:\,  \gamma_p(v)=g\}.\]
By $\langle \cdot , \cdot \rangle$ we denote the duality pairing in $W^{-1,p'}(\Omega)\times W_0^{1,p}(\Omega)$ for some $1\leq p<\infty$, where $W^{-1,p'}(\Omega)$ denotes the dual space of $W_0^{1,p}(\Omega)$. In particular, we will also use this notation for the action of measures considered as elements $W^{-1,p'}(\Omega)$ for $p>d$, taking into account the Sobolev embedding into continuous functions.

Whenever $p=2$ we also use the standard notation $H^1(\Omega)=W^{1,2}(\Omega)$ and analogously for $H^1_0(\Omega)$, the trace space $H^{1/2}(\partial \Omega)$ and its dual $H^{-1/2}(\partial \Omega)$. For a vector valued function $\bm\psi\in \left[C^\infty\big(\overline\Omega\big)\right]^d$, the evaluation of the normal trace operator $\gamma_n$ is defined almost everywhere on $\partial\Omega$ as the restriction of $\bm\psi\cdot \bm n_{\partial\Omega} \text{ to }\partial\Omega$. It is well known that the mapping $\gamma_n$,  can be extended by continuity to a continuous linear operator from $H(\div;\Omega)$ onto $H^{-1/2}(\partial\Omega)$, which we still denote by $\gamma_n$ (see, e.g., Theorem 2 in Section 1.3 of \cite{Dautray_Lions_Volume_6}).

We would like to handle elliptic equations with measures as right hand side. In view of the Riesz representation theorem (see Theorem B.111 in \cite{Leoni_2017}) that tells us that the Banach space of bounded signed Radon measures is given as 
\begin{equation} \label{definition_C_0_overline_Omega}
\Mp(\Omega):=\left(C_0(\overline\Omega)\right)^*\ \text{ with }\ C_0(\overline \Omega):=\{v\in C(\overline \Omega)\,:\, v=0 \text{ on } \partial\Omega\},
\end{equation} 
we should test weak formulations for such equations only with continuous functions. The Morrey-Sobolev inequality then suggests that it is natural to introduce 
\begin{equation}\label{Definition_M_and_N}
\Mb:=\bigcap_{p<\frac{d}{d-1}}{W^{1,p}(\Omega)} \quad\text{ and }\quad \Nb:=\bigcup_{q>d}{W_0^{1,q}(\Omega)},
\end{equation} 
where we note that not only $\Mb$ but also $\Nb$ is a linear space, since $\Omega$ is bounded and the spaces $W_0^{1,q}(\Omega)$ are nested. The following lemma is easy to check (see Exercises 9.12 and 11.50 in \cite{Leoni_2017}).
\begin{lemma}\label{Lemma_extension_pf_Lipschitz}
Let $g\in C^{0,1}(\partial\Omega)$. Then there exists an extension $u_g\in C^{0,1}(\Omega)$ such that $\restr{u_g}{\partial\Omega}=g$. Moreover, $u_g\in W^{1,\infty}(\Omega)$.
\end{lemma}
Therefore, for a given Lipschitz function $g\in C^{0,1}(\partial\Omega)$, we have that $g$ is in all trace spaces $W^{1-1/p, p}(\partial \Omega)$ and we denote 
\begin{equation}\label{Definition_Mg}
\Mb_{g}:=\bigcap_{p<\frac{d}{d-1}}{W_{g}^{1,p}(\Omega)}.
\end{equation} 

To understand the interplay between the differential operator $\phi\mapsto -\div(\epsilon(x)\nabla \phi)$, the measure $\Fb$ and the nonlinearity $b(x,\cdot)$ in \eqref{GPBE_dimensionless}, we start by discussing the linearized problem \eqref{LGPBE_dimensionless}.

\subsection{Linear elliptic equations with measure right hand side}
First we notice that \eqref{LGPBE_dimensionless} falls in the more general class of linear elliptic problems of the form
\begin{equation} \label{general_form_linear_problem_with_measure_RHS}
\begin{aligned}
-\div\left(\bm A \nabla \phi\right)+c\phi&=\mu \quad\text{ in }\Omega,\\
\phi&=g \quad\text{ on } \partial\Omega,
\end{aligned}
\end{equation} 
where $\bm A$ is a symmetric matrix with entries in $L^\infty(\Omega)$, which satisfies the usual uniform ellipticity condition
\begin{equation} \label{uniform_ellipticity_of_A}
\underline \alpha\abs{\xi}^2\leq \bm A(x)\xi\cdot \xi\text{ for some }\underline{\alpha}>0,\text{ all } \xi=(\xi_1,\ldots,\xi_d)\in\mathbb R^d\text{ and }a.e.\, x\in \Omega,
\end{equation} 
$c\in L^\infty(\Omega)$, and $\mu \in \Mp(\Omega)$. There are different notions of solution to \eqref{general_form_linear_problem_with_measure_RHS}. Here we mention two approaches in the case $g=0$. The first one is due to Stampacchia \cite{Stampacchia_1965}, where he introduced a notion of a solution to \eqref{general_form_linear_problem_with_measure_RHS} defined by duality using the adjoint of the complete second order operator. The second one is due to Boccardo and Gallou\"et and first appearing in \cite{Boccardo_Gallouet_1989}, where they defined weak solutions of \eqref{general_form_linear_problem_with_measure_RHS} to be those satisfying
\begin{equation} \label{Distributional_solution_defined_by_approximation_Boccardo_Gallouet}
\begin{aligned}
 &\phi\in \Mb_0 \,\text{ and }\,\int_{\Omega}{\bm A\nabla \phi\cdot\nabla v \dd x} + \int_{\Omega}{c\phi v \dd x}=\int_{\Omega}{v \dd\mu}\quad \text{ for all } v\in \Nb,
\end{aligned}
\end{equation} 
and whose existence is proved passing to the limit in the solutions for more regular data $\mu$.

The solution $\phi$ defined by duality in the framework of Stampacchia is unique and can be shown to satisfy the weak formulation \eqref{Distributional_solution_defined_by_approximation_Boccardo_Gallouet} above as well. However applying this approach is not always possible, and for the nonlinear GPBE \eqref{GPBE_dimensionless} it is not clear how to do so since we have discontinuous space dependent coefficients in the principal part. On the other hand the approximation approach of Boccardo and Gallou\"et can be extended for relatively general nonlinear elliptic problems (see Theorems 1 and 3 in \cite{Boccardo_Gallouet_1989}, Theorem 1 in \cite{Boccardo_Gallouet_Nonlinear_equations_with_RHS_measures_1992}). Since the weak formulation for the latter notion of solutions only involves integrating by parts once, it is a problem which can be immediately posed for \eqref{GPBE_dimensionless} as in \eqref{weak_formulation_General_PBE_W1p_spaces} below, so in this work we focus on this type of weak solutions. Some works further discussing the relations between these notions of solution are \cite{Prignet_1995, Meyer_Panizzi_Schiela_2011}.

A difficulty in adopting this notion is is that in dimension $d\ge 3$ and for a general diffusion coefficient matrix $\bm A$ which is in $L^\infty(\Omega)$ and satisfies the uniform ellipticity condition \eqref{uniform_ellipticity_of_A}, the weak formulation \eqref{Distributional_solution_defined_by_approximation_Boccardo_Gallouet} could nevertheless exhibit nonuniqueness, as shown by a counterexample due to Serrin \cite{Serrin_Pathological_Solutions_Elliptic_PDEs,Prignet_1995}. However, under some assumptions on the regularity of the coefficient matrix $\bm A$, one can still show the uniqueness of a weak solution to \eqref{general_form_linear_problem_with_measure_RHS} in the sense of \eqref{Distributional_solution_defined_by_approximation_Boccardo_Gallouet} by employing an adjoint problem with a more regular right-hand side. One such applicable result is a classical one due to Meyers (see \cite{Meyers_Lp_estimates_1964}, Theorem 4.1, Theorem 4.2 in \cite{Bensoussan_Lions_Papanicolaou_book_1978}, or \cite{Gallouet_Monier_On_the_regularity_of_solutions_Meyers_Theorem_1999} for Lipschitz domains), covering general $\bm A$ but only $d=2$, as used for uniqueness in Theorem 2 in  \cite{Gallouet_Herbin_Existence_of_a_solution_to_coupled_elliptic_systems_1994}. For higher dimensions, restrictions on $\bm A$ are necessary, and for example the case $\bm A = \bm I$ and $d=3$ is treated in Theorem 2.1 of \cite{Droniou_Gallouet_Herbin_2003} by using a regularity result of Grisvard \cite{Grisvard_Elliptic_Problems_in_Nonsmooth_Domains}. 

Since these do not apply for our case, in Theorem~\ref{Theorem_uniqueness_of_linear_elliptic_problems_with_diffusion_and_measure_rhs} below we instead apply a more recent optimal regularity result for elliptic interface problems proved in \cite{Optimal_regularity_for_elliptic_transmission_problems}, which still requires the interface $\partial\Omega_m$ to be quite smooth.

\begin{theorem}\label{Theorem_uniqueness_of_linear_elliptic_problems_with_diffusion_and_measure_rhs}
Assume that $\Omega\subset \mathbb R^d$ with $d\in \{2, 3\}$ is a bounded Lipschitz domain and let $\Omega_m\subset\Omega$ be another domain with a $C^1$ boundary and $\partial \Omega_m \cap \partial \Omega = \emptyset$. Let $\bm A$ be a function on $\Omega$ with values in the set of real, symmetric $d\times d$ matrices which is uniformly continuous on both $\Omega_m$ and $\Omega\setminus \overline{\Omega_m}$. Additionally, $\bm A$ is supposed to satisfy the ellipticity condition \eqref{uniform_ellipticity_of_A}. Further, let $g\in H^{1/2}(\partial\Omega)$, $\mu\in\Mp(\Omega)$ be a bounded Radon measure, and $c\in L^\infty(\Omega)$ is such that $c(x)\ge 0$ a.e. in $\Omega$. Then the problem
\begin{equation}\label{General_linear_problem_with_measure_RHS}
\begin{aligned}
&\text{Find }\varphi\in \Mb_g \text{ such that }\\
&\int_{\Omega}{{\bm A}\nabla \varphi\cdot\nabla v \dd x}+\int_{\Omega}{c\,\varphi v \dd x}=\int_{\Omega}{v \dd\mu}\quad \text{ for all } v\in \Nb
\end{aligned}
\end{equation} 
has a unique solution.
\end{theorem}
\begin{proof}
\textbf{Existence:} The existence of a solution $\varphi$ of problem \eqref{General_linear_problem_with_measure_RHS} in the case of homogeneous Dirichlet boundary condition, i.e., $g=0$ on $\partial\Omega$ follows from Theorem~3 in \cite{Boccardo_Gallouet_1989} where a solution is obtained as the limit of the solution to problems with regular right-hand sides. In the case where $g$ is not identically zero on $\partial\Omega$ one can find a solution of \eqref{General_linear_problem_with_measure_RHS} using the linearity. We split $\varphi$ into two components $\varphi_D$ and $\varphi_0$ such that $\varphi_D$ satisfies a linear problem with nonhomogeneous Dirichlet boundary condition and zero right-hand side, i.e.,
\begin{equation} \label{linear_problem_for_uD_W1p_spaces}
\begin{aligned}
&\varphi_D\in \Mb_g \text{ such that }\\
&\int_{\Omega}{{\bm A}\nabla \varphi_D\cdot\nabla v \dd x}+\int_{\Omega}{c\,\varphi_Dv \dd x}=0\quad \text{ for all } v\in \Nb
\end{aligned}
\end{equation} 
and $\varphi_0$ satisfies a linear problem with homogeneous Dirichlet boundary condition and measure right-hand side.
\begin{equation} \label{linear_problem_for_u0_W1p_spaces}
\begin{aligned}
&\varphi_0\in \Mb_0 \text{ such that }\\
&\int_{\Omega}{{\bm A}\nabla \varphi_0\cdot\nabla v \dd x}+\int_{\Omega}{c\,\varphi_0v \dd x}=\int_{\Omega}{v d\mu}\quad \text{ for all } v\in \Nb.
\end{aligned}
\end{equation} 
Clearly, if we replace the solution space in \eqref{linear_problem_for_uD_W1p_spaces} with $H_g^1(\Omega)$ and the test space with $H_0^1(\Omega)$, there exists a unique solution $\varphi_D\in H_g^1(\Omega)$ (by the Lax-Milgram theorem). Since $H_g^1(\Omega)\subset \bigcap_{p<\frac{d}{d-1}}{W_g^{1,p}(\Omega)}$ and $\bigcup_{q>d}{W_0^{1,q}(\Omega)}\subset H_0^1(\Omega)$ it is clear that this $\varphi_D$ also solves \eqref{linear_problem_for_uD_W1p_spaces}. By Theorem~3 in \cite{Boccardo_Gallouet_1989}, problem \eqref{linear_problem_for_u0_W1p_spaces} also possesses a solution $\varphi_0$ obtained by approximation as the weak (even strong) limit in $W_0^{1,q}(\Omega)$ for every fixed $q<\frac{d}{d-1}$ of a sequence of solutions $\{\varphi_{0,n}\}_{n \in \mathbb{N}}$ of $H^1$ weak formulations with regularized right-hand sides. Now it is clear that $\varphi=\varphi_D+\varphi_0$ solves \eqref{General_linear_problem_with_measure_RHS}, and in fact the functions $\varphi_D+\varphi_{0,n}$ provide the same kind of approximation, since they satisfy $H^1$ weak formulations of linear problems with the same regularized right-hand sides and nonhomogeneous boundary condition given by $g$ on $\partial\Omega$.

\textbf{Uniqueness:} It is enough to show that if $\varphi$ satisfies the homogeneous problem \eqref{General_linear_problem_with_measure_RHS} with $\mu=0$ then $\varphi=0$. For a fixed $\theta\in L^\infty(\Omega)$, we consider the auxiliary problem 
\begin{equation} \label{Auxiliary_problem_for_uniqueness_of_General_linear_elliptic_problem_with_measure_RHS}
\begin{aligned}
&\text{Find }z\in H_0^1(\Omega) \text{ such that }\\
&\int_{\Omega}{{\bm A}\nabla z\cdot\nabla v \dd x}+\int_{\Omega}{c\,zv \dd x}=\int_{\Omega}{\theta v \dd x},\,\mathrm{for}\ \mathrm{all}\  v\in H_0^1(\Omega).
\end{aligned}
\end{equation} 
By the Lax-Milgram Theorem, this problem has a unique solution $z\in H_0^1(\Omega)$. In view of the Sobolev embedding theorem, for $d=3$, $H^1(\Omega)\hookrightarrow L^6(\Omega)$ and for $d=2$, $H^1(\Omega)\hookrightarrow L^r(\Omega)$ for all $1\leq r<\infty$. Therefore, in both cases $z\in L^6(\Omega)$ and consequently $(-cz+\theta)\in L^6(\Omega)$. Since $v \mapsto \int_{\Omega}{\left(-cz+\theta\right) v \dd x}$ defines a bounded linear functional in $W^{-1,p'}$ for all $p'\in (d,6]$ (for all $\frac{6}{5}\leq p<\frac{d}{d-1}$) and since $\Gamma \in C^1$, by applying Theorem~\ref{Thm_Optimal_Regularity_for_Elliptic_Interface_Problems}, it follows that $z\in W_0^{1,q_0}(\Omega)$ for some $q_0\in(d,6]$. By a density argument we see that \eqref{Auxiliary_problem_for_uniqueness_of_General_linear_elliptic_problem_with_measure_RHS} holds for all test functions $v\in W_0^{1,q_0'}(\Omega)$ with $1/q_0+1/q_0'=1$. Thus, we can use $z$ as a test function in \eqref{General_linear_problem_with_measure_RHS} ($\mu=0$, $g=0$) and $\varphi$ as a test function in \eqref{Auxiliary_problem_for_uniqueness_of_General_linear_elliptic_problem_with_measure_RHS}. Thus, we obtain
\begin{equation} 
0=\int_{\Omega}{{\bm A}\nabla z\cdot\nabla \varphi \dd x}+\int_{\Omega}{cz \varphi \dd x}=\int_{\Omega}{\theta \varphi \dd x}.
\end{equation} 
Since $\theta$ was an arbitrary function in $L^\infty(\Omega)$, it follows that $\varphi=0$ a.e. in $\Omega$.
\end{proof}

In light of this existence and uniqueness result, this notion of weak solution is indeed applicable for the LGPBE:

\begin{definition}\label{definition_weak_solution_LGPBE_W1p_spaces}
A measurable function $\phi$ is called a weak solution of \eqref{LGPBE_dimensionless} if it satisfies
\begin{equation} \label{LGPBE_dimensionless_Weak_formulation_W1p_spaces}
\begin{aligned}
\phi \in \Mb_{g_\Omega}\ \text{ and }\ \int_{\Omega}{\epsilon \nabla \phi\cdot \nabla v \dd x}+\int_{\Omega}{\overline{m}^2\phi v\dd x}=\langle \Fb,v\rangle + \int_{\Omega}{\ell v\dd x} \text{ for all } v\in \Nb.
\end{aligned}\tag{wLGPBE}
\end{equation} 
\end{definition}

\subsection{Semilinear elliptic equations with measure right hand side}

A natural way to extend the weak formulation \eqref{LGPBE_dimensionless_Weak_formulation_W1p_spaces} to the semilinear case of the GPBE is as follows:
\begin{definition}\label{definition_weak_formulation_full_GPBE_dimensionless}
We call $\phi$ a weak solution of problem \eqref{GPBE_dimensionless} if
\begin{equation}\label{weak_formulation_General_PBE_W1p_spaces}
\begin{aligned}
&\phi \in \Mb_{g_\Omega}, \ \ {b(x,\phi)v\in L^1(\Omega)}\text{ for all }v\in \Nb, \text{ and}\\
&\int_{\Omega}{\epsilon\nabla \phi\cdot\nabla v \dd x}+\int_{\Omega}{b(x,\phi)v \dd x}=\langle \Fb, v\rangle \, \text{ for all } \,v\in \Nb.
\end{aligned}\tag{wGPBE}
\end{equation} 
\end{definition}

The approximation schemes used for existence of this type of solutions in Theorems 2 and 3 of \cite{Boccardo_Gallouet_1989} treat $L^1$ data and with no growth condition on the semilinear term, or measure data but with growth conditions on it. In our case however, the nonlinearity $b$ is not bounded onto any $L^q$ space so its not quite clear how to implement such an approximation scheme. We instead treat existence in Section \ref{Section_Existence_GPBE} by a variational approach that exploits the particular structure of the biomolecular geometry introduced in the previous section. Since the right hand side is supported on $\Omega_m$ but $b$ vanishes on it, so the solution can be split additively reflecting these different contributions, as done in Section \ref{Section_Splitting} below. The energy formally associated to \eqref{definition_weak_formulation_full_GPBE_dimensionless} with $\Fb=0$ is convex and to apply the direct method of the calculus of variations no growth bounds on $b$ are needed, but their absence means that this energy functional is not differentiable, so to go back to the weak formulation we will prove an a priori $L^\infty$ estimate for the minimizers.

The question of existence and uniqueness for more general linear and nonlinear elliptic problems involving measure data is studied in many further works, some notable ones being \cite{Droniou_Gallouet_Herbin_2003,
Boccardo_Murat_property_of_nonlinear_elliptic_PDEs_with_measure_1994,
Bartolucci_Leoni_Orsina_Ponce_exp_nonlinearity_and_measure_2005,
Orsina_Prignet_strong_stability_results_for_elliptic_PDEs_with_measure_2002,
Ponce_book_2016,
Brezis_Marcus_Ponce_Nonlinear_Elliptic_With_Measures_Revisited, Brezis_Nonlinear_Elliptic_Equations_Involving_Measures,
Benilan_Brezis_Thomas_Fermi_Equation_2003}. There are many nontrivial cases, for example in \cite{Brezis_Nonlinear_Elliptic_Equations_Involving_Measures} it is shown that even the simple equation $-\Delta u+\abs{u}^{p-1}u=\delta_a$ with $u=0$ on $\partial\Omega$ and $a\in\Omega$ does not have a solution in $L_{loc}^p(\Omega)$ for any $p\ge \frac{d}{d-2}$ when $d\ge 3$.

\section{Electrostatics of point charges and solution splittings}\label{Section_Splitting}
As we have already mentioned, we aim to use the particular geometry and coefficients of the biomolecular setting to show existence of solutions for \eqref{GPBE_dimensionless} in the weak sense of \eqref{weak_formulation_General_PBE_W1p_spaces}. This is done by an additive splitting of the solutions based on the Green function for the Poisson equation on the full space, or in more physical terms, the Coulomb potential for electrostatics in a uniform dielectric. This procedure is common in the applications literature, so we orient ourselves to the same kind of decompositions done there, which we explain in this section.

To this end, define the function $G: \Omega \to \overline{\R}$ by
\begin{equation} \label{expression_for_G_2d_and_3d}
\begin{aligned}
G(x)&=\sum_{i=1}^{N_m}{G_i(x)}=-\frac{2 e_0^2}{\epsilon_m k_BT}\sum_{i=1}^{N_m}{z_i \ln{|x-x_i|}}\quad\text{ if }d=2,\\
G(x)&=\sum_{i=1}^{N_m}{G_i(x)}=\frac{e_0^2}{\epsilon_m k_BT}\sum_{i=1}^{N_m}{\frac{z_i}{|x-x_i|}}\quad\text{ if }d=3.
\end{aligned}
\end{equation} 
This function describes the singular or Coulomb part of the potential due to the point charges $\{z_ie_0\}_{i=1}^{N_m}$ in a uniform dielectric medium with a dielectric constant $\epsilon_m$. It satisfies
\begin{equation} \label{PDE_for_G}
-\nabla\cdot(\epsilon_m \nabla G)=\Fb\quad\text{in }\mathbb R^d,\, d\in\{2,3\},
\end{equation}  
in the sense of distributions, that is
\begin{equation} \label{Distributional_Laplacian_for_G}
-\int_{\mathbb R^d}{\epsilon_m G \Delta v \dd x}=\langle \Fb,v\rangle\quad\text{for all }v\in C_c^\infty(\mathbb R^d).
\end{equation} 
In particular, \eqref{Distributional_Laplacian_for_G} is valid for all $v\in C_c^\infty(\Omega)$. Note that $G$ and $\nabla G$ are in $L^p(\Omega)$ for all $p<\frac{d}{d-1}$ and thus $G\in \Mb_G=\bigcap_{p<\frac{d}{d-1}}{W_G^{1,p}(\Omega)}$. This means that we can integrate by parts on \eqref{Distributional_Laplacian_for_G} to obtain
\begin{equation}\label{Weak_formulation_for_G_smooth_test_functions}
\int_{\Omega}{\epsilon_m \nabla G \cdot \nabla v \dd x}=\langle \Fb,v\rangle\quad\text{for all }v\in C_c^\infty(\Omega).
\end{equation} 
For a fixed $q>d$, owing to the Sobolev embedding $W_0^{1,q}(\Omega)\hookrightarrow C_0(\overline\Omega)$ (see Theorem 4.12 in \cite{Adams_Fournier_Sobolev_Spaces}), 
$\Fb$ is bounded on $W_0^{1,q}(\Omega)$, i.e.,
\begin{equation*}
\abs{\langle \Fb,v\rangle}=\Bigg|\frac{4\pi e_0^2}{k_BT}\sum_{i=1}^{N_m}{z_i v(x_i)}\Bigg|\leq \frac{4\pi e_0^2}{k_BT}\sum_{i=1}^{N_m}{\abs{z_i}}\|v\|_{L^\infty(\Omega)}\leq C_E \frac{4\pi e_0^2}{k_BT}\sum_{i=1}^{N_m}{\abs{z_i}}\|v\|_{W^{1,q}(\Omega)},
\end{equation*}
where $C_E$ is the constant in the inequality $\|v\|_{L^\infty(\Omega)}\leq C_E\|v\|_{W^{1,q}(\Omega)}$.
 Since $C_c^\infty(\Omega)$ is dense in $W_0^{1,q}(\Omega)$, we see that \eqref{Weak_formulation_for_G_smooth_test_functions} is valid for all $v\in W_0^{1,q}(\Omega)$, and consequently, for all $v\in \Nb=\bigcup_{q>d}{W_0^{1,q}(\Omega)}$. Hence, the electrostatic potential $G$ generated by the charges $\{e_0z_i\}_{i=1}^{N_m}$ in a uniform dielectric with the dielectric coefficient $\epsilon_m$ belongs to  $\Mb_{G}$ and satisfies the integral relation
\begin{equation}\label{Weak_formulation_for_G}
\int_{\Omega}{\epsilon_m \nabla G \cdot \nabla v \dd x}=\langle \Fb,v\rangle\quad\text{for all }v\in \Nb,
\end{equation} 
indicating that subtracting $G$ from a weak solution $\phi$ satisfying either \eqref{LGPBE_dimensionless_Weak_formulation_W1p_spaces} or \eqref{weak_formulation_General_PBE_W1p_spaces} allows us to remain within the same weak solution framework. In fact \eqref{Weak_formulation_for_G} can be seen as motivation for this notion of solution, since general measure data cannot be more singular than a point charge.

This observation leads to the definition of linear 2-term and 3-term splittings of~$\phi$ based on $G$, which we describe below. In Section \ref{Section_Existence_GPBE} we will use these splittings to obtain the existence of a solution to \eqref{LGPBE_dimensionless_Weak_formulation_W1p_spaces} and \eqref{weak_formulation_General_PBE_W1p_spaces} without having to deal with the measure data $\Fb$ directly. Moreover, if $\phi$ is unique as proved under mild assumptions in Section \ref{Section_uniqueness_for_GPBE}, then there is no difference between the particular solutions found by 2-term and 3-term splitting, providing justification for these commonly used strategies.

\subsection{2-term splitting}\label{Subsection_2-term_splitting_GPBE}
As anticipated above and also commonly used in practice (see, e.g. \cite{Chen2006b,Zhou_Payne_Vasquez_Kuhn_Levitt_1996,Niedermeier_Schulten_1992}) we can split the full potential of \eqref{weak_formulation_General_PBE_W1p_spaces} as $\phi=u+G$, where $u$ is a well behaved regular component and $G$ is defined by \eqref{expression_for_G_2d_and_3d}. In this case $u$ is usually called the reaction field potential, which accounts for the forces acting on a biomolecule due to the presence of the solvent (see \cite{Niedermeier_Schulten_1992,Rogers_Sternberg_1984}). Taking into account \eqref{Weak_formulation_for_G}, we obtain the following integral identity for $u$:
\begin{equation}\label{Weak_formulation_for_uregular_General_PBE_W1p_spaces}
\begin{aligned}
&\text{Find }u\in \Mb_{g_\Omega-G} \text{ such that } b(x,u+G)v \in L^1(\Omega)\,\text{ for all }v\in \Nb\,\text{ and }\\
&\int_{\Omega}{\epsilon \nabla u\cdot \nabla v \dd x}+\int_{\Omega}{b(x,u+G)v \dd x}=\int_{\Omega_s}{(\epsilon_m-\epsilon_s)\nabla 
G\cdot\nabla v \dd x}=:\langle \Gb_2,v\rangle \, \text{ for all } \,v\in \Nb.
\end{aligned}
\end{equation}
The advantage of this formulation is that in contrast to the situation in \eqref{weak_formulation_General_PBE_W1p_spaces} the right hand side $\Gb_2$ belongs to $H^{-1}(\Omega)$ and is supported on $\Gamma$ if $\epsilon_s$ is constant (see Remark~\ref{Remark_RHS_2-term_splitting_rewritten} below). Noticing that $H_{g_\Omega-G}^1(\Omega)\subset \Mb_{g_\Omega-G}$ we can consider the weak formulation with $H^1$ trial space
\begin{equation}\label{3_weak_formulations_2_term_splitting_uregular}
\begin{aligned}
&\text{Find } u\in H_{g_\Omega-G}^1(\Omega)\text{ such that }  b(x,u+G)v\in L^1(\Omega) \text{ for all } v\in \Wb \,\text{ and }\\ 
&\int_{\Omega}{\epsilon \nabla u\cdot \nabla v \dd x}+\int_{\Omega}{b(x,u+G)v\dd x}=\langle \Gb_2,v\rangle \,\text{ for all }v\in \Wb.
\end{aligned}
\end{equation} 
In \eqref{3_weak_formulations_2_term_splitting_uregular} we don't fix the testing space $\Wb$ yet, since proving existence of such a $u$ will be nontrivial. In any case we remark that we can go back to a solution $\phi$ of \eqref{weak_formulation_General_PBE_W1p_spaces} as soon as $\Nb \subset \Wb$. For example, using $\Wb=H_0^1(\Omega)$ or $\Wb=H_0^1(\Omega)\cap L^\infty(\Omega)$ (for which finding $u$ is clearly easier) would be enough, since functions in $\Nb$ are bounded by the Sobolev embedding $W^{1,q}(\Omega)\subset L^\infty(\Omega)$ for $q>d$.
\begin{remark}\label{Remark_RHS_2-term_splitting_rewritten}
Recalling that $\epsilon_s\in C^{0,1}(\overline \Omega_s)$ (and therefore $\epsilon_s\in W^{1,\infty}(\Omega_s)$) and that $G$ is harmonic in a neighborhood of $\Omega_s$, we see that $\bm \psi:=(\epsilon_m-\epsilon_s)\nabla G\in \left[H^1(\Omega_s)\cap L^\infty(\Omega_s)\right]^d$ and that its weak divergence is given by $\div(\bm \psi)=\nabla(-\epsilon_s)\cdot\nabla G+(\epsilon_m-\epsilon_s) \Delta G$ (see Proposition 9.4 in \cite{Brezis_FA}). Thus, we can rewrite the term $\langle\Gb_2,v\rangle$ on the right-hand side of \eqref{3_weak_formulations_2_term_splitting_uregular} by applying the integration by parts formula:
\begin{equation}\label{RHS_2-term_splitting_rewritten}
\begin{aligned}
\int_{\Omega_s}{(\epsilon_m-\epsilon_s)\nabla G\cdot\nabla v \dd x}=&-\int_{\Gamma}{(\epsilon_m-\epsilon_s)\nabla G\cdot \bm n_\Gamma\, v \dd s} + \int_{\partial\Omega}{(\epsilon_m-\epsilon_s)\nabla G\cdot \bm n_{\partial\Omega}\, v \dd s}\\
&\qquad -\int_{\Omega_s}{\left(\nabla(-\epsilon_s)\cdot\nabla G+(\epsilon_m-\epsilon_s) \Delta G\right)v \dd x}\\
=&-\int_{\Gamma}{(\epsilon_m-\epsilon_s)\nabla G\cdot \bm n_\Gamma\, v \dd s}+\int_{\Omega_s}{\nabla \epsilon_s \cdot\nabla G\, v \dd x},
\end{aligned}
\end{equation} 
where the appearances of $v$ should be interpreted as traces if necessary. Now, it is seen that \eqref{3_weak_formulations_2_term_splitting_uregular} is the weak formulation of a nonlinear elliptic interface problem with a jump condition on the normal flux i.e., $\left[\epsilon\nabla u\cdot \bm n_\Gamma\right]_\Gamma=-(\epsilon_m-\epsilon_s)\nabla G\cdot \bm n_\Gamma=-\left[\epsilon\nabla G\cdot \bm n_\Gamma\right]_\Gamma$. Moreover, using the equality \eqref{RHS_2-term_splitting_rewritten} which is valid for $v \in \Nb$ we can go back to \eqref{Weak_formulation_for_uregular_General_PBE_W1p_spaces} and obtain a weak formulation of the reaction field potential analogous to the one for the full potential in \eqref{weak_formulation_General_PBE_W1p_spaces} but with right hand side the measure 
\[\Fb_\Gamma:=\left((\epsilon_m-\epsilon_s)\nabla G\cdot \bm n_\Gamma\right) \Hb^{d-1} \mres\Gamma+\left(\nabla \epsilon_s \cdot\nabla G\right)\,\Lb^d \mres \Omega_s \, \in \Mp(\Omega),\] 
where $\Hb^{d-1} \mres \Gamma$ is the $(d-1)$-dimensional Hausdorff measure restricted to $\Gamma$ and $\Lb^d \mres \Omega_s$ is the Lebesgue measure restricted to $\Omega_s$. That is, formally we have
\begin{equation}\label{Elliptic_equation_for_2-term_u}
\begin{aligned}
-\nabla\cdot\left(\epsilon\nabla u\right)+b(x,u+G)&=\Fb_\Gamma \quad &\text{ in } \Omega,\\
\phi&=g_\Omega-G\quad &\text{ on } \partial\Omega.
\end{aligned}
\end{equation}
By doing this manipulation through the potential $G$ obtained from the Newtonian kernel we have replaced the measure $\Fb$ which seen as a distribution does not belong to $H^{-1}(\Omega)$ with another irregular distribution $\Fb_\Gamma$, which this time is in fact in $H^{-1}(\Omega)$ (acting through the trace on $\Gamma$, which is $C^1$), making it suitable for $H^1$ weak formulations.
\end{remark}
The above considerations imply that in numerical computations, which are one of the main motivations to introduce this splitting, the full potential $\phi$ can be obtained without needing to approximate the singularities that arise at the positions $x_i$ of the fixed partial charges. Some other problems appear however, motivating the introduction of a further splitting in three terms that we discuss in Section \ref{Subsection_3-term_splitting_GPBE} below. One such problem arises when $u$ has almost the same magnitude as $G$ but opposite sign and $\abs{\phi}=\abs{G+u}\ll \abs{u}$ (see e.g.~\cite{Holst2012}). This typically happens in the solvent region $\Omega_s$ and under the conditions that the ratio $\epsilon_m / \epsilon_s$ is much smaller than $1$ and the ionic strength $I_s$ is nonzero. In this case a small relative error in $u$ generates a substantial relative error in $\phi = G+u$. However, the 2-term splitting remains useful in practice, since it allows to directly compute the electrostatic contribution to the solvation free energy through the reaction field potential as $\frac{1}{2}\sum_{i=1}^{N_m}{z_ie_0u(x_i)}.$

\subsection{3-term splitting}
\label{Subsection_3-term_splitting_GPBE}
Although the 2-term splitting we just introduced would suffice to obtain existence of solutions, we describe now another commonly used splitting with the aim of providing some justification for it through our uniqueness results. In it one considers three components $\phi=G+u^H+u$, where $\phi=u$ in $\Omega_s$ and $u^H$ is such that $u^H=-G$ in $\Omega_s$ (see \cite{Chern_Liu_Wang_Electrostatics_for_Macromolecules_2003,Holst2012}). By substituting this expression for $\phi$ into \eqref{weak_formulation_General_PBE_W1p_spaces},  using \eqref{Weak_formulation_for_G}, the fact that $u^H=-G$ in $\Omega_s$, and assuming that $u^H\in \Mb$, we obtain the following weak formulation\footnote{Notice that $\phi,u^H\in \Mb$ implies $u\in\Mb$. In particular, the integral $\int_{\Omega_m}{\epsilon_m\nabla u^H\cdot\nabla v \dd x}$ is well defined.} for $u$:
\begin{equation} \label{Weak_formulation_for_uregular_3-term_splitting_General_PBE_W1p_spaces}
\begin{aligned}
&\text{Find }u\in \Mb_{g_\Omega} \text{ such that } b(x,u)v \in L^1(\Omega)\,\text{ for all }v\in \Nb\,\text{ and }\\
&\int_{\Omega}{\epsilon \nabla u\cdot \nabla v \dd x}+\int_{\Omega}{b(x,u)v \dd x}\\
&\qquad\qquad=-\int_{\Omega_m}{\epsilon_m\nabla u^H\cdot\nabla v \dd x}+\int_{\Omega_s}{\epsilon_m\nabla G\cdot\nabla v \dd x}=:\langle \Gb_3,v\rangle \,\text{ for all } \, v\in \Nb.
\end{aligned}
\end{equation} 
To define $u^H$ in $\Omega_m = \Omega \setminus \overline{\Omega_s}$ we must satisfy the condition $u^H\in \Mb$, which holds in particular if $u^H\in H^1(\Omega)$. Again if $\Nb\subset \Wb$ and since $H_{g_\Omega}^1(\Omega)\subset \Mb_{g_\Omega}$, we can find a particular solution $u$ of \eqref{Weak_formulation_for_uregular_3-term_splitting_General_PBE_W1p_spaces} by considering yet another $H^1$ weak formulation:
\begin{equation}\label{Weak_formulation_for_uregular_3-term_splitting_PBE_H1}
\begin{aligned}
&\text{Find }u\in H_{g_\Omega}^1(\Omega) \text{ such that } b(x,u)v\in L^1(\Omega)\, \text{ for all } v\in  \Wb\, \text{ and } \\
&\int_{\Omega}{\epsilon \nabla u\cdot \nabla v \dd x}+\int_{\Omega}{b(x,u)v \dd x}=\langle \Gb_3,v\rangle \,\text{ for all } v\in \Wb. 
\end{aligned}
\end{equation} 
where again $\Gb_3 \in H^{-1}(\Omega)$ since we have chosen $u^H$  in $H^1(\Omega)$. Testing \eqref{Weak_formulation_for_uregular_3-term_splitting_PBE_H1} with functions $v$ supported in $\Omega_m$ and such that $v\in H_0^1(\Omega_m)$ we obtain\footnote{Note that one can immediately test \eqref{Weak_formulation_for_uregular_3-term_splitting_PBE_H1} with $v\in H_0^1(\Omega_m)\cap L^\infty(\Omega_m)$. Now, since in this case we obtain a linear problem in $\Omega_m$, by a standard density argument one sees that this linear problem can always be tested with $v\in H_0^1(\Omega_m)$.}
\[
\int_{\Omega_m}{\epsilon_m \nabla u\cdot \nabla v \dd x}=-\int_{\Omega_m}{\epsilon_m\nabla u^H\cdot\nabla v \dd x}\, \text{ for all }v\in H_0^1(\Omega_m).
\]
It is particularly convenient (for example in a posteriori error analysis, see \cite{Kraus_Nakov_Repin_PBE2_2020}) to impose that $u^H$ is weakly harmonic in $\Omega_m$, that is
\begin{equation} \label{definition_of_uH_3-term_splitting}
\begin{aligned}
&u^H\in H_{-G}^1(\Omega_m)\quad\text{ and }\quad\int_{\Omega_m}{\nabla u^H\cdot\nabla v \dd x}=0 \text{ for all } v\in H_0^1(\Omega_m),
\end{aligned}
\end{equation} 		
where the Dirichlet boundary condition $u^H=-G$ on $\partial\Omega$ ensures that $u^H$ has the same trace on $\Gamma$ from both sides and therefore $u^H\in H^1(\Omega)$.	

\begin{remark}\label{Remark_right_hand_side_functional_3-term_splitting}
In this case the right-hand side of equation \eqref{Weak_formulation_for_uregular_3-term_splitting_PBE_H1} depends on the solution of \eqref{definition_of_uH_3-term_splitting}, meaning in particular that for numerical approximations two concatenated elliptic problems  have to be solved. Moreover, by the divergence theorem (see, e.g., Theorem 2 in Section 1.3 in \cite{Dautray_Lions_Volume_6}, Theorem 3.24 in \cite{Peter_Monk}) we can compute
\begin{equation}\label{RHS_3-term_splitting_rewritten}
\begin{aligned}
\left\langle \Gb_3,v\right\rangle=&-\int_{\Omega_m}{\epsilon_m\nabla u^H\cdot\nabla v \dd x}+\int_{\Omega_s}{\epsilon_m\nabla G\cdot\nabla v \dd x}\\
=&-\big\langle \gamma_{\bm n_\Gamma,\Omega_m}\big(\epsilon_m\nabla u^H\big),\gamma_{2,\Gamma}(v)\big\rangle_{H^{-1/2}(\Gamma)\times H^{1/2}(\Gamma)}\\
&\qquad+\langle \gamma_{\bm n_\Gamma,\Omega_s}\left(\epsilon_m\nabla G\right),\gamma_{2,\Gamma}(v)\rangle_{H^{-1/2}(\Gamma)\times H^{1/2}(\Gamma)},
\end{aligned}
\end{equation} 
where we used that $\epsilon_m$ is constant, $\nabla u^H\in H(\div;\Omega_m)$ (see \eqref{definition_of_uH_3-term_splitting}), and that $G$ is harmonic in a neighborhood of $\Omega_s$. In \eqref{RHS_3-term_splitting_rewritten}, $\gamma_{\bm n_\Gamma,\Omega_m}$ and $\gamma_{\bm n_\Gamma,\Omega_s}$ are the normal trace operators in $H(\div;\Omega_m)$ and $H(\div;\Omega_s)$, respectively, and $\gamma_{2,\Gamma}$ is the trace of $v$ on $\Gamma$. This computation tells us that $\Gb_3 \in H^{-1}(\Omega)$ but in contrast to the situation in Remark \ref{Remark_RHS_2-term_splitting_rewritten}, it does not allow us to immediately conclude that the action of $\Gb_3$ can be interpreted as a measure. This is only possible when $\nabla u^H \cdot \bm n_\Gamma$ is regular enough to be defined pointwise, as is the case when $\Gamma$ is $C^3$ so that we can use regularity estimates up to the boundary (see for example Theorem 9.25 in \cite{Brezis_FA}) that would provide $u^H \in H^3(\Omega_m)$. In this smoother situation \eqref{RHS_3-term_splitting_rewritten} can be reformulated as
\begin{equation} \label{RHS_3-term_splitting_rewritten_nabla_u^H_with_L2_trace_on_Gamma}
\langle \Gb_3,v\rangle=\int_{\Gamma}{-\epsilon_m\nabla\big( u^H+G\big)\cdot \bm n_\Gamma\, \gamma_{2,\Gamma}(v) \dd s},\,\mathrm{for}\ \mathrm{all}\  v\in H_0^1(\Omega),
\end{equation} 
and in this case $\Gb_3$ can be thought of as a measure for a problem analogous to \eqref{Elliptic_equation_for_2-term_u} and again represents a jump condition on the normal component of $\epsilon\nabla u$. That is, if the function $u$ is smooth in $\Omega_m$ and $\Omega_s$ it should satisfy the jump condition $\left[\epsilon\nabla u\cdot \bm n_\Gamma\right]_\Gamma=-\epsilon_m \nabla\left(u^H+G\right)\cdot \bm n_\Gamma$.
\end{remark}

Within our context  we can obtain some milder regularity of $u^H$ without additional assumptions on $\Gamma$:

\begin{proposition}\label{proposition_u^H_is_in_W1s_s_ge_d}
If $u^H$ is defined as in Section~\ref{Subsection_3-term_splitting_GPBE}, i.e., $u^H\in H^1(\Omega)$ with $u^H=-G$ in $\Omega_s$ and satisfies \eqref{definition_of_uH_3-term_splitting}, then $u^H\in W^{1,\bar q}(\Omega)$ for some $\bar q>d$.
\end{proposition}
\begin{proof}
Since $\Gamma=\partial\Omega_m$ is Lipschitz (it is even $C^1$ by assumption), we can apply Theorem~\ref{Thm_Optimal_Regularity_for_Elliptic_Interface_Problems} on $\Omega_m$ to the homogenized version of \eqref{definition_of_uH_3-term_splitting}:
\begin{equation} \label{definition_of_uH_3-term_splitting_homogenized}
\begin{aligned}
&\text{Find } u_0^H\in H_0^1(\Omega_m)\text{ such that }\\
&\int_{\Omega_m}{\nabla u_0^H\cdot\nabla v \dd x}=-\int_{\Omega_m}{\nabla u_{-G}^H\cdot\nabla v \dd x} \text{ for all } v\in H_0^1(\Omega_m),
\end{aligned}
\end{equation} 	
where $u_{-G}^H\in H^1(\Omega_m)$ and $\gamma_2\left(u_{-G}^H\right)=-G$ on $\Gamma$. We can choose $u_{-G}^H$ to be in the space $W^{1,\infty}(\Omega)$ by noting that $G\in C^{0,1}(\Gamma)$ and using a Lipschitz extension (see Lemma~\ref{Lemma_extension_pf_Lipschitz}). We can even choose $u_{-G}^H$ to be smooth in $\Omega_m$. To see this, let $r>0$ be so small that all balls $B(x_i,r)$ centered at $x_i,\,i=1,\ldots, N_m$ (the locations of the point charges, as defined after \eqref{general_PBE_1_classical_form}) and with radius $r$ are strictly contained in $\Omega_m$. Then, we define the function $u_{-G}^H:=\restr{(\psi G)}{\overline\Omega_m}\in C^\infty(\overline{\Omega_m})$, where $\psi\in C_c^\infty(\mathbb R^d)$ is  such that it is equal to $1$ in a neighborhood of $\Gamma$ and with support in $\mathbb R^d\setminus \bigcup_{i=1}^{N_m}{B(x_i,r)}$. It follows that the right-hand side of \eqref{definition_of_uH_3-term_splitting_homogenized} defines a bounded linear functional over $W_0^{1,p}(\Omega_m)$ for all $1\leq p<\infty$ and by Theorem~\ref{Thm_Optimal_Regularity_for_Elliptic_Interface_Problems} we conclude that $u_0^H\in W^{1,\overline{q}}(\Omega_m)$ for some $\overline{q}>d$. Now, $u^H=u_{-G}^H+u_0^H\in W^{1,\overline{q}}(\Omega_m)$.
\end{proof}	

\subsection{Splitting for the linearized GPBE}\label{Section_The_LGPBE}	
The 2- and 3-term splittings can also be applied in the case of the linearized GPBE, as done routinely in numerical works \cite{Niedermeier_Schulten_1992, Zhou_Payne_Vasquez_Kuhn_Levitt_1996, Bond_2009}. After substituting the expressions $\phi=G+u$ and $\phi=G+u^H+u$ into \eqref{LGPBE_dimensionless_Weak_formulation_W1p_spaces} we obtain the respective weak formulations which the regular component $u$ has to satisfy in each case. Those formulations can be written in one common form:
\begin{equation} \label{common_form_weak_formulation_u_W1p_spaces_LGPBE}
\begin{aligned}
&\text{Find }u\in \Mb_{\overline g} \text{ such that }\\
&\int_{\Omega}{\epsilon \nabla u\cdot \nabla v \dd x}+\int_{\Omega}{\overline{m}^2uv \dd x}=\int_{\Omega}{\bm f\cdot\nabla v \dd x} + \int_{\Omega}{f_0v \dd x}\,\text{ for all } v\in \Nb,
\end{aligned}
\end{equation} 
where in the case of the 2-term splitting we have
\begin{equation} \label{f0_barg_f_expressions_2-term_splitting}
f_0=-\overline m^2G+\ell, \quad {\bm f}={\bm f}_{\Gb_2}:=\chi_{\Omega_s}(\epsilon_m-\epsilon_s)\nabla G,\quad \text{ and }\quad\overline{g}=g_\Omega-G \text{ on }\partial\Omega,
\end{equation} 
 whereas in the case of the 3-term splitting we have
 \begin{equation} \label{f0_barg_f_expressions_3-term_splitting}
f_0=\ell, \quad {\bm f}={\bm f}_{\Gb_3}:=-\chi_{\Omega_m}\epsilon_m\nabla u^H+\chi_{\Omega_s}\epsilon_m\nabla G,\quad \text{ and }\quad \overline{g}=g_\Omega \text{ on }\partial\Omega.
\end{equation} 
We recall that $\ell$ and $\overline m^2$ as defined in \eqref{definition_overline_m_and_l} are zero in $\Omega_m\cup\Omega_{IEL}$ and  constants in $\Omega_{ions}$, $G$ is harmonic in a neighborhood of $\Omega_s$, and $u^H\in H^1(\Omega)\subset \Mb$. Therefore, all integrals in \eqref{common_form_weak_formulation_u_W1p_spaces_LGPBE} are well defined.
By observing that $\Nb\subset H^1(\Omega)$ and $H_{\overline g}^1(\Omega)\subset \Mb_{\overline g}$ we can find a particular solution $u$ of problem \eqref{common_form_weak_formulation_u_W1p_spaces_LGPBE} by posing a standard $H^1$ weak formulation for $u$: the trial space in \eqref{common_form_weak_formulation_u_W1p_spaces_LGPBE} is swapped with $H_{\overline g}^1(\Omega)$ and the test space is exchanged for $H_0^1(\Omega)$.

An application of Theorem~\ref{Theorem_uniqueness_of_linear_elliptic_problems_with_diffusion_and_measure_rhs} provides us with existence of a solution $\phi$ to \eqref{LGPBE_dimensionless_Weak_formulation_W1p_spaces} by the approximation strategy of \cite{Boccardo_Gallouet_1989}, which is also unique because $\Gamma\in C^1$ by Assumption~\ref{Assumption_Domain_permittivity}.
\begin{theorem}\label{Theorem_well_posedness_for_full_potential_LPBE}
 The unique weak solution $\phi$ of equation \eqref{LGPBE_dimensionless_Weak_formulation_W1p_spaces} can be given either in the form $\phi=G+u$ or in the form $\phi=G+u^H+u$, where $u\in H_{\overline g}^1(\Omega)$ is the unique solution of the problem
 \begin{equation} \label{common_form_weak_formulation_u_H1_spaces_LGPBE}
\begin{aligned}
&\text{Find }u\in H_{\overline g}^1(\Omega) \text{ such that }\\
&\int_{\Omega}{\epsilon \nabla u\cdot \nabla v \dd x}+\int_{\Omega}{\overline{m}^2uv \dd x}=\int_{\Omega}{\bm f\cdot\nabla v \dd x} + \int_{\Omega}{f_0v \dd x}\,\text{ for all } v\in H_0^1(\Omega)
\end{aligned}
\end{equation} 
with $f_0$, $\bm f$, $\overline g$ defined by either \eqref{f0_barg_f_expressions_2-term_splitting} or \eqref{f0_barg_f_expressions_3-term_splitting} for the 2- or 3-term splittings, respectively.
 
\end{theorem}

\begin{proof}
We will only show the existence of a solution $\phi$ of \eqref{LGPBE_dimensionless_Weak_formulation_W1p_spaces} by using the 2-term splitting where  $f_0,\,\bm f,\,\overline g$ are defined by \eqref{f0_barg_f_expressions_2-term_splitting}, the case of the 3-term splitting is similar. 

By using an extension of $\overline g$ (see Lemma~\ref{Lemma_extension_pf_Lipschitz}) and linearity we can reduce to homogeneous boundary conditions and use the Lax-Milgram Theorem to obtain a unique solution $u\in H_{\overline g}^1(\Omega) = H_{g_\Omega-G}^1(\Omega)$ of \eqref{common_form_weak_formulation_u_H1_spaces_LGPBE}. It is clear that $u$ is also in $\Mb_{g_\Omega-G}$ since $p<\frac{d}{d-1}\leq2$ and $W^{1,2}(\Omega)\equiv H^1(\Omega)$. Therefore, $G+u\in \Mb_{g_\Omega}$. Moreover, $H_0^1(\Omega)\supset \Nb$, and therefore \eqref{common_form_weak_formulation_u_H1_spaces_LGPBE} is valid for all test functions $v\in \Nb$. By adding together \eqref{Weak_formulation_for_G} and \eqref{common_form_weak_formulation_u_H1_spaces_LGPBE} we conclude that $\phi=G+u$ satisfies the weak formulation \eqref{LGPBE_dimensionless_Weak_formulation_W1p_spaces}. 
\end{proof}
Let us note that even without the regularity assumption $\Gamma \in C^1$ one would still get particular solutions $\phi$ given by the two splittings above. However, it would not be clear if these are equal.

\section{Existence and uniqueness for the nonlinear GPBE}\label{Section_Existence_and_uniqueness_of_solution}
\label{Section_The_GPBE}		
For existence, our strategy is to consider either the 2-term or 3-term splitting to separate the effect of the singular right hand side. For the regular components of these, since $H^1 \subset \Mb$ and $\Nb\subset H_0^1$ it is enough to consider an $H^1$ formulation, of which we give a complete treatment. This treatment still requires some care. Since the nonlinearity $b$ of \eqref{weak_formulation_General_PBE_W1p_spaces} has exponential growth, the functional in the minimization problem corresponding to this $H^1$ formulation is not differentiable, so its minimizers do not automatically satisfy the formulation with $H^1_0$ test functions. To conclude that they do, we need an a priori $L^\infty$ estimate for them, which we prove in a slightly more general situation than the one of the GPBE. For uniqueness, we work directly on the original formulation \eqref{weak_formulation_General_PBE_W1p_spaces}, which is the best possible scenario.

\subsection{Existence of a full potential \texorpdfstring{$\phi$}{phi}}\label{Section_Existence_GPBE}
Equations \eqref{Weak_formulation_for_uregular_General_PBE_W1p_spaces} and \eqref{Weak_formulation_for_uregular_3-term_splitting_General_PBE_W1p_spaces} for the regular component $u$ can be written in one common form:
\begin{equation} \label{general_form_of_the_weak_formulation_for_u_2_and_3_term_splitting_W1p_spaces}
\begin{aligned}
&\text{Find }u\in \Mb_{\overline g} \text{ such that } b(x,u+w)v\in L^1(\Omega) \,\text{ for all } v\in \Nb \text{ and }\\
&a(u,v)+\int_{\Omega}{b(x,u+w)v \dd x}=\int_{\Omega}{{\bm f}\cdot\nabla v \dd x}\, \text{ for all } v\in \Nb,
\end{aligned}
\end{equation} 
where $\Mb_{\overline g},\, \Nb$ are as defined in \eqref{Definition_Mg} and \eqref{Definition_M_and_N}, denoting $a(u,v):=\int_{\Omega}{\epsilon\nabla u\cdot\nabla v \dd x}$, $w\in L^\infty(\Omega_{ions})$, ${\bm f}=(f_1,f_2,\ldots, f_d)\in \left[L^s(\Omega)\right]^d$ with\footnote{For the 2-term splitting, ${\bm f}$ is obviously in $\left[L^s(\Omega)\right]^d$ for some $s>d$ since $G$ is smooth in $\overline \Omega_s$ and $\epsilon_s\in C^{0,1}(\overline \Omega_s)$. In the case of the 3-term splitting, from Proposition~\ref{proposition_u^H_is_in_W1s_s_ge_d} it follows that $\nabla u^H\in \left[L^s(\Omega_m)\right]^d$ for some $s>d$ since $\Gamma \in C^1$.} $s>d$, and $\overline g$ specifies a Dirichlet boundary condition for $u$ on $\partial\Omega$. In the case of the 2-term splitting we have
\begin{equation} \label{w_barg_f_expressions_2-term_splitting}
w=G, \quad {\bm f}={\bm f}_{\Gb_2}:=\chi_{\Omega_s}(\epsilon_m-\epsilon_s)\nabla G,\quad \text{ and }\quad\overline{g}=g_\Omega-G \text{ on }\partial\Omega,
\end{equation} 
 whereas in the case of the 3-term splitting we have
 \begin{equation} \label{w_barg_f_expressions_3-term_splitting}
 w=0, \quad {\bm f}={\bm f}_{\Gb_3}:=-\chi_{\Omega_m}\epsilon_m\nabla u^H+\chi_{\Omega_s}\epsilon_m\nabla G,\quad \text{ and }\quad \overline{g}=g_\Omega \text{ on }\partial\Omega.
 \end{equation} 

Similarly, equations \eqref{3_weak_formulations_2_term_splitting_uregular} and \eqref{Weak_formulation_for_uregular_3-term_splitting_PBE_H1}, which determine \emph{particular} representatives for the regular component $u$, can also be written in one common form:
\begin{equation}\label{general_form_of_the_weak_formulation_for_u_2_and_3_term_splitting_H1_spaces}
\begin{aligned}
&\text{Find }u\in H_{\overline g}^1(\Omega) \text{ such that } b(x,u+w)v\in L^1(\Omega) \,\text{ for all } v\in \Wb \text{ and }\\
&a(u,v)+\int_{\Omega}{b(x,u+w)v \dd x}=\int_{\Omega}{{\bm f}\cdot\nabla v \dd x}\, \text{ for all } v\in \Wb,
\end{aligned}\tag{RCH1}
\end{equation} 
where we will consider the three test spaces 
\begin{equation}\label{what_is_Wb}
\Wb=H_0^1(\Omega), \ \ \Wb=H_0^1(\Omega)\cap L^\infty(\Omega)\ \text{ and }\ \Wb=C_c^\infty(\Omega).
\end{equation}
Of course, the larger the test space $\Wb$ the harder it will be to prove existence, and the other way around for uniqueness. For the first two, we have the inclusion $\Nb\subset \Wb$, which combined with $H_{\overline g}^1(\Omega)\subset \Mb_{\overline g}$ makes it clear that if $u\in H_{\overline g}^1(\Omega)$ solves \eqref{general_form_of_the_weak_formulation_for_u_2_and_3_term_splitting_H1_spaces} with either $\Wb=H_0^1(\Omega)$ or $\Wb=H_0^1(\Omega)\cap L^\infty(\Omega)$, then $u$ also solves \eqref{general_form_of_the_weak_formulation_for_u_2_and_3_term_splitting_W1p_spaces}. Consequently, we obtain a particular solution $\phi$ of \eqref{weak_formulation_General_PBE_W1p_spaces} through the formula $\phi=G+u$ in the case where $w, {\bm f}, \overline g$ are given by \eqref{w_barg_f_expressions_2-term_splitting} and through the formula $\phi=G+u^H+u$ in the case where $w, {\bm f}, \overline g$ are given by \eqref{w_barg_f_expressions_3-term_splitting}.

Thus, our goal in this section is to show existence and uniqueness of a solution $u$ to the weak formulation \eqref{general_form_of_the_weak_formulation_for_u_2_and_3_term_splitting_H1_spaces}. We will mainly work with the first two spaces in \eqref{what_is_Wb}, while the third represents distributional solutions where we will see that uniqueness can still be obtained.

\begin{remark}\label{Remark_exponents_of_H01_functions_in_d_2_and_3}
If $d=2$ then by the Moser-Trudinger inequality $\me^{u_0}\in L^2(\Omega)$ for all $u_0\in H_0^1(\Omega)$ (see \cite{Trudinger_1967,Best_constants_in_some_exponential_Sobolev_inequalities}) and, therefore, $e^{u_{\overline g}+u_0}\in L^2(\Omega)$ with $u_{\overline g}$ an extension of $\overline g$ as in Lemma~\ref{Lemma_extension_pf_Lipschitz}. Consequently, for $d=2$, $b(x,u+w)\in L^2(\Omega)$ for all $u\in H_{\overline g}^1(\Omega)$ and the weak formulations \eqref{general_form_of_the_weak_formulation_for_u_2_and_3_term_splitting_H1_spaces} with $\Wb=H_0^1(\Omega)\cap L^\infty(\Omega)$ and with $\Wb=H_0^1(\Omega)$ are equivalent by a density argument. For $d\ge 3$ the situation is more complicated: consider for example $u=\ln{\frac{1}{|x|^d}}\in H_0^1(B(0,1))$ on the unit ball $B(0,1) \subset \mathbb R^d$, for which $\me^u\notin L^1(B(0,1))$. This also means that the condition $b(x,u+w)v \in L^1(\Omega)$ for all $v\in \Wb$ used in \eqref{general_form_of_the_weak_formulation_for_u_2_and_3_term_splitting_H1_spaces} is not superfluous.
\end{remark}

We prove existence by considering the natural associated convex energy, whose minimizers directly provide solutions for \eqref{general_form_of_the_weak_formulation_for_u_2_and_3_term_splitting_H1_spaces} with $\Wb=H_0^1(\Omega)\cap L^\infty(\Omega)$. To pass to the larger test space $\Wb=H_0^1(\Omega)$ we will prove boundedness of these minimizers in Section \ref{Section_A_priori_Linfty_estimate}.

Let us consider some basic properties of the nonlinearity $b$. Since $\frac{d}{dt}b(x,t)\ge0$ for every $x\in\Omega$ it follows that $b(x,\cdot)$ is monotone increasing. This in particular implies that 
\begin{equation} \label{monotonicity_of_b}
\big(b(x,t_1)-b(x,t_2)\big)\left(t_1-t_2\right)\ge 0,\,\mathrm{for}\ \mathrm{all}\  t_1,t_2\in\mathbb R\ \mathrm{and}\  x\in\Omega.
\end{equation}
For semilinear equations, in addition to monotonicity a sign condition (ensuring that the nonlinearity always has the same sign as the solution) is often assumed in the literature. Since $b(x,0)=-\frac{4\pi e_0^2}{k_BT} \sum_{j=1}^{N_{ions}}{\overline M_j(x)\xi_j}$, when the charge neutrality condition \eqref{charge_neutrality_condition} is satisfied it follows that $b(x,0)=0$ for all $x\in\Omega$. If additionally one uses the 3-term splitting so that $w=0$, we would have such a sign condition for \eqref{general_form_of_the_weak_formulation_for_u_2_and_3_term_splitting_H1_spaces}. However as mentioned in previous sections, we do not impose charge neutrality from the outset and would like to treat both splitting schemes simultaneously, so the sign condition may fail. This poses some difficulties for the boundedness estimates, the context of which is discussed at the start of Section \ref{Section_A_priori_Linfty_estimate}.

An important remark is that truncation methods as used in \cite{Webb_1980} (an easy calculation shows that the assumptions G1 and G2 postulated there are satisfied for the nonlinearity $b$) would also provide existence of solutions for \eqref{general_form_of_the_weak_formulation_for_u_2_and_3_term_splitting_H1_spaces} directly and consequently for \eqref{general_form_of_the_weak_formulation_for_u_2_and_3_term_splitting_W1p_spaces}. Our main focus is therefore on the fact that we may obtain bounded solutions, which on the one hand makes the uniqueness results in Section \ref{Section_uniqueness_for_GPBE} applicable, and on the other leads to weak formulations tested with $\Wb = H^1_0(\Omega)$. These are important in practical applications such as the reliable numerical solution of this equation through duality methods.

\subsubsection{Uniqueness of solutions of \eqref{general_form_of_the_weak_formulation_for_u_2_and_3_term_splitting_H1_spaces} for all test spaces}\label{Section_Uniqueness_in_H1}
First, we prove uniqueness of a solution to \eqref{general_form_of_the_weak_formulation_for_u_2_and_3_term_splitting_H1_spaces} for all three choices of the test space $\Wb$ in \eqref{what_is_Wb}. Suppose that $u_1$ and $u_2$ are two solutions of \eqref{general_form_of_the_weak_formulation_for_u_2_and_3_term_splitting_H1_spaces}. Then, we have
\begin{equation} \label{weak_form_2_and_3_term_splitting_difference_of_two_solutions_test_space_V}
a(u_1-u_2,v)+\int_{\Omega}{\left(b(x,u_1+w)-b(x,u_2+w)\right) v \dd x}=0\,\mathrm{for}\ \mathrm{all}\  v\in \Wb.
\end{equation} 
In the case of $\Wb=H_0^1(\Omega)$, $u_1-u_2\in \Wb$ and thus we can test \eqref{weak_form_2_and_3_term_splitting_difference_of_two_solutions_test_space_V} with $v:=u_1-u_2$ to obtain
\begin{equation*}
a(u_1-u_2,u_1-u_2)+\int_{\Omega}{\left(b(x,u_1+w)-b(x,u_2+w)\right)(u_1-u_2) \dd x}=0.
\end{equation*}
Since $a(\cdot,\cdot)$ is coercive and $b(x,\cdot)$ is monotone increasing, we obtain $u_1-u_2=0$.

In the case when $\Wb=H_0^1(\Omega)\cap L^\infty(\Omega)$ we can test with the (truncated) test functions $T_k(u_1-u_2)\in H_0^1(\Omega)\cap L^\infty(\Omega),\,k\ge 0$, where $T_k(s):=\max\{-k,\min\{k,s\}\}$ and use the monotonicity of $b(x,
\cdot)$ and the coercivity of $a(
\cdot,\cdot)$ to obtain $u_1-u_2=0$. This method and the method that we mentioned for the case $\Wb=H_0^1(\Omega)$ do not work when $\Wb=C_c^\infty(\Omega)$ because neither the difference $u_1-u_2$ of two weak solutions nor its truncations $T_k(u_1-u_2)$ are  necessarily in $C_c^\infty(\Omega)$. We overcome this difficulty by applying Theorem~\ref{Thm_Brezis_Browder_for_the_extension_of_bdd_linear_functional_1978}. 
For this we consider two solutions $u_1$ and $u_2$, so that
\begin{equation} \label{weak_form_2_and_3_term_splitting_difference_of_two_solutions_test_space_C0_infty}
a(u_1-u_2,v)+\int_{\Omega}{\left(b(x,u_1+w)-b(x,u_2+w)\right) v \dd x}=0\ \mathrm{for}\ \mathrm{all}\  v\in C_c^\infty(\Omega).
\end{equation} 
Since $a(u_1-u_2,\cdot)$ defines a bounded linear functional over $H_0^1(\Omega)$, the functional $T_b$ defined by the formula $\langle T_b,v\rangle:=\int_{\Omega}{\left(b(x,u_1+w)-b(x,u_2+w)\right)v \dd x}$ for all $v\in C_c^\infty(\Omega)$ satisfies the condition $T_b\in H^{-1}(\Omega)\cap L_{loc}^1(\Omega)$ in Theorem~\ref{Thm_Brezis_Browder_for_the_extension_of_bdd_linear_functional_1978}. By using the monotonicity of $b(x,\cdot)$ we see that $\left(b(x,u_1+w)-b(x,u_2+w)\right) (u_1-u_2)\ge 0=:f(x)\in L^1(\Omega)$. Therefore by Theorem~\ref{Thm_Brezis_Browder_for_the_extension_of_bdd_linear_functional_1978} (see also Remark \ref{remark_on_the_theorem_of_Brezis_and_Browder_for_the_extension_of_bdd_linear_functional_1978}) it follows that \[\left(b(x,u_1+w)-b(x,u_2+w)\right)(u_1-u_2)\in L^1(\Omega)\] and the duality product ${\langle T_b, u_1-u_2\rangle_{H^{-1}(\Omega)\times H_0^1(\Omega)}}$ coincides with 
\[
\int_{\Omega}{\left(b(x,u_1+w)-b(x,u_2+w)\right) (u_1-u_2) \dd x}. 
\]
 This means that 
\begin{equation} \label{weak_form_2_and_3_term_splitting_difference_of_two_solutions_tested_with_u1_minus_u2}
a(u_1-u_2,u_1-u_2)+\int_{\Omega}{\left(b(x,u_1+w)-b(x,u_2+w)\right) (u_1-u_2) \dd x}=0,
\end{equation} 
which implies $u_1-u_2=0$. Of course, this approach can also be applied to show uniqueness when $\Wb=H_0^1(\Omega)\cap L^\infty(\Omega)$ instead of using the truncations $T_k(u_1-u_2)$. The uniqueness of a solution to \eqref{general_form_of_the_weak_formulation_for_u_2_and_3_term_splitting_H1_spaces} with all three choices of the test space $\Wb$ is now clear.

\subsubsection{Existence of a solution of \eqref{general_form_of_the_weak_formulation_for_u_2_and_3_term_splitting_H1_spaces} with the test spaces $H^1_0 \cap L^\infty$ and $C^\infty_c$}\label{Section_Existence}
We consider the variational problem: 
\begin{equation} \label{Variational_problem_for_J}
\text{Find } u_{\rm{min}}\in H_{\overline{g}}^1(\Omega)\text{ such that } J(u_{\rm{min}})=\min_{v\in H_{\overline{g}}^1(\Omega)}{J(v)},
\end{equation} 
where the functional $J:H_{\overline g}^1(\Omega)\to\mathbb R\cup \{+\infty\}$ is defined by 
\begin{equation} \label{definition_of_J_general_form_General_PBE}
J(v):=\left\{
\begin{aligned}
&\frac{1}{2}a(v,v)+\int_{\Omega}{B(x,v+w)\dd x}-\int_{\Omega}{{\bm f}\cdot \nabla v \dd x},\,\text{ if }  B(x,v+w)\in L^1(\Omega),\\
&+\infty, \text{ if } B(x,v+w) \notin L^1(\Omega)
\end{aligned}
\right.
\end{equation} 
with $B(x,\cdot)$ denoting an antiderivative of the monotone nonlinearity $b(x,\cdot)$ of the GPBE defined in \eqref{definition_b_GPBE}, given by
\begin{equation} \label{nonlinearity_B}
B(x,t):=\frac{4\pi e_0^2}{k_BT} \sum_{j=1}^{N_{ions}}{\overline M_j(x) \me^{-\xi_jt}}\ge 0,\,\mathrm{for}\ \mathrm{all}\  x\in\Omega\ \mathrm{and}\ t\in\mathbb R
\end{equation}
which is clearly convex in $t$, and $w$, ${\bm f}$, $\overline g$ defined either in \eqref{w_barg_f_expressions_2-term_splitting} or \eqref{w_barg_f_expressions_3-term_splitting}. 

We have seen in Remark \ref{Remark_exponents_of_H01_functions_in_d_2_and_3} that if $d\leq 2$ then $\me^v\in L^2(\Omega)$ for all $v\in H_0^1(\Omega)$ and therefore (for $\overline g=0$) $\dom(J)=\left\{v\in H_0^1(\Omega)\text{ such that } J(v)<\infty\right\}  = H_0^1(\Omega)$. However, in dimension $d=3$, $\dom(J)$ is only a convex set and not a linear space (see Section~\ref{Sections_remarks_on_J}). In fact, ${\rm dom}(J)$ is also not closed, since it contains $C_c^\infty(\Omega)$ which is dense in $H_0^1(\Omega)$. If ${\rm dom}(J)$ were closed, it would coincide with $H_0^1(\Omega)$ and we know by Remark \ref{Remark_exponents_of_H01_functions_in_d_2_and_3} that this is not true in dimension $d\ge 3$.

\begin{theorem}\label{Thm_existence_and_uniqueness_of_minimizer_of_J_General_PBE}
Problem \eqref{Variational_problem_for_J} has a unique solution $u_{\rm{min}}\in H_{\overline{g}}^1(\Omega)$.
\end{theorem}

\begin{proof}
Since $\dom(J)$ is convex and $J$ is also convex over $\dom(J)$ it follows that $J$ is convex over $H_{\overline g}^1(\Omega)$. To show existence of a minimizer of $J$ over the set $H_{\overline g}^1(\Omega)$ it is enough to verify the following assertions:
\begin{itemize}
\item[(1)] $H_{\overline g}^1(\Omega)$ is a closed convex set in $H^1(\Omega)$;
\item[(2)] $J$ is proper, i.e., $J$ is not identically equal to $+\infty$ and does not take the value $-\infty$;
\item[(3)] $J$ is sequentially weakly lower semicontinuous (s.w.l.s.c.), i.e., if $\{v_n\}_{n=1}^{\infty}\subset H_{\overline g}^1(\Omega)$ and $v_n\rightharpoonup v$ (weakly in $H_{\overline g}^1(\Omega)$) then $J(v)\leq \liminf_{n\to\infty}{J(v_n)}$;
\item [(4)] $J$ is coercive, i.e., $\lim_{n\to\infty}{J(v_n)}=+\infty$ whenever $\|v_n\|_{H^1(\Omega)}\to \infty$.
\end{itemize}
That  $H_{\overline{g}}^1(\Omega)$ is norm closed in $H^1(\Omega)$ and convex follows easily by the linearity and boundedness of the trace operator $\gamma_2$. 
Assertion (2) is obvious since $\int_{\Omega}{B(x,u+w)\dd x}\ge 0$ and $J(0)$ is finite. To see that (3) is fulfilled, notice that $J$ is the sum of the functionals $v\mapsto A(v):=\frac{1}{2}a(v,v)$, $v\mapsto \int_{\Omega}{B(x,v+w) \dd x}$, and $v\mapsto -\int_{\Omega}{\bm f\cdot \nabla v \dd x}$. The first one is convex and Gateaux differentiable, and therefore s.w.l.s.c. (for the proof of this implication, see, e.g. Corollary VII.2.4 in \cite{Showalter}).

However, for $d=3$, the functional $v\mapsto \int_{\Omega}{B(x,v+w) \dd x}$ is not Gateaux differentiable or even continuous (see Section~\ref{Sections_remarks_on_J}). Nevertheless, one can show that this functional is s.w.l.s.c. using Fatou's lemma and the compact embedding of $H^1(\Omega)$ into $L^2(\Omega)$ as follows. Let $\{v_n\}_{n=1}^{\infty}\subset H^1(\Omega)$ be a sequence which converges weakly in $H^1(\Omega)$ to an element $v\in H^1(\Omega)$, i.e., $v_n\rightharpoonup v$. Since the embedding $H^1(\Omega)\hookrightarrow L^2(\Omega)$ is compact it follows that $v_n\to v$ (strongly) in $L^2(\Omega)$, and therefore we can extract a pointwise almost everywhere convergent subsequence $v_{n_m}(x)\to v(x)$ (see Theorem 4.9 in \cite{Brezis_FA}). Since $B(x,\cdot)$ is a continuous function for any $x\in\Omega$ and $x\mapsto B(x,t)$ is measurable for any $t\in\mathbb R$ it means that $B$ is a Carath\'eodory function and as a consequence the function $x\mapsto B(x,v_{n_m}(x)+w(x))$ is measurable for all $k\in\mathbb N$ (see Proposition 3.7 in \cite{Dacorogna}).
By noting that $B(x,z(x)+w(x))\ge 0$ for all $z\in H^1(\Omega)$ and using the fact that $B(x,\cdot)$ is a continuous function for any $x\in\Omega$, from Fatou's lemma we obtain
\begin{equation}\label{Applying_Fatous_Lemma}
\begin{aligned}
\liminf_{m\to \infty}{\int_{\Omega}{B(x,v_{n_m}(x)+w(x))}\dd x}&\ge \int_{\Omega}{\liminf_{m\to \infty}{B(x,v_{n_m}(x)+w(x))}\dd x}\\
&=\int_{\Omega}{B(x,v(x)+w(x))\dd x}.
\end{aligned}
\end{equation} 
Now it is clear that if $\{v_{n_m}\}_{m=1}^\infty$ is an arbitrary subsequence of $\{v_n\}_{n=1}^\infty$, then there exists a further subsequence $\{v_{n_{m_s}}\}_{s=1}^{\infty}$ for which \eqref{Applying_Fatous_Lemma} is satisfied. This means that in fact \eqref{Applying_Fatous_Lemma} is also satisfied for the whole sequence $\{v_n\}_{n=1}^\infty$, and hence $v\mapsto \int_{\Omega}{B(x,v+w) \dd x}$ is s.l.w.s.c.

It is left to see that $J$ is coercive over $H_{\overline{g}}^1(\Omega)$. Let $u_{\overline{g}}\in H^1(\Omega)$ be such that $\gamma_2(u_{\overline{g}})=\overline{g}$ on $\partial\Omega$. For any $v\in H_{\overline{g}}^1(\Omega)$, we have $\gamma_2(v-u_{\overline{g}})=0$. Since $\Omega$ is a bounded Lipschitz domain, it follows that $v-u_{\overline{g}}\in H_0^1(\Omega)$. By applying Poincar{\'e}'s inequality we obtain
\begin{equation} \label{triangle_inequalities_Hg1_space}
\begin{aligned}
\abs{\|v\|_{H^1(\Omega)}-\|u_{\overline{g}}\|_{H^1(\Omega)}}&\leq \|v-u_{\overline{g}}\|_{H^1(\Omega)}\leq \sqrt{1+C_P^2}\|\nabla (v-u_{\overline{g}})\|_{L^2(\Omega)}\\
&\leq \sqrt{1+C_P^2}\left(\|\nabla v\|_{L^2(\Omega)}+\|\nabla u_{\overline{g}}\|_{L^2(\Omega)}\right).
\end{aligned}
\end{equation} 
After squaring both sides of \eqref{triangle_inequalities_Hg1_space} and using the inequality $2ab\leq a^2+b^2$ for $a,b\in\mathbb R$ we obtain the estimate
\begin{equation}\label{lower_bound_H1_seminorm_in_Hg1_space}
\|v\|_{H^1(\Omega)}^2-2\|v\|_{H^1(\Omega)}\|u_{\overline{g}}\|_{H^1(\Omega)}+\|u_{\overline{g}}\|_{H^1(\Omega)}\leq 2\left(1+C_P^2\right)\left(\|\nabla v\|_{L^2(\Omega)}^2+\|\nabla u_{\overline{g}}\|_{L^2(\Omega)}^2\right).
\end{equation} 
Now, coercivity of $J$ follows by recalling that $B(x,t)\ge 0$ for all $x\in\Omega$ and $t\in\mathbb R$ and using \eqref{lower_bound_H1_seminorm_in_Hg1_space}:  
\begin{equation}\label{coercivity_of_J_General_PBE}
\begin{aligned}
J(v)&=\frac{1}{2}a(v,v)+\int_{\Omega}{B(x,v+w)\dd x}-\int_{\Omega}{{\bm f}\cdot \nabla v \dd x}\ge\frac{\epsilon_{\rm{min}}}{2}\|\nabla v\|_{L^2(\Omega)}^2-\|{\bm f}\|_{L^2(\Omega)}\|\nabla v\|_{L^2(\Omega)}\\
&\ge \frac{\epsilon_{\min}}{4\left(1+C_P^2\right)}\left(\|v\|_{H^1(\Omega)}^2-2\|v\|_{H^1(\Omega)}\|u_{\overline{g}}\|_{H^1(\Omega)}+\|u_{\overline{g}}\|_{H^1(\Omega)}\right)\\
&\qquad-\frac{\epsilon_{\min}}{2}\|\nabla u_{\overline{g}}\|_{H^1(\Omega)}^2-\|{\bm f}\|_{L^2(\Omega)}\|v\|_{H^1(\Omega)}\to +\infty \text{ whenever } \|v\|_{H^1(\Omega)}\to \infty,
\end{aligned}
\end{equation} 
where $\epsilon_{\min} = \inf_{x \in \Omega} \epsilon(x) > 0$. We have proved the existence of a minimizer $u_{\rm{min}}$ of $J$ over the set $H_{\overline g}^1(\Omega)$. Moreover, since $a(v,v)$ is a strictly convex functional it follows that $J$ is also strictly convex, and therefore this minimizer is unique.
\end{proof}

Now we show that the minimizer $u_{\rm{min}}$ is a solution to the weak formulation \eqref{general_form_of_the_weak_formulation_for_u_2_and_3_term_splitting_H1_spaces}, which is not immediate since $J$ is not Gateaux differentiable at any element of $H_{\overline g}^1(\Omega) \cap L^\infty(\Omega)$, see Section~\ref{Sections_remarks_on_J} below.

\begin{proposition} The unique minimizer $u_{\rm{min}}$ of $J$ over $H_{\overline g}^1(\Omega)$ satisfies \eqref{general_form_of_the_weak_formulation_for_u_2_and_3_term_splitting_H1_spaces} for $\Wb=C_c^\infty(\Omega)$ and ${\Wb=H_0^1(\Omega)\cap L^\infty(\Omega)}$.\end{proposition} 

\begin{proof}We will use the Lebesgue dominated convergence theorem and the fact that at $u_{\rm{min}}$ it holds that $B(x,u_{\rm{min}}+w)\in L^1(\Omega)$. We have that $J(u_{\rm{min}}+\lambda v)-J(u_{\rm{min}})\ge 0$ for all $v\in H_0^1(\Omega)$ and all $\lambda \ge 0$, i.e.,
\begin{equation*}
\begin{aligned}
&\frac{1}{2}a\left(u_{\rm{min}}+\lambda v,u_{\rm{min}}+\lambda v\right)+\int_{\Omega}{B\left(x,u_{\rm{min}}+\lambda v+w\right)\dd x}-\int_{\Omega}{\bm f\cdot \nabla (u_{\rm{min}}+\lambda v)\dd x}\\
&\quad-\frac{1}{2}a\left(u_{\rm{min}},u_{\rm{min}}\right)-\int_{\Omega}{B\left(x,u_{\rm{min}}+w\right)\dd x}+\int_{\Omega}{\bm f\cdot \nabla u_{\rm{min}} \dd x}\ge 0,
\end{aligned}
\end{equation*}
which, by using the symmetry of $a(\cdot,\cdot)$, is equivalent to 
\begin{equation}\label{varying_J_01}
\begin{aligned}
\lambda a\left(u_{\rm{min}},v\right)+\frac{\lambda}{2}a(v,v)+\int_{\Omega}{\left(B\left(x,u_{\rm{min}}+\lambda v+w\right)-B\left(x,u_{\rm{min}}+w\right)\right) \dd x}-\lambda \int_{\Omega}{\bm f\cdot\nabla v \dd x}\ge 0.
\end{aligned}
\end{equation} 
Dividing both sides of the above inequality by $\lambda>0$ and letting $\lambda\to 0^+$ we obtain
\begin{equation} \label{varying_J_02}
a(u_{\rm{min}},v)+\lim_{\lambda \to 0^+}{\frac{1}{\lambda}\int_{\Omega}{B(x,u_{\rm{min}}+\lambda v+w)-B(x,u_{\rm{min}}+w) \dd x}}-\int_{\Omega}{\bm f\cdot\nabla v \dd x}\ge 0.
\end{equation} 
To compute the limit in the second term of \eqref{varying_J_02}, we will apply the Lebesgue dominated convergence theorem. We have
\begin{equation} \label{pointwise_convergence_of_the_variation_of_the_nonlinear_term}
\begin{aligned}
f_\lambda(x)&:=\frac{1}{\lambda}\Big(B\big(x,u_{\rm{min}}(x)+w(x)+\lambda v(x)\big)-B\big(x,u_{\rm{min}}(x)+w(x)\big)\Big)\\
&\xrightarrow{\lambda\to 0^+}b\big(x,u_{\rm{min}}(x)+w(x)\big)v(x)\quad\text{for a.e.}\quad x\in\Omega
\end{aligned}
\end{equation} 
By the mean value theorem we have
\begin{equation*}
f_\lambda(x)=b\big(x,u_{\rm{min}}+w(x)+\Xi(x)\lambda v(x)\big)\,v(x),\ \text{where}\ \Xi(x)\in (0,1)\ \mathrm{for}\ \mathrm{all}\ x\in\Omega
\end{equation*}
and hence, if $v\in L^\infty(\Omega)$, we can obtain the following bound on $f_\lambda$ whenever $\lambda \le 1$:
\begin{equation}\label{Summable_majorant_for_the_nonlinearity_in_General_PBE}
\begin{aligned}
\abs{f_\lambda(x)}&=\Bigg|-\frac{4\pi e_0^2}{k_BT}v(x)\sum_{j=1}^{N_{ions}}{\overline{M}_j(x)\xi_j\me^{-\xi_j\big(u_{\rm{min}}(x)+w(x)+\Xi(x)\lambda v(x)\big)}}\Bigg|\\
&\leq \max_{j}{\abs{\xi_j}}\|v\|_{L^\infty(\Omega)}\frac{4\pi e_0^2}{k_BT}\sum_{j=1}^{N_{ions}}{\overline M_j(x)\me^{-\xi_j\big(u_{\rm{min}}(x)+w(x)\big)-\xi_j\Xi(x)\lambda v(x)}}\\
&\leq \max_{j}{\abs{\xi_j}}\max_{j}{\me^{\abs{\xi_j} \|v\|_{L^\infty(\Omega)}} }\|v\|_{L^\infty(\Omega)}\frac{4\pi e_0^2}{k_BT}\sum_{j=1}^{N_{ions}}{\overline M_j(x)\me^{-\xi_j\big(u_{\rm{min}}(x)+w(x)\big)}}\\
&=\max_{j}{\abs{\xi_j}}\max_{j}{\me^{\abs{\xi_j} \|v\|_{L^\infty(\Omega)}} }\|v\|_{L^\infty(\Omega)}B\big(x,u_{\rm{min}}(x)+w(x)\big)\in L^1(\Omega).
\end{aligned}
\end{equation} 
From the Lebesgue dominated convergence theorem, by using \eqref{pointwise_convergence_of_the_variation_of_the_nonlinear_term} and \eqref{Summable_majorant_for_the_nonlinearity_in_General_PBE}, it follows that the limit in \eqref{varying_J_02} is equal to $\int_{\Omega}{b(x,u_{\rm{min}}+w)v \dd x}$, and therefore we obtain
\begin{equation} \label{Final_ineuqality_after_varying_the_functional_J}
a(u_{\rm{min}},v)+\int_{\Omega}{b(x,u_{\rm{min}}+w)v \dd x}-\int_{\Omega}{\bm f\cdot\nabla v \dd x}\ge 0 \,  \text{ for all }\, v\in H_0^1(\Omega)\cap L^\infty(\Omega).
\end{equation} 
This means that $u=u_{\rm{min}}$ is the unique solution to the weak formulation \eqref{general_form_of_the_weak_formulation_for_u_2_and_3_term_splitting_H1_spaces} for $\Wb=H_0^1(\Omega)\cap L^\infty(\Omega)$ and $\Wb=C_c^\infty(\Omega)$.
\end{proof}

In fact $u=u_{\rm{min}}$ is also a solution to \eqref{general_form_of_the_weak_formulation_for_u_2_and_3_term_splitting_H1_spaces} with $\Wb=H_0^1(\Omega)$, which we prove in Section \ref{Section_A_priori_Linfty_estimate} with the help of the a priori $L^\infty$ bound obtained there.

\subsubsection{Some remarks on the functional \texorpdfstring{$J$}{J}}\label{Sections_remarks_on_J}
It is worth noting that for $\overline g=0$, the domain $\dom(J)$ of $J$ as defined in \eqref{definition_of_J_general_form_General_PBE} is a linear subspace of $H_0^1(\Omega)$ for $d\leq 2$ and not a linear subspace of $H_0^1(\Omega)$ if $d\ge 3$. In dimension $d\leq 2$, from the Moser-Trudinger inequality \cite{Trudinger_1967,Best_constants_in_some_exponential_Sobolev_inequalities} we know that $\me^v\in L^2(\Omega)$ for any $v\in H_0^1(\Omega)$ and thus $\me^{\lambda v_1+\mu v_2}\in L^2(\Omega)$ for any $\lambda,\mu\in\mathbb R$ and any $v_1,v_2\in H_0^1(\Omega)$.

On the other hand,  if $d\ge 3$, first observe that 
\[\dom(J)=\big\{v\in H_0^1(\Omega)\,:\, B(x,v+w)\in L^1(\Omega)\big\}.\]
For simplicity we consider the case of the PBE, i.e., $B(x,v+w)=\overline{k}^2\cosh(v+w)$. Let us consider an example situation in which $B(0,r) \subset \Omega_{ions} \subset \Omega = B(0,1)$, where $B(0,r)$ denotes the ball in $\mathbb R^d,\,d\ge 3$ with radius $r$ and a center~at~0.
We consider the function $v=\ln{\frac{1}{\abs{x}}}\in H_0^1(B(0,1))$. Since $\me^v=\frac{1}{\abs{x}}\in L^1(\Omega_{ions})$ and $\me^{\lambda v}=\frac{1}{\abs{x}^\lambda}\notin L^1(\Omega_{ions})$ for any $\lambda\ge d$, we obtain
\begin{equation*}
\begin{aligned}
\int_{\Omega}{\overline{k}^2\cosh(v+w)\dd x}&= \int_{\Omega_{ions}}{\overline{k}^2_{ions}\frac{\left(\me^{v+w}+\me^{-v-w}\right)}{2} \dd x}\\
&\leq\frac{1}{2}\overline{k}_{ions}^2\me^{\|w\|_{L^\infty(\Omega_{ions})}}\int_{\Omega_{ions}}{\left(\me^v+\me^{-v}\right)\dd x}\\
&\leq \frac{1}{2}\overline{k}_{ions}^2\me^{\|w\|_{L^\infty(\Omega_{ions})}}\Big(\int_{\Omega_{ions}}{\me^v\dd x}+\abs{\Omega_{ions}}\Big)<+\infty
\end{aligned}
\end{equation*}
but for any $\lambda > d$ we have
\begin{equation*}
\int_{\Omega}{\overline{k}^2\cosh(\lambda v+w)\dd x}\ge\frac{1}{2}\int_{\Omega_{ions}}{\overline{k}^2_{ions}\me^{\lambda v+w}\dd x}\ge \frac{1}{2}\overline{k}^2_{ions}\me^{-\|w\|_{L^\infty(\Omega_{ions})}}\int_{\Omega_{ions}}{\me^{\lambda v}\dd x}=+\infty.
\end{equation*}
This means that $v\in \dom(J)$, but $\lambda v\notin \dom(J)$ for any $\lambda\ge d$. Therefore $\dom(J)$ is not a linear space. However, $\dom(J)\subset H_0^1(\Omega)$ is a convex set. To see this, let $v_1,v_2\in \dom(J)$, i.e., $B(x,v_1+w),\,B(x,v_2+w)\in L^1(\Omega)$. Since $B(x,\cdot)$ is convex it follows that for almost every $x\in\Omega$ and every $\lambda\in [0,1]$ we have
\[B(x,\lambda v_1(x)+(1-\lambda)v_2(x)+w(x))\leq \lambda B(x,v_1(x)+w(x))+(1-\lambda)B(x,v_2(x)+w(x)).\]
By integrating the above inequality over $\Omega$, since both terms of the right hand side are finite, we get $\lambda v_1+(1-\lambda) v_2\in \dom(J)$ for all $\lambda\in[0,1]$. 

Analogously, in dimension $d=3$ the functional $\int_{\Omega}{B(x,v+w)\dd x}$ is not Gateaux differentiable at any $u\in H_0^1(\Omega)\cap L^\infty(\Omega)$. In fact $\int_{\Omega}{B(x,v+w)\dd x}$ is discontinuous at every $u\in H_0^1(\Omega)\cap L^\infty(\Omega)$. To see this, consider any $\Omega_{ions}$ and $\Omega$ such that $B(0,r) \subset \Omega_{ions} \subset \Omega$, and again for simplicity $B(x,v+w)=\overline{k}^2\cosh(v+w)$. We define $z=\psi \abs{x}^{-1/3}$, where $\psi \in C_c^\infty(\Omega)$ is equal to 1 in $B(0,r)$. Then $z\in H_0^1(\Omega)$, but $\me^{\lambda z}\notin L^1(\Omega_{ions})$ for any $\lambda >0$. To see this, notice that for any fixed $\lambda >0$, in a neighborhood of the origin (of size depending on $\lambda$) we have $\me^{\lambda z}> \abs{x}^{-3}$. In this case, for any $u\in H_0^1(\Omega)\cap L^\infty(\Omega)$ and any $\lambda >0$ we have 
\begin{equation*}
\begin{aligned}
\int_{\Omega}{\overline{k}^2\cosh(u+\lambda z+w) \dd x}&\ge \frac{1}{2}\int_{\Omega_{ions}}{\overline{k}_{ions}^2\me^{u+\lambda z+w}\dd x}\\&\ge \frac{\overline{k}_{ions}^2\me^{-\|u+w\|_{L^\infty(\Omega_{ions})}}}{2}\int_{\Omega_{ions}}{\me^{\lambda z}\dd x}=+\infty.
\end{aligned}
\end{equation*}

\subsubsection{A priori \texorpdfstring{$L^\infty$}{Linfty} estimate for the solution of \eqref{general_form_of_the_weak_formulation_for_u_2_and_3_term_splitting_H1_spaces} and existence with test space $H^1_0$}\label{Section_A_priori_Linfty_estimate}
In this section we prove a boundedness result for semilinear problems resembling \eqref{general_form_of_the_weak_formulation_for_u_2_and_3_term_splitting_H1_spaces} but under slightly more general assumptions on the nonlinearity, which is not necessarily assumed to be monotone. While such a boundedness result is not necessary to obtain existence of a solution $u$ to \eqref{general_form_of_the_weak_formulation_for_u_2_and_3_term_splitting_W1p_spaces}, it is important for the PBE for two reasons. The first is that our uniqueness analysis for \eqref{weak_formulation_General_PBE_W1p_spaces} requires $x \mapsto b(x,\phi(x))\in L^1(\Omega)$, which holds for $u \in L^\infty(\Omega)$. The second is that the 2-term and 3-term splittings are often used numerically in practice, where having standard $H^1$ formulations is advantageous. 

Results on a priori $L^\infty$ estimates for linear elliptic equations of second order appear for example in \cite{Stampacchia_1965, Kinderlehrer_Stampacchia} and for nonlinear elliptic equations  in \cite{Boccardo_Murat_Puel_1992, Boccardo_Segura_de_Leon_Trombetti_2001, Trombetti_2003, Boccardo_Brezis_2003, Boccardo_Dall_Aglio_Orsina_1998}. Vital techniques in the analysis of these papers are different adaptations of the $L^\infty$ regularity procedure introduced by Stampacchia; these make use of families of `nonlinear' test functions $G_k(u)$ derived from the solution $u$. We can write \eqref{general_form_of_the_weak_formulation_for_u_2_and_3_term_splitting_H1_spaces} in the generic semilinear form
\begin{equation} \label{generic_semilinear_problem}
A(u)+H(x,u,\nabla u)=0\quad \text {in } \Omega,\,u=\overline{g}\quad\text{on }\partial\Omega.
\end{equation} 
Now, to obtain information from such a testing procedure on \eqref{generic_semilinear_problem} one typically requires either bounds of the form $H(x,t,\xi)\leq C(x) + h(|t|)|\xi|^p$ with well-behaved $h$ that ensure the effect of the nonlinearity can be dominated by the second order term (see \cite{Boccardo_Murat_Puel_1992} for the case $h, C$ constant and $A$ a Leray-Lions differential operator, or \cite{Boccardo_Segura_de_Leon_Trombetti_2001} with $h \in L^1(\mathbb R)$ and $C=0$), a sign condition of the form $H(x,t,\xi)t \geq 0$ (see \cite{Brezis_Strauss_1973} for some cases without gradient terms), or both \cite{Bensoussan_Boccardo_Murat_1988}. Most works in the literature seem to be centered on these assumptions, but in our situation neither is available: the nonlinearity $b(x,\cdot)$ has exponential growth, and when the ionic solution is not charge neutral it also does not follow the sign of its second argument. However the full strength of the sign condition is rarely needed, for example in \cite[Sec.~3]{Brezis_Strauss_1973} it is introduced as a simplification of more detailed conditions involving the actual test functions to be used. Starting by the observation that the sign condition is clearly satisfied if $b(x,t)$ is nondecreasing in $t$ and $b(x,w)=0$, we relax this by bounds of the type $c_1(x,t)\leq b(x,t)\leq c_2(x,t)$ with $c_1, c_2$ nondecreasing in their second argument and with adequate integrability on the functions $c_1(x,w)$, $c_2(x,w)$. These conditions are applicable to the general form of $b$ in the PBE \eqref{GPBE_dimensionless} and ensure that the effect of the nonlinear term when testing with $G_k(u)$ is bounded below by a fixed $L^1$ function, which is just enough to finish the proof. Since the dependence on $\nabla u$ in \eqref{general_form_of_the_weak_formulation_for_u_2_and_3_term_splitting_H1_spaces} consists of a linear term with high summability, it can be taken care of in an ad-hoc manner and does not pose major problems.

In the result presented below, we assume a linear operator $\bm A$ and a nonlinearity $b(x,t)$ which does not depend on the gradient of the solution and which is not assumed to be nondecreasing in the second argument. We allow for a linear gradient term and a nonhomogeneous Dirichlet boundary condition on $\partial\Omega$ given by $g$, covering the case of \eqref{general_form_of_the_weak_formulation_for_u_2_and_3_term_splitting_H1_spaces}. With the assumptions we make on $b$, we prove that every weak solution $u\in H_{g}^1(\Omega)$ must be in $L^\infty(\Omega)$ with  $\|u\|_{L^\infty(\Omega)}\leq \gamma$ where $\gamma$ depends only on the data of the problem. As in \cite{Boccardo_Murat_Puel_1992}, our $L^\infty$ result seems to be optimal in the sense that when $b(x,\cdot)$ is a linear term, $u\in L^\infty(\Omega)$ for $s>d$, $r>\frac{d}{2}$ which coincides with the classical (optimal) results of Stampacchia, De Giorgi, and Moser in the linear case (see, e.g. the references in \cite{Boccardo_Murat_Puel_1992}). 

In \cite{Boccardo_Segura_de_Leon_Trombetti_2001, Trombetti_2003}, the authors prove $L^\infty$ estimates on the solution of very general nonlinear elliptic equations but with a nonlinear zeroth order term with a growth condition which seems not to cover the case of exponential nonlinearities with respect to $u$, as it is the case of the general PBE, and homogeneous Dirichlet boundary conditions. In \cite{Boccardo_Dall_Aglio_Orsina_1998, Boccardo_Brezis_2003}, $L^\infty$ estimates are proved for nonlinear elliptic equations with homogeneous Dirichlet boundary conditions and with degenerate coercivity but without a nonlinear zeroth order term. 

\begin{definition}[see, e.g. Definition 3.5 in \cite{Dacorogna}]\label{Definition_Caratheodory_function}
Let $\Omega\subset \mathbb R^n$ be an open set and let $f:\Omega\times \mathbb R\to\mathbb R\cup \{+\infty\}$. Then $f$ is said to be a Carath\'eodory function if 
\begin{itemize}
\item [(i)] $t\mapsto f(x,t)$ is continuous for almost every $x\in\Omega$,
\item [(ii)] $x\mapsto f(x,t)$ is measurable for every $t\in\mathbb R$.
\end{itemize}
\end{definition}
If $f$ is a Carath\'eodory function and $u:\Omega\to\mathbb R$ is measurable, then it follows that the function $g:\Omega\to \mathbb R\cup \{+\infty\}$ defined by $g(x)=f(x,u(x))$ is measurable (see, e.g., Proposition 3.7 in \cite{Dacorogna}).

\begin{theorem}[A priori $L^\infty$ estimate]\label{Thm_Boundedness_of_solution_to_general_semilinear_elliptic_problem}
Let $\Omega\subset\mathbb R^d,\,d\ge 2$ be a bounded Lipschitz domain and let  
$b(x,t):\Omega\times\mathbb R\to\mathbb R$ be a Carath\'eodory function, not necessarily nondecreasing in its second argument, such that 
\begin{equation} \label{bounds_on_b}
c_1(x,t)\leq b(x,t)\leq c_2(x,t)\quad\text{for a.e } x\in\Omega\ \mathrm{and}\ \mathrm{all}\  t\in\mathbb R,
\end{equation} 
where $c_1, c_2:\Omega\times \mathbb R\to \mathbb R$ are Carath\'eodory functions which are nondecreasing in the second argument for a.e $x\in\Omega$. Let $a(u,v)=\int_{\Omega}{\bm A\nabla u\cdot\nabla v \dd x}$, where ${\bm A=(a_{ij})}$, $a_{ij}(x)\in L^\infty(\Omega)$, and ${\bm A}$ satisfies the uniform ellipticity condition \eqref{uniform_ellipticity_of_A} 
for some positive constant $\underline \alpha$. Finally, let 
\begin{equation} \label{L_infty_estimate_more_general_nonlinear_problem}
\begin{aligned}
&u\in H_0^1(\Omega)\text{ be such that }b(x,u+\omega)v \in L^1(\Omega)\,\text{ for all } v\in \Wb \text{ and } \\
&a(u,v)+\int_{\Omega}{b(x,u+\omega)v \dd x}=\int_{\Omega}{\left(f_0v+\bm f\cdot\nabla v \right)\dd x}\,\text{ for all } v\in \Wb,
\end{aligned}
\end{equation} 
where $\bm f=(f_1,\ldots,f_d)$ and the test space $\Wb$ can be either $C_c^\infty(\Omega)$, $H_0^1(\Omega)\cap L^\infty(\Omega)$, or~$H_0^1(\Omega)$.
Provided that $\omega \in L^\infty(\Omega)$, $\bm f \in \left[L^s(\Omega)\right]^d$ with $s>d$ and $f_0$, $c_1(x,\omega)$, $c_2(x,\omega)\in L^r(\Omega)$ with $r>d/2$, then $\|u\|_{L^\infty(\Omega)}\leq\gamma$ where $\gamma$ depends only on the data, i.e, $\underline \alpha$, $|\Omega|$, $\|a_{ij}\|_{L^\infty(\Omega)}$, $\|\omega\|_{L^\infty(\Omega)}$, $\|\bm f\|_{L^s(\Omega)}$, $\|f_0\|_{L^r(\Omega)}$, $\|c_1(x,\omega)\|_{L^r(\Omega)}$, $\|c_2(x,\omega)\|_{L^r(\Omega)}$.
\end{theorem}

\begin{remark}
We point out that in Theorem~\ref{Thm_Boundedness_of_solution_to_general_semilinear_elliptic_problem} the bilinear form $a(\cdot,\cdot)$, the nonlinearity $b(x,\omega)$ and the functions $f_0$ and $\bm f$ are more general than those for the GPBE in \eqref{general_form_of_the_weak_formulation_for_u_2_and_3_term_splitting_H1_spaces}. Moreover, we use the notation $\omega$ to distinguish it from $w$ appearing in \eqref{general_form_of_the_weak_formulation_for_u_2_and_3_term_splitting_H1_spaces}, since they do not refer to the same object. Namely, below we apply Theorem~\ref{Thm_Boundedness_of_solution_to_general_semilinear_elliptic_problem} to the particular problem \eqref{general_form_of_the_weak_formulation_for_u_2_and_3_term_splitting_H1_spaces} with $\omega=\chi_{\Omega_{ions}} w= \chi_{\Omega_{ions}} G$ in the case of the 2-term splitting and with $\omega=0$ in the case of the 3-term splitting. 
\end{remark}

\begin{remark}
Since $c_1(x,\cdot)$ and $c_2(x,\cdot)$ are nondecreasing it follows that
\begin{equation}
c_i\left(x,-\|\omega\|_{L^\infty(\Omega)}\right)\leq c_i(x,\omega)\leq c_i\left(x,\|\omega\|_{L^\infty(\Omega)}\right)\,\text{ for }\, i=1,2.
\end{equation}
Then the condition that $c_1(x,\omega),\,c_2(x,\omega)\in L^r(\Omega)$ where $r>d/2$ can be achieved if $c_1(x,t)$ and $c_2(x,t)$ define functions in $L^r(\Omega)$ for every $t\in\mathbb R$. For example, this condition will be fulfilled if $c_1(x,t)=k_1(x)a_1(t)$ and $c_2(x,t)=k_2(x)a_2(t)$ where $k_1,k_2\ge 0$, $k_1,k_2\in L^r(\Omega)$ and $a_1,a_2:\mathbb R\to\mathbb R$ are nondecreasing and continuous functions.

If $b(x,\cdot)$ is nondecreasing for almost every $x\in\Omega$ and if $b(x,\omega)\in L^r(\Omega)$ with $r>d/2$, then $c_1(x,\cdot)$ and $c_2(x,\cdot)$ can be taken equal to $b(x,\cdot)$. Notice also that there is neither sign condition nor growth condition on the nonlinearity $b(x,\cdot)$.
\end{remark}

\begin{remark}[Nonhomogeneous Dirichlet boundary condition]\label{nonhomogeneous_Dirichlet_BC}
Let $g$ be in the trace space $W^{1-1/s,s}(\partial\Omega)$ for some $s>d$ (in particular this is true if $g\in C^{0,1}(\partial\Omega)$, see Lemma~\ref{Lemma_extension_pf_Lipschitz}) and let $u_g\in W^{1,s}(\Omega)\subset L^\infty(\Omega)$ be such that  $\gamma_s(u_g)=g =\gamma_2(u_g)$. Suppose that $u$ satisfies
\begin{equation} \label{L_infty_estimate_more_general_nonlinear_problem_nonhomogeneous_Dirichlet_BC}
\begin{aligned}
&u\in H_g^1(\Omega)\text{ such that }b(x,u+\omega)v \in L^1(\Omega)\,\text{ for all } v\in \Wb \text{ and } \\
&a(u,v)+\int_{\Omega}{b(x,u+\omega)v \dd x}=\int_{\Omega}{\left(f_0v+\bm f\cdot\nabla v \right)\dd x}\,\text{ for all } v\in \Wb.
\end{aligned}
\end{equation} 
Then we can apply Theorem~\ref{Thm_Boundedness_of_solution_to_general_semilinear_elliptic_problem} to the homogenized version of problem \eqref{L_infty_estimate_more_general_nonlinear_problem_nonhomogeneous_Dirichlet_BC}, that is
\begin{equation} \label{L_infty_estimate_more_general_nonlinear_problem_nonhomogeneous_Dirichlet_BC_homogenized}
\begin{aligned}
&\text{Find }u_0\in H_0^1(\Omega)\text{ such that }b(x,u_0+u_g+\omega)v \in L^1(\Omega)\,\text{ for all } v\in \Wb \text{ and } \\
&a(u_0,v)+\int_{\Omega}{b(x,u_0+u_g+\omega)v \dd x}=\int_{\Omega}{\left(f_0v+\left(\bm f-\bm A\nabla u_g\right)\cdot\nabla v \right)\dd x}\,\text{ for all } v\in \Wb
\end{aligned}
\end{equation} 
with $\omega\rightarrow u_g+\omega\in L^\infty(\Omega)$, $\bm f\rightarrow \bm f-\bm A\nabla u_g\in \left[L^s(\Omega)\right]^d$.
\end{remark}

\begin{remark}\label{additional_regularity_of_general_semilinear_problem}
From Theorem~\ref{Thm_Boundedness_of_solution_to_general_semilinear_elliptic_problem} and Remark~\ref{nonhomogeneous_Dirichlet_BC} it follows that the nonlinearity evaluated at the solution $u$ of \eqref{L_infty_estimate_more_general_nonlinear_problem_nonhomogeneous_Dirichlet_BC} (if it exists), is in $L^\infty(\Omega)$, i.e., $b(x,u_0+u_g+\omega)\in L^\infty(\Omega)$. Therefore, if a solution exists, then the classical regularity results of De Giorgi-Nash-Moser (see, e.g., Theorem 2.12 in \cite{Bei_Hu_Blow_up_theories_for_semilinear_parabolic_equations_2011}, p. 65, Theorem 3.5 in \cite{YZ_Chen_LC_Wu_Second_order_elliptic_equations_and_systems_1998}) for linear elliptic equations  can be applied to the unique solution~$z_0$ (by the Lax-Milgram Theorem) of the linear equation arising from \eqref{L_infty_estimate_more_general_nonlinear_problem_nonhomogeneous_Dirichlet_BC_homogenized}
\begin{equation} \label{solution_of_nonlinear_problem_satisfies_a_linear_one}
a(z_0,v)=\int_{\Omega}{\left[\left(-b(x,u_0+u_g+\omega)+f_0\right)v+\left(\bm f-\bm A\nabla u_g\right)\cdot\nabla v \right]\dd x}\, \text{ for all }\, v\in H_0^1(\Omega)
\end{equation} 
and conclude that $z_0\equiv u_0$ is H\"older continuous and so is $u=u_g+u_0$, since $u_g\in W^{1,s}(\Omega)\subset C^{0,\lambda}(\overline\Omega)$ for $0<\lambda\leq 1-d/s$.

In addition, if we assume that $\bm A$ satisfies the assumptions in Theorem~\ref{Theorem_uniqueness_of_linear_elliptic_problems_with_diffusion_and_measure_rhs} with $\Omega_m$, $\Gamma=\partial\Omega_m$, and $\Omega$ as defined there, then we can apply Theorem~\ref{Thm_Optimal_Regularity_for_Elliptic_Interface_Problems} and obtain a $p>3$ such that 
$-\nabla\cdot \bm A \nabla$ is a topological isomorphism between $W_0^{1,q}(\Omega)$ and $W^{-1,q}(\Omega)$ for all $q\in (d,p)$ (since $d\in\{2,3\}$). By the assumptions on $f_0$ and $\bm f$ and taking into account the regularity of $u_g$, the right-hand side of \eqref{solution_of_nonlinear_problem_satisfies_a_linear_one} belongs to $W^{-1,q}(\Omega)$ for some $q>d$ depending on $s$ and $r$. We conclude that $z_0\in W_0^{1,\bar q}(\Omega)$ for some $\bar q>d$ (which also implies H\"older continuity of $z_0$ and consequently of $u$).
\end{remark}

\begin{proof}[Proof of Theorem~\ref{Thm_Boundedness_of_solution_to_general_semilinear_elliptic_problem}]
The proof is based on techniques introduced by Stampacchia, see e.g., the proof of Theorem B.2 in \cite{Kinderlehrer_Stampacchia}. There the $L^\infty$ estimate is proved for a linear elliptic problem tested with the space  $V=H_0^1(\Omega)$. Similarly to \cite{Kinderlehrer_Stampacchia}, we construct the following test functions
\begin{equation} \label{the_functions_G_k}
u_k:=G_k(u)=\left\{
\begin{array}{lll}
u-k,\,& \text{ a.e on } \{u(x)>k\},\\
0,\,& \text{ a.e on } \{\abs{u(x)}\leq k\},\\
u+k,\,& \text{ a.e on } \{u(x)<-k\},
\end{array}
\right.
\end{equation} 
\normalsize
for any $k\ge 0$, where $u_0:= u$.  Since $G_k(t)=\text{sign}(t)(\abs{t}-k)^+$ is Lipschitz continuous with $G_k(0)=0$ and $u\in H_0^1(\Omega)$, by Stampacchia's theorem (e.g., see \cite{Kinderlehrer_Stampacchia,Gilbarg_Trudinger_book_2001}) it follows that $G_k(u)\in H_0^1(\Omega)\,\text{ for all }\, k\ge 0$. Moreover, the weak partial derivatives are given by
\begin{equation} \label{weak_partial_derivative_of_G_k}
\frac{\partial u_k}{\partial x_i}=\left\{
\begin{array}{lll}
\frac{\partial u}{\partial x_i},\,& \text{ a.e on } \{u(x)>k\},\\
0,\,& \text{ a.e on } \{\abs{u(x)}\leq k\},\\
\frac{\partial u}{\partial x_i},\,& \text{ a.e on } \{u(x)<-k\}.
\end{array}
\right.
\end{equation} 
We have the Sobolev Embedding $H^1(\Omega)\hookrightarrow L^q(\Omega)$ where $q<\infty$ for $d=2$ and $q=\frac{2d}{d-2}$ for $d>2$. With $q'$ we will denote the H{\"o}lder conjugate to $q$. Thus $q'=\frac{q}{q-1}>1$ for $d=2$, and $q'=\frac{2d}{d+2}$ for $d>2$. With $C_E$ we denote the embedding constant in the inequality $\|u\|_{L^q(\Omega)}\leq C_E\|u\|_{H^1(\Omega)}$, which depends only on the domain $\Omega$, $d$, and $q$.

\textbf{Testing with $u_k$:} By applying Theorem~\ref{Thm_Brezis_Browder_for_the_extension_of_bdd_linear_functional_1978}, we will show that we can test equation \eqref{L_infty_estimate_more_general_nonlinear_problem} with $u_k$ for any $k>0$, as well as with $u$, which is not obvious because $u_k$ and $u$ need not be in the test space $\Wb$. For this observe that 
\begin{equation} \label{the_extension_of_Tb_is_equal_to_the_RHS}
\int_{\Omega}{b(x,u+\omega)v \dd x}=-a(u,v)+\int_{\Omega}{\left(f_0v+\bm f\cdot\nabla v\right)\dd x}\,\text{ for all }\, v\in \Wb
\end{equation} 
and that the right-hand side of \eqref{the_extension_of_Tb_is_equal_to_the_RHS} defines a bounded linear functional over $H_0^1(\Omega)$:
\begin{equation} \label{boundedness_of_a}
\abs{a(u,v)}\leq \Bigg(\sum_{i,j=1}^{d}{\|a_{ij}\|_{L^\infty(\Omega)}}\Bigg)\|u\|_{H^1(\Omega)}\|v\|_{H^1(\Omega)}\,\text{ for all }\, v\in H_0^1(\Omega)
\end{equation} 
 and 
\begin{equation} \label{boundedness_of_right_hand_side}
\begin{aligned}
&\Big|\int_{\Omega}{\left(f_0v+\bm f\cdot \nabla v\right)\dd x}\Big|\leq \|f_0\|_{L^{q'}(\Omega)}\|v\|_{L^q(\Omega)}+\|\bm f\|_{L^2(\Omega)}\|\nabla v\|_{L^2(\Omega)}\\
\leq& C_E\|f_0\|_{L^{q'}(\Omega)}\|v\|_{H^1(\Omega)}+\|\bm f\|_{L^2(\Omega)}\|v\|_{H^1(\Omega)}\,\text{ for all }\, v\in H^1(\Omega).
\end{aligned}
\end{equation} 
From \eqref{the_extension_of_Tb_is_equal_to_the_RHS}, \eqref{boundedness_of_a}, and \eqref{boundedness_of_right_hand_side}, it is clear that the linear functional $T_b$ defined by the formula $\langle T_b,v\rangle=\int_{\Omega}{b(x,u+\omega)v \dd x}\,\text{ for all }\, v\in \Wb$ is bounded in the norm of $H^1(\Omega)$ over the dense subspace\footnote{$\Wb$ is a dense subspace of $H_0^1(\Omega)$ when $\Wb=H_0^1(\Omega)\cap L^\infty(\Omega)$ or $\Wb=C_c^\infty(\Omega)$.} $\Wb$ and therefore it can be uniquely extended by continuity to a functional $\overline T_b\in H^{-1}(\Omega)$ over the whole space $H_0^1(\Omega)$. Moreover, the fact that $\int_{\Omega}{b(x,u+\omega)v \dd x}$ is finite for all $v \in C_c^\infty(\Omega)$ implies that $b(x,u+\omega) \in L^1_{loc}(\Omega)$. Therefore, if we show that $b(x,u+\omega)u_k\ge f_k(x)$ for some function $f_k\in L^1(\Omega)$, taking into account Remark \ref{remark_on_the_theorem_of_Brezis_and_Browder_for_the_extension_of_bdd_linear_functional_1978} we may apply Theorem~\ref{Thm_Brezis_Browder_for_the_extension_of_bdd_linear_functional_1978} and conclude that $b(x,u+\omega)u_k\in L^1(\Omega)$ and that $\langle \overline T_b, u_k\rangle=\int_{\Omega}{b(x,u+\omega)u_k\dd x}$. Since by density the extension $\overline T_b$ is also equal to the right hand side of \eqref{the_extension_of_Tb_is_equal_to_the_RHS}, we will also have that
\begin{equation} \label{the_weak_formulation_tested_with_the_special_test_function}
\begin{aligned}
a(u,u_k)=&-\int_{\Omega}{b(x,u+\omega)u_k \dd x}+\int_{\Omega}{\left(f_0u_k +\bm f\cdot\nabla u_k\right) \dd x} \,\text{ for all }\, k\ge 0.
\end{aligned}
\end{equation} 
By using the definition \eqref{the_functions_G_k} of $u_k$ we can write
\begin{equation*}
b(x,u+\omega)u_k=
\left\{
\begin{array}{lll}
b(x,u+\omega)(u-k),\,& \text{ a.e on } \{u(x)>k\},\\
0,\,& \text{ a.e on } \{\abs{u(x)}\leq k\},\\
b(x,u+\omega)(u+k),\,& \text{ a.e on } \{u(x)<-k\}.
\end{array}
\right.
\end{equation*}
Therefore, on the set $\{u(x)>k\}$ we obtain the estimate
\begin{equation} \label{inequality_for_b_on_the_set_where_u_g_greater_than_k}
b(x,u+\omega)(u-k)\ge c_1(x,u+\omega)(u-k)\ge c_1(x,\omega)(u-k),
\end{equation} 
and on the set $\{u(x)<-k\}$ the estimate
\begin{equation} \label{inequality_for_b_on_the_set_where_u_g_less_than_k}
b(x,u+\omega)(u+k)\ge c_2(x,u+\omega)(u+k)\ge c_2(x,\omega)(u+k).
\end{equation} 
If we define the function $f_k(x)$ through the equality
\begin{equation} 
f_k(x):=
\left\{
\begin{array}{lll}
c_1(x,\omega(x))(u(x)-k),\,& \text{ a.e on } \{u(x)>k\},\\
0,\,& \text{ a.e on } \{\abs{u(x)}\leq k\},\\
c_2(x,\omega(x))(u(x)+k),\,& \text{ a.e on } \{u(x)<-k\},
\end{array}
\right.
\end{equation} 
then $f_k$ will be in $L^1(\Omega)$ if $c_1(x,\omega)(u-k)$ and $c_2(x,\omega)(u+k)\in L^1(\Omega)$, because
\begin{equation*}
\abs{f_k(x)}\leq \abs{c_1(x,\omega(x))(u(x)-k)}+\abs{c_2(x,\omega(x))(u(x)+k)} \text{ a.e }x\in\Omega.
\end{equation*}
To ensure that $c_1(x,\omega)(u-k)\in L^1(\Omega)$ and $c_2(x,\omega)(u+k)\in L^1(\Omega)$, it is enough to require that $c_1(x,\omega)$, $c_2(x,\omega)\in L^{q'}(\Omega)$ which is true by assumption since $r>d/2>q'$. In this case, it follows that $b(x,u+\omega)u_k\in L^1(\Omega)$ for each $k\ge 0$ and by Theorem~\ref{Thm_Brezis_Browder_for_the_extension_of_bdd_linear_functional_1978} \eqref{the_weak_formulation_tested_with_the_special_test_function} holds.

\textbf{Estimation of the terms in \eqref{the_weak_formulation_tested_with_the_special_test_function}:} Now the goal is to show that the measure of the set $A(k)$ becomes zero for all $k\ge k_1>0$, where for $k\ge 0$ the set $A(k)$ is defined by 
\begin{equation*}
A(k):=\{x\in\Omega\,:\,|u(x)|>k\}.
\end{equation*}
This would mean that $\abs{u}\leq k_1$ for almost every $x\in\Omega$. The idea to show this is to obtain an inequality of the form \eqref{Lemma_B1_main_inequality} in Lemma~\ref{Lemma_B1_Kinderlehrer_Stampacchia} for the nonnegative and nonincreasing function $\Theta(k):=\abs{A(k)}$. To obtain such an inequality we estimate from below the term on the left-hand side of \eqref{the_weak_formulation_tested_with_the_special_test_function} and from above the terms on the right-hand side of \eqref{the_weak_formulation_tested_with_the_special_test_function}. 

First, by using \eqref{inequality_for_b_on_the_set_where_u_g_greater_than_k} and \eqref{inequality_for_b_on_the_set_where_u_g_less_than_k} we observe that for all $k\ge 0$ it holds
\begin{equation} \label{intermediate_inequality_01_boundedness_for_General_Semilinear_Problem}
\begin{aligned}
&\int_{\Omega}{b(x,u+\omega)u_k \dd x}=\int_{A(k)}{b(x,u+\omega)u_k \dd x}\\
=&\int_{\{u>k\}}{b(x,u+\omega)(u-k) \dd x}+\int_{\{u<-k\}}{b(x,u+\omega)(u+k) \dd x}\\
\ge & \int_{\{u>k\}}{c_1(x,\omega)(u-k)\dd x}+\int_{\{u<-k\}}{c_2(x,\omega)(u+k)\dd x}=\int_{A(k)}{c(x,\omega)u_k\dd x},
\end{aligned}
\end{equation} 
where the function $c:\Omega\times \mathbb R\to \mathbb R\cup \{+\infty\}$ is defined by
\begin{equation*}
c(x,\omega(x)):=
\left\{
\begin{array}{lll}
c_1(x,\omega(x)),\,& \text{ a.e on } \{u(x)\ge 0\},\\
c_2(x,\omega(x)),\,& \text{ a.e on } \{u(x)<0\}.
\end{array}
\right.
\end{equation*}
Now, we estimate the left-hand side of \eqref{the_weak_formulation_tested_with_the_special_test_function} from below. 
First by using the expression \eqref{weak_partial_derivative_of_G_k} for the weak partial derivatives of $u_k$, then the coercivity of $a(\cdot,\cdot)$, and finally Poincar\'e's inequality, we obtain
\begin{equation} \label{intermediate_inequality_02_boundedness_for_General_Semilinear_Problem}
\begin{aligned}
&a(u,u_k)=\int_{\Omega}{\bm A\nabla u\cdot \nabla u_k \dd x}=a(u_k,u_k)\ge \underline\alpha\|\nabla u_k\|_{L^2(\Omega)}^2\ge  \frac{\underline\alpha}{C_P^2+1} \|u_k\|_{H^1(\Omega)}^2
\end{aligned}
\end{equation} 
By combining \eqref{the_weak_formulation_tested_with_the_special_test_function} with the estimates \eqref{intermediate_inequality_01_boundedness_for_General_Semilinear_Problem} and  \eqref{intermediate_inequality_02_boundedness_for_General_Semilinear_Problem} we obtain the intermediate estimate
\begin{equation} \label{LHS_RHS_in_the_L_infty_estimate_of_the_general_nonlinear_problem}
\begin{aligned}
\frac{\underline\alpha}{C_P^2+1} \|u_k\|_{H^1(\Omega)}^2&\leq \absb{\int_{A(k)}c(x,\omega)u_k \dd x}+\absb{\int_{A(k)}f_0u_k \dd x}+\absb{\int_{A(k)}\bm f\cdot\nabla u_k \dd x}.
\end{aligned}
\end{equation} 
We continue by estimating from above all terms on the right-hand side of \eqref{LHS_RHS_in_the_L_infty_estimate_of_the_general_nonlinear_problem}. By applying H\"older's and Poincar\'e's inequalities we obtain 
\begin{equation} \label{third_term1_in_RHS_in_the_L_infty_estimate_of_the_general_nonlinear_problem}
\begin{aligned}
\absb{\int_{A(k)}f_0u_k \dd x}&\leq \|f_0\|_{L^{q'}(A(k))}\|u_k\|_{L^{q}(\Omega)}\leq C_E \|f_0\|_{L^{q'}(A(k))}\|u_k\|_{H^1(\Omega)}.
\end{aligned}
\end{equation} 
Thus if $f_0\in L^r(\Omega)$ with $r>q'$, again by using H\"older's inequality we obtain 
\begin{equation*}
\|f_0\|_{L^{q'}(A(k))}^{q'}\leq\|f_0\|_{L^r(A(k))}^{q'}\abs{A(k)}^{\frac{r-q'}{r}}.
\end{equation*}
By combining the last estimate with \eqref{third_term1_in_RHS_in_the_L_infty_estimate_of_the_general_nonlinear_problem}, we obtain
\begin{equation} \label{third_term2_in_RHS_in_the_L_infty_estimate_of_the_general_nonlinear_problem}
\absb{\int_{A(k)}{f_0u_k \dd x}}\leq C_E\|f_0\|_{L^r(\Omega)}\abs{A(k)}^{\frac{r-q'}{rq'}} \|u_k\|_{H^1(\Omega)}.
\end{equation} 
Similarly, we estimate ($r>q'$)
\begin{equation} \label{first_term_in_RHS_in_the_L_infty_estimate_of_the_general_nonlinear_problem}
\begin{aligned}
\absb{\int_{A(k)}c(x,\omega)u_k \dd x}&\leq \|c(x,\omega)\|_{L^{q'}(A(k))}\|u_k\|_{L^{q}(\Omega)}\\&\leq C_E\abs{A(k)}^{\frac{r-q'}{rq'}}\|c(x,\omega)\|_{L^r(\Omega)}\|u_k\|_{H^1(\Omega)}.
\end{aligned}
\end{equation} 
We continue with the estimation of the third term in the right-hand side of \eqref{LHS_RHS_in_the_L_infty_estimate_of_the_general_nonlinear_problem}:
\begin{equation} \label{fourth_term1_in_RHS_in_the_L_infty_estimate_of_the_general_nonlinear_problem}
\absb{\int_{A(k)}\bm f\cdot \nabla u_k \dd x}\leq \|\bm{f}\|_{L^2(A(k))}\|u_k\|_{H^1(\Omega)}
\end{equation} 
If $\bm f\in \left[L^s(\Omega)\right]^d$ with $s>2$, by using H\"older's inequality we obtain 
\begin{equation*}
\|\bm{f}\|_{L^2(A(k))}^2=\int_{A(k)}{\underbrace{\abs{\bm{f}}^2}_{\in L^{\frac{s}{2}}(\Omega)}1\dd x}\leq \Big(\int_{A(k)}{\abs{\bm{f}}^s \dd x}\Big)^{\frac{2}{s}}\Big(\int_{A(k)}{1 \dd x}\Big)^{\frac{s-2}{s}}=\|\bm{f}\|_{L^s(A(k))}^2|A(k)|^{\frac{s-2}{s}},
\end{equation*}
and hence by combining with \eqref{fourth_term1_in_RHS_in_the_L_infty_estimate_of_the_general_nonlinear_problem}, we arrive at the estimate
\begin{equation} \label{fourth_term2_in_RHS_in_the_L_infty_estimate_of_the_general_nonlinear_problem}
\absb{\int_{A(k)}{\bm f\cdot \nabla u_k \dd x}}\leq \|\bm{f}\|_{L^s(\Omega)}\abs{A(k)}^{\frac{s-2}{2s}}\|u_k\|_{H^1(\Omega)}.
\end{equation} 
Combining \eqref{LHS_RHS_in_the_L_infty_estimate_of_the_general_nonlinear_problem} with the estimates \eqref{third_term2_in_RHS_in_the_L_infty_estimate_of_the_general_nonlinear_problem}, \eqref{first_term_in_RHS_in_the_L_infty_estimate_of_the_general_nonlinear_problem},  \eqref{fourth_term2_in_RHS_in_the_L_infty_estimate_of_the_general_nonlinear_problem}  for the right-hand side terms in \eqref{LHS_RHS_in_the_L_infty_estimate_of_the_general_nonlinear_problem}, and then dividing by $\|u_k\|_{H^1(\Omega)}$, we obtain 
\begin{equation} \label{LHS_RHS2_in_the_L_infty_estimate_of_the_general_nonlinear_problem}
\begin{split}
&\frac{\underline\alpha}{C_P^2+1} \|u_k\|_{H^1(\Omega)}\\
&\qquad \leq  C_E\|c(x,\omega)\|_{L^r(\Omega)}\abs{A(k)}^{\frac{r-q'}{rq'}}+ C_E\|f_0\|_{L^r(\Omega)}\abs{A(k)}^{\frac{r-q'}{rq'}}+\|\bm{f}\|_{L^s(\Omega)}\abs{A(k)}^{\frac{s-2}{2s}}.
\end{split}
\end{equation} 
Now, it is left to estimate the left-hand side of \eqref{LHS_RHS2_in_the_L_infty_estimate_of_the_general_nonlinear_problem} from below in terms of the measure of the set $A(h)$ for $h>k$. We use again the Sobolev embedding theorem and the fact that $A(k)\supset A(h)$ for all $h>k\ge 0$:
\begin{equation} \label{intermediate_inequality_03_boundedness_for_General_Semilinear_Problem}
\begin{aligned}
&\|u_k\|_{H^1(\Omega)}\ge \frac{1}{C_E}\|u_k\|_{L^{q}(\Omega)}=\frac{1}{C_E}\Big(\int_{\Omega}{\abs{u_k}^{q}\dd x}\Big)^{\frac{1}{q}}=\frac{1}{C_E}\Big(\int_{A(k)}{|\underbrace{\abs{u}-k}_{>0}|^q \dd x}\Big)^{\frac{1}{q}}\\
&= \frac{1}{C_E}\Big(\int_{A(k)\setminus A(h)}{\left(\abs{u}-k\right)^q \dd x}+\int_{A(h)}{\left(\abs{u}-k\right)^q \dd x}\Big)^{\frac{1}{q}}\\
&\ge \frac{1}{C_E}\Big(\int_{A(h)}{(h-k)^q \dd x}\Big)^{\frac{1}{q}}=\frac{1}{C_E}(h-k)\abs{A(h)}^{\frac{1}{q}}.
\end{aligned}
\end{equation} 
From \eqref{LHS_RHS2_in_the_L_infty_estimate_of_the_general_nonlinear_problem}  and \eqref{intermediate_inequality_03_boundedness_for_General_Semilinear_Problem} it follows that
\begin{equation} \label{LHS_RHS3_in_the_L_infty_estimate_of_the_general_nonlinear_problem}
\begin{aligned}
&(h-k)\abs{A(h)}^{\frac{1}{q}}\\
&\quad\leq \frac{C_E(C_P^2+1)}{\underline\alpha} \Big[C_E\|c(x,\omega)\|_{L^r(\Omega)}\abs{A(k)}^{\frac{r-q'}{rq'}}+C_E\|f_0\|_{L^r(\Omega)}\abs{A(k)}^{\frac{r-q'}{rq'}}+\|\bm{f}\|_{L^s(\Omega)}\abs{A(k)}^{\frac{s-2}{2s}}\Big]\\
&\quad\leq C_M\Big(\abs{A(k)}^{\frac{s-2}{2s}}+\abs{A(k)}^{\frac{r-q'}{rq'}}\Big),
\end{aligned}
\end{equation} 
where 
\small
\begin{equation*}
C_M:=\frac{C_E(C_P^2+1)}{\underline\alpha}\max{\left\{C_E\left(\|c(x,\omega)\|_{L^r(\Omega)}+\|f_0\|_{L^r(\Omega)}\right),\|\bm{f}\|_{L^s(\Omega)}\right\}}.
\end{equation*}
\normalsize
We have obtained the following inequality for the measure of $A(k)$:
\begin{equation} \label{final_inequaliy1_for_Dk}
(h-k)\abs{A(h)}^{\frac{1}{q}}\leq C_M\Big(\abs{A(k)}^{\frac{s-2}{2s}}+\abs{A(k)}^{\frac{r-q'}{rq'}}\Big)\,\text{ for all }\, h>k\ge 0.
\end{equation} 
Since $u$ is summable it follows that $\abs{A(k)}={\text{meas}\left(\{x\in\Omega: \abs{u(x)}>k\}\right)\to 0}$ monotonically decreasingly as $k\to \infty$. For this reason, there exists a $k_0>0$ such that ${\abs{A(k)}\leq 1\,\text{ for all }\, k\ge k_0}$ (if $\abs{\Omega}\leq 1$, this is satisfied for all $k\ge 0$). Therefore \eqref{final_inequaliy1_for_Dk} takes the form
\begin{equation*}
(h-k)\abs{A(h)}^{\frac{1}{q}}\leq 2 C_M\abs{A(k)}^{\min{\{\frac{s-2}{2s},\frac{r-q'}{rq'}\}}}\,\text{ for all } \, h>k\ge k_0,
\end{equation*}
which is equivalent to the inequality
\begin{equation} \label{final_inequaliy2_for_Dk_d_ge3}
\abs{A(h)}\leq (2 C_M)^q \frac{\abs{A(k)}^{\min{\{\frac{s-2}{2s},\frac{r-q'}{rq'}\}}q}}{(h-k)^q}\,\text{ for all }\, h>k\ge k_0.
\end{equation} 
However, we want to find a $k_0$ which depends only on the data of the problem. For this, observe that from \eqref{LHS_RHS_in_the_L_infty_estimate_of_the_general_nonlinear_problem} for $k=0$, using H{\"o}lder's inequality and the embedding $H^1(\Omega)\hookrightarrow L^q(\Omega)$, we have
\begin{equation} \label{upper_estimate_of_the_squared_H1_norm_of_u_g}
\frac{\underline\alpha}{C_P^2+1} \|u\|_{H^1(\Omega)}^2\leq C_E\|c(x,\omega)\|_{L^{q'}(\Omega)}\|u\|_{H^1(\Omega)}+ C_E\|f_0\|_{L^{q'}(\Omega)}\|u\|_{H^1(\Omega)}+\|\bm{f}\|_{L^2(\Omega)}\|u\|_{H^1(\Omega)}.
\end{equation} 
By dividing both sides of \eqref{upper_estimate_of_the_squared_H1_norm_of_u_g} by $\|u\|_{H^1(\Omega)}$, for arbitrary $k\ge 0$, we obtain
\begin{equation} \label{upper_estimate_for_the_square_root_of_the_measure_of_A_k}
k\abs{A(k)}^{\frac{1}{2}}\leq\Big(\int_{\Omega}{\abs{u}^2 \dd x}\Big)^{\frac{1}{2}}
\leq \frac{C_P^2+1}{\underline\alpha}\left(C_E\|c(x,\omega)\|_{L^{q'}(\Omega)}+C_E\|f_0\|_{L^{q'}(\Omega)}+\|\bm{f}\|_{L^2(\Omega)}\right).
\end{equation} 
If we denote by $C_D$ the constant on the right hand side of inequality \eqref{upper_estimate_for_the_square_root_of_the_measure_of_A_k}, which depends only on the data of the problem \eqref{L_infty_estimate_more_general_nonlinear_problem}, then a sufficient condition for $\abs{A(k)}\leq 1$ will be
\begin{equation*}
\frac{C_D^2}{k^2}\leq 1,
\end{equation*}
which is equivalent to $k\ge C_D=:k_0$. Here we recall that for $d=2$, $q'$ can be any number greater than 1 and for $d>2$ we have $q=\frac{2d}{d-2}$. Since we have required $r>q'$, the constant $C_D$ is well defined.
In order to apply Lemma \ref{Lemma_B1_Kinderlehrer_Stampacchia} to the nonnegative and nonincreasing function $\Theta(k)=\abs{A(k)}$ we need to ensure that
\begin{equation*}
\min{\left\{\frac{s-2}{2s},\frac{r-q'}{rq'}\right\}}>\frac{1}{q},
\end{equation*}
which is equivalent to 
\begin{equation} \label{the_two_inequalities_for_the_min_in_L_infty_general_theorem_1}
\frac{s-2}{2s}>\frac{1}{q}\quad\text{ and }\quad \frac{r-q'}{rq'}>\frac{1}{q}.
\end{equation} 
The first inequality in \eqref{the_two_inequalities_for_the_min_in_L_infty_general_theorem_1} is equivalent to $s>\frac{2q}{q-2}$ and the second to $r>\frac{q}{q-2}$. We also recall that in the course of the proof we have required that $s>2$.
\begin{itemize}
\item For $d=2$, we have $H^1(\Omega)\hookrightarrow L^q(\Omega)$ for any $q<\infty$. In this case the requirements on $s$ and $r$ become $s>2,\, r>1$.
\item For $d\ge3$, we have $H^1(\Omega)\hookrightarrow L^q(\Omega)$ where $q=\frac{2d}{d-2}$ and $q'=\frac{2d}{d+2}$. In this case the requirements on $s$ and $r$ are $s>d,\,r>\frac{d}{2}$.
\end{itemize}
We can summarize the conditions on $s$ and $r$ for $d\ge 2$ as $s>d$ and $r>\frac{d}{2}$. Now, if we denote $\beta:=\min{\{\frac{s-2}{2s},\frac{r-q'}{rq'}\}}q$ from Lemma~\ref{Lemma_B1_Kinderlehrer_Stampacchia} it follows that there exists a constant~$e$, defined by $e^q:=(2C_M)^q \abs{A(k_0)}^{\beta-1}2^{\frac{q\beta}{\beta-1}}$ such that $\abs{A(k_0+e)}=0$. Since $\abs{A(k_0)}\leq \abs{\Omega}$, we can write $\abs{A(k_1)}=0$, where $k_1:=k_0+\big((2C_M)^q \abs{\Omega}^{\beta-1}2^{\frac{q\beta}{\beta-1}}\big)^{\frac{1}{q}}=C_D+(2C_M) \abs{\Omega}^{\frac{\beta-1}{q}}2^{\frac{\beta}{\beta-1}}$. Thus, we have proved that $\|u\|_{L^\infty(\Omega)}\leq k_1$.
\end{proof}

\begin{theorem}\label{Thm_existence_and_uniqueness_and_Boundedness_of_minimizer_of_J_General_PBE}
The unique minimizer $u_{\rm{min}}\in H_{\overline{g}}^1(\Omega)$ of the variational problem \eqref{Variational_problem_for_J} provided by Theorem~\ref{Thm_existence_and_uniqueness_of_minimizer_of_J_General_PBE} coincides with the unique solution of problem \eqref{general_form_of_the_weak_formulation_for_u_2_and_3_term_splitting_H1_spaces} for the test space $H_0^1(\Omega)$.
\end{theorem}
\begin{proof}
We already showed that $u_{\rm{min}}$ equals the unique solution $u$ of \eqref{general_form_of_the_weak_formulation_for_u_2_and_3_term_splitting_H1_spaces} with $\Wb=H_0^1(\Omega)\cap L^\infty(\Omega)$. If we were able to show that $u\in L^\infty(\Omega)$, it would follow that $b(x,u+w)\in L^\infty(\Omega)$ and therefore by a standard density argument we obtain that $u$ is also the unique solution of \eqref{general_form_of_the_weak_formulation_for_u_2_and_3_term_splitting_H1_spaces} with $\Wb=H_0^1(\Omega)$. 

We would like to use the $L^\infty$ estimate of Theorem~\ref{Thm_Boundedness_of_solution_to_general_semilinear_elliptic_problem} through the modification for nonhomogeneous boundary conditions given in Remark~\ref{nonhomogeneous_Dirichlet_BC}, to the weak formulation
\begin{equation}\label{copy_RCH1}
\begin{aligned}
&\text{Find }u\in H_{\overline g}^1(\Omega) \text{ such that } b(x,u+w)v\in L^1(\Omega) \,\text{ for all } v\in \Wb \text{ and }\\
&\int_{\Omega}{\epsilon\nabla u\cdot\nabla v \dd x}+\int_{\Omega}{b(x,u+w)v \dd x}=\int_{\Omega}{{\bm f}\cdot\nabla v \dd x}\, \text{ for all } v\in \Wb,
\end{aligned}\tag{\ref{general_form_of_the_weak_formulation_for_u_2_and_3_term_splitting_H1_spaces}}
\end{equation} \noeqref{copy_RCH1}
with the choices of $\bm f$ and $\bar{g}$ corresponding to the 2-term or 3-term splitting, that is respectively
\begin{alignat}{12}
w&=G, \quad {\bm f}&=&{\bm f}_{\Gb_2}:=\chi_{\Omega_s}(\epsilon_m-\epsilon_s)\nabla G,\quad \quad &\text{ and }\quad &\overline{g}&=&g_\Omega-G &\text{ on }&\partial\Omega, \label{copy_4_2} \tag{\ref{w_barg_f_expressions_2-term_splitting}}\\
w&=0, \quad {\bm f}&=&{\bm f}_{\Gb_3}:=-\chi_{\Omega_m}\epsilon_m\nabla u^H+\chi_{\Omega_s}\epsilon_m\nabla G,\quad &\text{ and }\quad &\overline{g}&=&g_\Omega &\text{ on }&\partial\Omega.\label{copy_4_3} \tag{\ref{w_barg_f_expressions_3-term_splitting}}
\end{alignat}\noeqref{copy_4_2}\noeqref{copy_4_3}
For this, notice that on the one hand ${\bm f}_{\Gb_2} \in \left[L^s(\Omega)\right]^d$ for all $s>d$ since $\chi_{\Omega_s}\nabla G \in \left[L^\infty(\Omega)\right]^d$, $\epsilon_s \in C^{0,1}(\overline{\Omega_s})$, and $\epsilon_m$ is constant. For this case, moreover $\omega = \chi_{\Omega_{ions}} G \in L^\infty(\Omega)$. On the other hand, we also have ${\bm f}_{\Gb_3}  \in \left[L^s(\Omega)\right]^d $ for some $s>d$, since $\nabla u^H$ belongs to this space by Proposition~\ref{proposition_u^H_is_in_W1s_s_ge_d} taking into account that\footnote{Note that the assumption $\Gamma\in C^1$ is not needed to show the $L^\infty$ estimate on the regular component $u$ for the 2-term splitting.} $\Gamma \in C^1(\partial \Omega)$. Moreover, since in both cases $\overline{g} \in C^{0,1}(\partial \Omega)$ we have that its extension $u_{\overline{g}}$ from Lemma~\ref{Lemma_extension_pf_Lipschitz} belongs to $W^{1, \infty}(\Omega)$, in particular $u_{\overline{g}} \in L^\infty(\Omega)$ and $\epsilon \nabla u_{\overline{g}} \in \left[L^s(\Omega)\right]^d$ for all $s>d$.
\end{proof}

As a consequence of the previous theorem and the discussion after \eqref{what_is_Wb} we have obtained the following existence theorem for the General Poisson-Boltzmann equation \eqref{GPBE_dimensionless}.
\begin{theorem}\label{Theorem_Existence_for_full_potential_GPBE}
 There exists a weak solution $\phi$ of equation \eqref{GPBE_dimensionless} satisfying \eqref{weak_formulation_General_PBE_W1p_spaces}. A~particular $\phi$ satisfying \eqref{weak_formulation_General_PBE_W1p_spaces} can be given either in the form  $\phi=G+u$ or in the form $\phi=G+u^H+u$, where $u\in H_{\overline g}^1(\Omega)\cap L^\infty(\Omega)$ 
is the unique solution of \eqref{general_form_of_the_weak_formulation_for_u_2_and_3_term_splitting_H1_spaces} with either $\Wb=H_0^1(\Omega)$, $\Wb=H_0^1(\Omega)\cap L^\infty(\Omega)$ or $\Wb=C_c^\infty(\Omega)$, and $\overline g, w, {\bm f}$ are defined by \eqref{w_barg_f_expressions_2-term_splitting} and \eqref{w_barg_f_expressions_3-term_splitting} for the 2- and 3-term splitting, respectively.
\end{theorem}

\begin{remark}\label{Remark_after_Thm_Existence_for_full_potential_GPBE}
Even if $u$ is the unique solution of \eqref{general_form_of_the_weak_formulation_for_u_2_and_3_term_splitting_H1_spaces}, it might not be the unique solution of \eqref{general_form_of_the_weak_formulation_for_u_2_and_3_term_splitting_W1p_spaces} where the space of test functions, also used in \eqref{weak_formulation_General_PBE_W1p_spaces}, is smaller. To try to close this gap, in Theorem~\ref{Theorem_uniqueness_for_GPBE} we show that if the interface $\Gamma$ is $C^1$ then solutions $\phi$ of \eqref{weak_formulation_General_PBE_W1p_spaces} such that $b(x,\phi)$ is integrable\footnote{Notice that the condition $b(x,\phi)\in L^1(\Omega)$ is slightly more restrictive than the condition $b(x,\phi)v\in L^1(\Omega)$ for all test functions in $\Nb$.} are unique. Notice that the solutions provided by Theorem \ref{Theorem_Existence_for_full_potential_GPBE} satisfy this condition, since $b(x,\cdot)$ vanishes for all $x \in \Omega_m$, the Coulomb potential $G$ is by definition bounded in $\Omega \setminus \Omega_m$ and we just proved in Theorem \ref {Thm_Boundedness_of_solution_to_general_semilinear_elliptic_problem} that $u \in L^\infty(\Omega)$ as well. In case the 3-term splitting is used, $u^H$ is also bounded by Proposition \ref{proposition_u^H_is_in_W1s_s_ge_d}.
\end{remark}
	
\subsection{Uniqueness of the full potential \texorpdfstring{$\phi$}{phi}}\label{Section_uniqueness_for_GPBE}
The proof of uniqueness is based on the following two well-known results for the duality solution framework, which we are able to adapt to weak solutions in the sense of \eqref{weak_formulation_General_PBE_W1p_spaces} with just minor modifications.

\begin{lemma}[analogous to Lemma B.1 from \cite{Brezis_Marcus_Ponce_Nonlinear_Elliptic_With_Measures_Revisited}]\label{Lemma_B1}
Let $\Omega$, $\Omega_m$ and $\bm A$ be as in Theorem \ref{Theorem_uniqueness_of_linear_elliptic_problems_with_diffusion_and_measure_rhs}. Let $p:\mathbb R\to\mathbb R$, $p(0)=0$, be a nondecreasing, bounded and Lipschitz continuous function. Given $f\in L^1(\Omega)$, let $\varphi\in \Mb_0 = \bigcap_{p<\frac{d}{d-1}}{W_0^{1,p}(\Omega)}$ be the unique solution provided by Theorem~\ref{Theorem_uniqueness_of_linear_elliptic_problems_with_diffusion_and_measure_rhs} of 
\begin{equation} \label{linear_problem_with_L1_rhs}
\int_{\Omega}{\bm A\nabla \varphi\cdot\nabla v \dd x}=\int_{\Omega}{fv \dd x}\quad\text{ for all } v\in \Nb = \bigcup_{q>d} W_0^{1,q}(\Omega).
\end{equation} 
Then
\begin{equation} \label{int_fp(u)_nonnegative}
\int_{\Omega}{fp(\varphi)\dd x}\ge 0.
\end{equation} 
\end{lemma}
\begin{proof}
Let $\{f_n\}\subset L^\infty(\Omega)$ be a sequence such that $f_n\to f$ in $L^1(\Omega)$ with $\|f_n\|_{L^1(\Omega)}\leq \|f\|_{L^1(\Omega)}$ for all $n\ge 1$ ($f_n$ can be chosen in $C_c^\infty(\Omega)$ by mollification - see, e.g., Corollary 4.23 in \cite{Brezis_FA}). Then, we know that there is a unique $\varphi_n\in H_0^1(\Omega)$ which satisfies the problem\footnote{Note that from Theorem~\ref{Thm_Optimal_Regularity_for_Elliptic_Interface_Problems} it follows that there is some $\bar q>d$ such that  $\varphi_n\in W_0^{1,\bar q}(\Omega)$ for all $n\ge 1$. Therefore, \eqref{linear_problem_with_Linfty_rhs} also holds for all $v\in W_0^{1,\bar q'}(\Omega)$ with $\bar q'=\bar q/(\bar q-1)<d/(d-1)$.}
\begin{equation} \label{linear_problem_with_Linfty_rhs}
\int_{\Omega}{\bm A\nabla \varphi_n\cdot\nabla v \dd x}=\int_{\Omega}{f_nv \dd x}\quad\text{ for all } v\in  H_0^1(\Omega).
\end{equation} 
Since $p\in C^{0,1}(\mathbb R)$, $p(0)=0$, and $\varphi_n\in H_0^1(\Omega)$ by Stampacchia's superposition theorem it follows that $p(\varphi_n)\in H_0^1(\Omega)$  and we can test \eqref{linear_problem_with_Linfty_rhs} with it. Thus,
\begin{equation} \label{int_fp(un)_nonnegative}
\int_{\Omega}{f_n p(\varphi_n) \dd x}=\int_{\Omega}{p'(\varphi_n) \bm A \nabla \varphi_n\cdot \nabla \varphi_n \dd x}\ge 0.
\end{equation} 
Now, our goal is to pass to the limit in \eqref{int_fp(un)_nonnegative}. From Theorem 4.9 in \cite{Brezis_FA} it follows that there exists some $h\in L^1(\Omega)$ and a subsequence (not renamed) for which $f_n(x)\to f(x)$ a.e. and $\abs{f_n(x)}\leq h(x)$ a.e.. Also, from the proof of Theorem 1 in \cite{Boccardo_Gallouet_1989} (in particular equation (20) there) we know that $\varphi_n\rightharpoonup \varphi$ weakly in $W^{1,p}(\Omega)$ for every $p<d/(d-1)$. Thus, up to another subsequence (again not relabeled) one has $\varphi_n\to \varphi$ strongly in $L^p(\Omega)$ and hence also pointwise almost everywhere in~$\Omega$. With this in mind we obtain
\begin{equation} \label{passing_to_the_limit_fnp(un)}
\begin{aligned}
\absb{\int_{\Omega}{f_np(\varphi_n) \dd x}-\int_{\Omega}{fp(\varphi) \dd x}}&\leq \int_{\Omega}{\abs{f_n}\abs{p(\varphi_n)-p(\varphi)}\dd x} + \int_{\Omega}{\abs{f_n-f}\abs{p(\varphi)} \dd x}.
\end{aligned}
\end{equation} 
The first term in \eqref{passing_to_the_limit_fnp(un)} converges to zero by the Lebesgue dominated convergence theorem since we have pointwise convergence of the integrand and also $\abs{f_n}\abs{p(\varphi_n)-p(\varphi)}\leq 2h M\in L^1(\Omega)$, where $M:=\max_{t\in\mathbb R}\abs{p(t)}$. The second term in \eqref{passing_to_the_limit_fnp(un)} converges to zero because $p(\varphi)\in L^\infty(\Omega)$ and $f_n\to f$ in $L^1(\Omega)$.
\end{proof}

We define the function ${\rm sgn}:\mathbb R\to\mathbb R$ by ${\rm sgn}(t)=1$ if $t>0$, ${\rm sgn}(t)=-1$ if $t<0$ and ${\rm sgn}(t)=0$ if $t=0$. By $\mu^+$, $\mu^-\in \Mp(\Omega)$ we denote the positive and negative parts of $\mu$, obtained by the Jordan decomposition (see Theorem B.71 in \cite{Leoni_2017}), and such that $\mu=\mu^+-\mu^-$.

\begin{proposition}[analogous to Proposition B.3 from \cite{Brezis_Marcus_Ponce_Nonlinear_Elliptic_With_Measures_Revisited}]\label{Proposition_B3}
Let $\Omega$, $\Omega_m$ and $\bm A$ be as in Theorem \ref{Theorem_uniqueness_of_linear_elliptic_problems_with_diffusion_and_measure_rhs}. and let $f\in L^1(\Omega)$, $\mu\in \Mp(\Omega)$. Let $z\in \bigcap_{p<\frac{d}{d-1}}W_0^{1,p}(\Omega)$ be the unique solution of 
\begin{equation} \label{linear_problem_with_L1_term_and_measure_rhs}
\int_{\Omega}{\bm A\nabla z\cdot\nabla v \dd x} +\int_{\Omega}{fv\dd x}=\int_{\Omega}{v \dd\mu}\quad \text{ for all } v\in \bigcup_{q>d}{W_0^{1,q}(\Omega)}.
\end{equation} 
Then,
\begin{equation} \label{PropB3_auxiliary_inequalities}
\int_{\left[z>0\right]}{f \dd x}\leq \|\mu^+\|_{\Mp(\Omega)}\quad \text{ and }\quad  -\int_{\left[z<0\right]}{f \dd x}\leq \|\mu^-\|_{\Mp(\Omega)},
\end{equation} 
and therefore 
\begin{equation} \label{PropB3_main_result}
\int_{\Omega}{f\,{\rm sgn}(z) \dd x}\leq \|\mu\|_{\Mp(\Omega)}.
\end{equation} 
\end{proposition}
\begin{proof}
Since problem \eqref{linear_problem_with_L1_term_and_measure_rhs} is linear, it suffices to prove only the first inequality in \eqref{PropB3_auxiliary_inequalities}.  Let $\left\{\mu_n^+\right\}$, $\left\{\mu_n^-\right\}$ be sequences in $L^\infty(\Omega)$ such that $\mu_n^+(x),\,\mu_n^-(x)\ge 0$ a.e., $\mu_n^+\stackrel{\ast}{\rightharpoonup}\mu^+$, $\mu_n^-\stackrel{\ast}{\rightharpoonup}\mu^-$, and $\|\mu_n^+\|_{L^1(\Omega)}\leq \|\mu^+\|_{\Mp(\Omega)}$, $\|\mu_n^-\|_{L^1(\Omega)}\leq \|\mu^-\|_{\Mp(\Omega)}$ ($\mu_n^+$ and $\mu_n^-$ can even be chosen in $C_c^\infty(\Omega)$, see e.g. Problem 24 in \cite{Brezis_FA}). Let $z_n$ denote the solution, unique by Theorem~\ref{Theorem_uniqueness_of_linear_elliptic_problems_with_diffusion_and_measure_rhs}, of \eqref{linear_problem_with_L1_term_and_measure_rhs} with $\mu$ replaced by $\mu_n:=\mu_n^+-\mu_n^-$, i.e., $z_n\in \bigcap_{p<\frac{d}{d-1}}W_0^{1,p}(\Omega)$ satisfies
\begin{equation} \label{linear_problem_with_L1_term_and_regularized_rhs}
\int_{\Omega}{\bm A\nabla z_n\cdot\nabla v \dd x} +\int_{\Omega}{fv\dd x}=\int_{\Omega}{\mu_n v\dd x}\quad \text{ for all } v\in \bigcup_{q>d}{W_0^{1,q}(\Omega)}.
\end{equation} 
 If $p:\mathbb R\to\mathbb R$ is a nondecreasing bounded Lipschitz continuous function satisfying $p(0)=0$, then by Lemma~\ref{Lemma_B1} we have
\begin{equation} 
\int_{\Omega}{(\mu_n-f)p(z_n)\dd x}\ge 0.
\end{equation} 
If we further assume that $0\leq p(t)\leq 1$ for all $t\in\mathbb R$, then by using the facts that $\mu_n^+(x),\,\mu_n^-(x)\ge 0$ a.e. and $\|\mu_n^+\|_{L^1(\Omega)}\leq \|\mu^+\|_{\Mp(\Omega)}$ we obtain
\begin{equation} \label{intermediate_result_01_in_PropB3}
\begin{aligned}
\int_{\Omega}{fp(z_n)\dd x}&\leq\int_{\Omega}{\mu_np(z_n)\dd x}=\int_{\Omega}{\mu_n^+p(z_n)\dd x}-\int_{\Omega}{\mu_n^-p(z_n)\dd x}\\
&\leq \int_{\Omega}{\mu_n^+p(z_n)\dd x}\leq \|\mu^+\|_{\Mp(\Omega)}.
\end{aligned}
\end{equation} 
Our goal now is to pass to the limit in \eqref{intermediate_result_01_in_PropB3}. First, observe that for all $ v\in C_0(\overline\Omega)$ we have
\begin{equation*}
\int_{\Omega}{\big(\mu_n^+-\mu_n^-\big)\,v\dd x}=\int_{\Omega}{\mu_n^+v\dd x}-\int_{\Omega}{\mu_n^-v\dd x}\to \langle\mu^+,v\rangle-\langle \mu^-,v\rangle=\langle\mu,v\rangle
\end{equation*}
and 
\begin{equation*}
\|\mu_n\|_{L^1(\Omega)}\leq \|\mu_n^+\|_{L^1(\Omega)}+\|\mu_n^-\|_{L^1(\Omega)}\leq \|\mu^+\|_{\Mp(\Omega)}+\|\mu^-\|_{\Mp(\Omega)}=\|\mu\|_{\Mp(\Omega)}.
\end{equation*}
Therefore, as in the proof of the previous lemma and up to another subsequence again denoted $\{z_n\}$, we have $z_n\to z$ pointwise almost everywhere in $\Omega$. Since we also have $\abs{fp(z_n)}\leq \abs{f}\in L^1(\Omega)$, by dominated convergence, from \eqref{intermediate_result_01_in_PropB3} we obtain 
\begin{equation} \label{intermediate_result_02_in_PropB3}
\int_{\Omega}{fp(z)\dd x}\leq \|\mu^+\|_{\Mp(\Omega)}.
\end{equation} 
Now, we apply \eqref{intermediate_result_02_in_PropB3} to a sequence of nondecreasing Lipschitz continuous functions $\{p_n\}$ such that $p_n(s)=0$ for $s\leq 0$ and $p_n(s)=1$ for $s\ge\frac{1}{n}$. As $n\to\infty$, again by dominated convergence we obtain the first inequality in \eqref{PropB3_auxiliary_inequalities}. By changing $f$ with $-f$ and $\mu$ with $-\mu$ in \eqref{linear_problem_with_L1_term_and_measure_rhs} and then applying the first inequality in \eqref{PropB3_auxiliary_inequalities} we easily obtain the second one. Finally, summing up both of these inequalities gives \eqref{PropB3_main_result}.
\end{proof}

In the next theorem we show that if we additionally impose the condition $b(x,\phi)\in L^1(\Omega)$ in Definition~\ref{definition_weak_formulation_full_GPBE_dimensionless}, then one can show that there is only one such $\phi$ that satisfies~\eqref{weak_formulation_General_PBE_W1p_spaces}. 
\begin{theorem}[Uniqueness of the weak solution of the GPBE]\label{Theorem_uniqueness_for_GPBE}
Under Assumption~\ref{Assumption_Domain_permittivity}, there can only be one solution $\phi$ to problem \eqref{weak_formulation_General_PBE_W1p_spaces} (where $b$ is defined in \eqref{definition_b_GPBE}) such that $b(x,\phi)\in L^1(\Omega)$.
\end{theorem}
\begin{proof}
Let $\phi_1,\phi_2\in \Mb_{g_\Omega}$ be two solutions of \eqref{weak_formulation_General_PBE_W1p_spaces} such that $b(x,\phi_1),\,b(x,\phi_2)\in L^1(\Omega)$.  Subtracting the corresponding weak formulations for $\phi_1$ and $\phi_2$ we get $\phi_1-\phi_2\in \Mb_0$ and 
\begin{equation} \label{difference_weak_formulations_General_PBE_W1p_spaces}
\int_{\Omega}{\epsilon\nabla (\phi_1-\phi_2)\cdot\nabla v \dd x}+\int_{\Omega}{\left(b(x,\phi_1)-b(x,\phi_2)\right)v \dd x}=0 \, \text{ for all } \,v\in \Nb.
\end{equation} 
By applying Proposition~\ref{Proposition_B3} with $f=b(x,\phi_1)-b(x,\phi_2)\in L^1(\Omega)$  and $\mu=0$ we obtain
\begin{equation} \label{application_of_PropB3}
\int_{\Omega}{\left(b(x,\phi_1)-b(x,\phi_2)\right)\,{\rm sgn}(\phi_1-\phi_2)\dd x}\leq 0.
\end{equation} 
Recalling the definition of $b(x,\cdot)$ in \eqref{definition_b_GPBE} we have $b(x,t)=0$ for all $x\in \Omega \setminus \Omega_{ions}$, and $b(x,\cdot)$ strictly increasing whenever $x\in \Omega_{ions}$, with $\Omega = \Omega_m \cup \Gamma \cup \Omega_{IEL} \cup \Omega_{ions}$, where the union is pairwise disjoint and $\Gamma = \partial \Omega_m$. Taking this into account,  \eqref{application_of_PropB3} implies
\begin{equation} \label{intermediate_result_01_Theorem_uniqueness_GPBE}
\left(b(x,\phi_1)-b(x,\phi_2)\right)\,{\rm sgn}(\phi_1-\phi_2)=0 \text{ a.e. } x\in\Omega,
\end{equation} 
which in turn gives $\phi_1(x)=\phi_2(x)$ for a.e. $x\in\Omega_{ions}$, but provides no information on $\Omega \setminus \Omega_{ions}$. 

To see that $\phi_1=\phi_2$ a.e. in the whole domain $\Omega$, note that the second integral on the left hand side of \eqref{difference_weak_formulations_General_PBE_W1p_spaces} is zero. This allows us to apply Theorem~\ref{Theorem_uniqueness_of_linear_elliptic_problems_with_diffusion_and_measure_rhs} to the resulting linear problem on the complete $\Omega$ and conclude that it has a unique solution. Moreover, it clearly  admits the trivial solution as well, so $\phi_1 - \phi_2=0$.
\end{proof}

We notice that in the proof of this theorem, only two features of the nonlinearity have been used: for $x \in \Omega \setminus \overline{\Omega_{ions}}$ we have that $b(x, \cdot)\equiv 0$, and for $x \in \Omega_{ions}$ we have that $b(x, s)$ is strictly monotone in $s$. This kind of behaviour allows us to infer uniqueness in the semilinear problem from the linear one, also for more general coefficient matrices $\bm A\in \left[L^\infty(\Omega)\right]^{d\times d}$. In particular, we get the following:

\begin{corollary}Let $\Omega\subset \mathbb R^d$ with $d\in \{2, 3\}$ be a bounded domain with Lipschitz boundary, and $\Omega_0\subset\Omega$ a subdomain with $C^1$ boundary and $\dist(\Omega_0, \, \partial \Omega)>0$. Let $\bm A$ be a $d\times d$ symmetric matrix valued function on $\Omega$ satisfying the uniform ellipticity condition \eqref{uniform_ellipticity_of_A} and which is uniformly continuous on both $\Omega_0$ and $\Omega\setminus \overline{\Omega_0}$. Assume further that $\Omega_1$ is a measurable subset of $\Omega$ and that $b:\Omega\times\R\to\R\cup\{+\infty\}$ is a Carath\'eodory function such that $t \mapsto b(x,t)$ is strictly monotone for almost every $x\in\Omega_1$ and vanishes identically for almost every $x\in \Omega\setminus\Omega_1$. Then for $g \in W^{1-1/p, p}(\partial \Omega)$ with $p=\frac{d}{d-1}$ and $\mu \in \Mp(\Omega)$ the problem
\begin{equation}\label{weak_formulation_general_semilinear_measure_RHS}
\begin{aligned}
&\text{Find }z \in \bigcap_{p<\frac{d}{d-1}}{W_{g}^{1,p}(\Omega)}\,\text{ with }\, {b(x,z)\in L^1(\Omega)}\, \text{ such that }\\
&\int_{\Omega}{\bm A \nabla z\cdot\nabla v \dd x}+\int_{\Omega}{b(x,z)v \dd x}=\int_{\Omega}{v \dd\mu} \, \text{ for all } \,v\in \bigcup_{q>d}{W_0^{1,q}(\Omega)}
\end{aligned}
\end{equation} 
has at most one solution.
\end{corollary}

To conclude, we reiterate that having such a uniqueness result for the weak formulation \eqref{weak_formulation_General_PBE_W1p_spaces} ensures that both the 2-term and 3-term splittings lead to the same full potential $\phi$, as would any other decomposition compatible with this natural notion of weak solution for PDE with measure data.

\section*{Acknowledgments}
The second author is grateful for the financial support received from the Austrian Science Fund (FWF) through the Doctorate College program ``Nano-Analytics of Cellular Systems (NanoCell)'' with grant number W1250 and project P 33154-B, and from Johannes Kepler University in conjunction with the State of Upper Austria through projects LIT-2017-4-SEE-004 and LIT-2019-8-SEE-120. The first author is partially supported by the State of Upper Austria. We would like to thank Hannes Meinlschmidt for helpful comments on a preliminary version of this manuscript.
\appendix
\section{Auxiliary results}\label{Auxiliary_results}
\renewcommand*{\thesection}{\Alph{section}}
A key element in the proof of classical boundedness results for solutions for elliptic problems such as those of \cite{Stampacchia_1965} (see also Theorem B.2 in \cite{Kinderlehrer_Stampacchia}) is the following `extinction' lemma, which we also use for our $L^\infty$ estimate in Theorem \ref{Thm_Boundedness_of_solution_to_general_semilinear_elliptic_problem}.
\begin{lemma}[Lemma B.1 in \cite{Kinderlehrer_Stampacchia}]\label{Lemma_B1_Kinderlehrer_Stampacchia}
Let $\Theta(t)$ denote a function which is nonnegative and nonincreasing for $k_0\leq t<\infty$. Further, assume that
\begin{equation} \label{Lemma_B1_main_inequality}
\Theta(t)\leq C\frac{\Theta(k)^\beta}{(t-k)^\alpha},\,\mathrm{for}\ \mathrm{all}\ t>k>k_0,
\end{equation} 
where $C$ and $\alpha$ are positive constants and $\beta>1$.
If $t_e \in  \mathbb R$ is defined by $t_e^\alpha:=C \Theta(k_0)^{\beta-1}2^{\frac{\alpha\beta}{\beta-1}}$, then $\Theta(k_0+t_e)=0$.
\end{lemma}

The following Theorem of optimal regularity of linear elliptic interface problems is central to many of our arguments, since it applies to the realistic PBE situation with $d=3$:
\begin{theorem}[Optimal regularity of elliptic interface problems, Theorem 1.1 in~\cite{Optimal_regularity_for_elliptic_transmission_problems}]\label{Thm_Optimal_Regularity_for_Elliptic_Interface_Problems} 
Assume that $\Omega\subset \mathbb R^d$ is a bounded Lipschitz domain and let $\Omega_0\subset\Omega$ be another domain with a $C^1$ boundary, which does not touch the boundary of $\Omega$. Let $\bm\mu$ be a function on $\Omega$ with values in the set of real, symmetric $d\times d$ matrices which is uniformly continuous on both $\Omega_0$ and $\Omega\setminus \overline{\Omega_0}$. Additionally, $\bm\mu$ is supposed to satisfy the usual ellipticity condition 
\begin{equation*}
\underline{c}\abs{\xi}^2\leq \bm\mu(x)\xi\cdot \xi\ \text{ for some }\underline{c}>0,\text{ all } \xi=(\xi_1,\ldots,\xi_d)\in\mathbb R^d\text{ and }a.e.\, x\in \Omega.
\end{equation*}
Then there is a $p>3$ such that for every $\lambda\ge 0$,
\begin{equation*}
-\nabla\cdot \bm\mu\nabla + \lambda: W_0^{1,q}(\Omega)\to W^{-1,q}(\Omega)
\end{equation*}
is a topological isomorphism for all $q\in (p',p)$ with $p'$ being the H{\"o}lder conjugate of $p$. If $\Omega$ itself is also a $C^1$ domain, then $p$ may be taken as $+\infty$.
\end{theorem}

In the proofs of uniqueness and a priori $L^\infty$ estimates for the regular component $u$ in the 2- and 3-term splittings of the full potential $\phi$, we use the following result due to Br{\'e}zis and Browder, which is very useful for the analysis of semilinear elliptic equations which do not have any growth conditions on the nonlinearity, such as the PBE.
 
\begin{theorem}[A property of Sobolev spaces, H. Br{\'e}zis and F. Browder, 1978, \cite{Brezis_Browder_1978_One_Property_of_Sobolev_Spaces}]\label{Thm_Brezis_Browder_for_the_extension_of_bdd_linear_functional_1978}
Let $\Omega$ be an open set in $\mathbb R^d$, $T$ a distribution such that $T\in H^{-1}(\Omega)\cap L_{loc}^1(\Omega)$, and $v\in H_0^1(\Omega)$. If there exists a function $f\in L^1(\Omega)$ such that $T(x)v(x)\ge f(x)$, a.e in $\Omega$, then $Tv\in L^1(\Omega)$ and the duality product $\langle T,v\rangle$ in $H^{-1}(\Omega)\times H_0^1(\Omega)$ coincides with $\int_{\Omega}{Tv\dd x}$.
\end{theorem}
\begin{remark}\label{remark_on_the_theorem_of_Brezis_and_Browder_for_the_extension_of_bdd_linear_functional_1978}
In other words, we have the following situation (see the proof of the Theorem in \cite{Brezis_Browder_1978_One_Property_of_Sobolev_Spaces}): a locally summable function $b\in L_{loc}^1(\Omega)$ defines a bounded linear functional $T_b$ over the dense subspace $D(\Omega)\equiv C_c^\infty(\Omega)$ of $H_0^1(\Omega)$ through the integral formula $\langle T_b, \varphi\rangle =\int_{\Omega}{b\varphi \dd x}$. It is clear that the functional $T_b$ is uniquely extendable by continuity to a bounded linear functional $\overline T_b$ over the whole space $H_0^1(\Omega)$. Now the question is whether this extension is still representable by the same integral formula for any $v\in H_0^1(\Omega)$ (if the integral makes sense at all). If the function $v\in H_0^1(\Omega)$ is fixed, then Theorem \ref{Thm_Brezis_Browder_for_the_extension_of_bdd_linear_functional_1978} gives a sufficient condition for $bv$ to be summable and for the extension $\overline T_b$ evaluated at $v$ to be representable with the same integral formula as above, i.e $\langle \overline T_b, v \rangle =\int_{\Omega}{bv\dd x}$.
\end{remark}

\bibliographystyle{plain}
\bibliography{references_new_v05}
\end{document}